\documentclass[dvipsnames]{amsart}[12]

\usepackage[a4paper]{geometry}
\usepackage[backend=biber, datamodel=mrnumber, maxbibnames=99, sortcites]{biblatex}
\usepackage{amsmath}
\usepackage{amsthm}
\usepackage{mathrsfs}
\usepackage{amsfonts}
\usepackage{amssymb}
\usepackage{amscd}
\usepackage{array}
\usepackage{amssymb}
\usepackage[all, cmtip]{xy}
\usepackage{tikz}
\usetikzlibrary{arrows}
\usetikzlibrary{cd}
\usepackage{ulem}
\usepackage{enumitem}
\setlist{
  listparindent=\parindent,
  parsep=0pt,
}

\usepackage{hyperref}

\usepackage{xcolor}
\usepackage{quiver}

\newtheorem{theorem}{Theorem}[subsection]
\newtheorem{lemma}[theorem]{Lemma}

\newtheorem{definition}[theorem]{Definition}
\newtheorem{corollary}[theorem]{Corollary}
\newtheorem{proposition}[theorem]{Proposition}
\newtheorem{remark}[theorem]{Remark}
\newtheorem{construction}{Construction}
\newtheorem{warning}[theorem]{Warning}

\newcommand{\CC}{\mathbb{C}}

\newcommand{\RR}{\mathbb{R}}
\newcommand{\PP}{\mathbb{P}}
\newcommand{\ZZ}{\mathbb{Z}}

\newcommand{\EE}{\mathbb{E}}
\newcommand{\LL}{\mathbb{L}}
\renewcommand{\AA}{\mathbb{A}}

\newcommand{\OO}{\mathscr{O}}

\newcommand{\Hom}{Hom}

\newcommand{\FF}{\mathbb{F}}

\newcommand{\fX}{\mathfrak X}
\newcommand{\sE}{\mathscr{E}}

\newcommand{\relSpec}{\underline{\mathrm{Spec}}}
\newcommand{\VV}{\mathbb{V}}

\renewcommand{\AA}{\mathbb{A}}

\newcommand{\fM}{\mathfrak M}

\newcommand{\vir}{\mathrm{vir}}

\newcommand{\Gm}{\mathbb G_m}

\newcommand{\rank}{\mathrm{rank}}
\newtheorem{example}{Example}

\DeclareMathOperator{\Spec}{Spec}
\DeclareMathOperator{\Ext}{Ext}
\renewcommand{\Hom}{\mathrm{Hom}}

\DeclareMathOperator{\Sym}{Sym}

\newcommand{\mf}{\mathfrak}

\newif\ifmoditem
\newcommand{\setupmodenumerate}{%
  \global\moditemfalse
  \let\origmakelabel\makelabel
  \def\moditem##1{\global\moditemtrue\def\mesymbol{##1}\item}%
  \def\makelabel##1{%
    \origmakelabel{##1\ifmoditem\rlap{\mesymbol}\fi\enspace}%
    \global\moditemfalse}%
}

\newcommand{\mls}{\mathscr}
\newcommand{\mc}{\mathcal}

\begin{document}

\title{Virtual classes: An introduction with exercises}

\author{Xuanchun Lu}
\author{Rachel Webb}


\begin{abstract} We introduce virtual classes in the sense of Behrend-Fantechi, explaining how these directly generalize Fulton's localized top Chern class. We derive Siebert's formula for a virtual class on any Deligne-Mumford stack with a perfect obstruction theory having a global resolution.
\end{abstract}

\maketitle

\section*{Introduction}
If $X$ is a scheme or Deligne-Mumford stack over a field $k$, a \textit{virtual class} on $X$ is an element of the Chow group $A_*(X)$ that somehow captures information about the singularities of $X$. There are many approaches to making this idea precise, most
using local equations for embedding $X$ in a smooth space. Capping with the virtual class defines a notion of ``integration'' for $A^*(X)$, making it useful for applications to enumerative geometry.
In these expository notes we present the virtual class construction of Behrend--Fantechi, a construction using some derived category data called a \textit{perfect obstruction theory}, and the application of Behrend--Fantechi's construction to Gromov--Witten theory.

There are other important constructions of virtual classes that we will not discuss here. In symplectic geometry, Li-Tian \cite{LT} were the first to give a construction and the introduction of \cite{HS} discusses more recent developments. \textit{Kuranishi charts} are the symplectic counterpart to Behrend--Fantechi's perfect obstruction theory. In derived algebraic geometry, Sch\"urg-To\"en-Vezzosi \cite{STV} explained how the cotangent complex of a quasi-smooth scheme or stack encodes the same data. Both the symplectic and the derived approaches were used in the cited locations to construct virtual classes for Gromov--Witten theory.

There is at least  one other exposition of Behrend--Fantechi's work available, namely Battistella-Carocci-Manolache's excellent article \cite{BCM}. Their paper focuses on explicit computations of virtual classes, including explicit computations in Gromov--Witten theory. By contrast, our presentation centers on the conceptual link between Behrend--Fantechi's construction and Fulton's localized top Chern class. We also include a more thorough treatment of Siebert's formula \cite[Thm 4.6]{siebert} for the virtual class, following our paper \cite{LW2}.

We note that an abridged version of this article (without exercises or solution notes) is submitted to the Proceedings of the Bootcamp of the 2025 Summer Research Institute in Algebraic Geometry.

\subsection*{Outline of the notes}
The notes are organized into ten sections with exercises accompanying the last nine. All proofs are either cited from the literature or left as exercises.

Section 0 is a non-rigorous overview of the paper, explaining at a high level the connection between Behrend-Fantechi virtual classes and Fulton's localized top Chern class. Sections 1-3 present two constructions of an obstruction theory that rely only on ideas in Fulton's book \cite{fulton}.
Sections 4-6 introduce Behrend-Fantechi obstruction theories and the associated virtual classes, Sections 7-8 explain their basic application to Gromov-Witten theory, and Section 9 discusses Siebert's formula.

\subsection*{Prerequisites}

We assume familiarity with chapters 1-6 of Fulton's work \cite{fulton}. The reader interested in virtual classes on Deligne-Mumford stacks should also be familiar with Vistoli's paper \cite{vistoli}. These are the only prerequisites for Sections 0-4. Beginning in Section 5 we also assume minimal knowledge of the derived category of sheaves on $X$ with quasi-coherent cohomology.
Finally, sections 5-9 will be easier for readers with some familiarity with algebraic stacks, especially global quotient stacks. We suggest the lecture \cite{webbtalk} for intuition and the paper \cite{fantechi} for a rigorous introduction to stacks.

\subsection*{Notation and conventions}

\begin{itemize}
\item We work over a ground field $k$. All schemes and stacks are locally finite type over $k$, in particular locally Noetherian.
\item All varieties are irreducible.
\item All Chow groups are taken to have rational coefficients.
\item If $\sE$ is a coherent sheaf on a scheme or algebraic stack $X$, then $Sym(\sE)$ is its symmetric algebra; i.e., $Sym(\sE) = \oplus_{n \geq 0} Sym^n(\sE).$

\item A vector bundle $E$ on an algebraic stack $X$ is an $X$-stack equal to the relative spectrum of the symmetric algebra of a finite rank locally free sheaf on $X$.

\item If $I$ is a quasicoherent ideal sheaf on an algebraic stack $X$ or a section of a vector bundle on $X$,
    then we use $\VV(I)$ for the associated closed substack of $X$.

\item If $\EE = \ldots \to E^{i-1} \xrightarrow{d^{i-1}} E^{i} \xrightarrow{d^{i} }\ldots $ is a complex of $\mls O_X$-modules then $h^i(\EE) = \mathrm{ker}(d^i)/\mathrm{image}(d^{i-1})$.

\end{itemize}
\subsection*{Acknowledgements}
Lu was supported by the ERC Advance Grant MSAG.
Webb was supported by the NSF grant DMS 2501528.
These notes were first drafted for the Bootcamp of the 2025 Summer Research Institute in Algebraic Geometry, held at Colorado State University. We thank the participants of our working group for their questions and comments on these notes. An anonymous referee also provided helpful corrections.

\section{Overview: what is a virtual class?}

This section is an overview of the entire set of lecture notes. It prioritizes intuition and motivation over precise definitions or results.

\subsection{A problem posed by physics}\label{sec:problem}

In theoretical physics, string theory \textit{compactifies} spacetime by introducing a 6-real dimensional Calabi-Yau manifold $X$ at every point. In other words, string theory replaces our usual $\RR^4$ model for spacetime with $\RR^4 \times X$. (Classical spacetime has four dimensions to represent our three physical dimensions plus one time dimension.)  The geometry of $X$ determines physical \textit{correlation functions} that compute partical interactions in the spacetime model $\RR^4 \times X$. These correlation functions depend on additional discrete data, including a class in $H_2(X)$.\footnote{
The other discrete data are an integral genus $g \geq 0$ and number of markings $n \geq 0$, as well as a choice of $n$ elements of $H^*(X; \CC)$. We will fix $g=n=0$ in this overview so that Gromov-Witten invariants of $X$ only depend on the choice of class $\beta$.
}

From a mathematical perspective, physics provides a machine that produces numerical invariants of $X$. One mathematical model for this machine is \textit{Gromov-Witten theory.}
Mathematically, how are Gromov-Witten invariants of $X$ defined?

The first example of a Calabi-Yau having 6 real dimensions, or 3 complex dimensions, is the quintic $Q = \mathbb{V}(x_0^5 + x_1^5 + x_2^5 + x_3^5 + x_4^5)$, a hypersurface in $\PP^4_{\CC}$. To define the invariant of $Q$ of class $dH$, where $H$ is the class of a line in $\PP^4_{\CC}$ and $d$ is positive, we introduce the moduli space (stack)
\[
\overline{\mathcal{M}}(Q, dH) \quad \text{ parametrizing maps }C \to Q\text{ with finitely many automorphisms.}
\]
Here, $C$ is a complex projective curve of arithmetic genus zero, allowed to have mild (nodal) singularities so that the moduli space $\overline{\mathcal{M}}(Q, dH)$ is proper.
If $\overline{\mathcal{M}}(Q, dH)$ is just a collection of finitely many reduced points--- which is the expected behavior by a dimension count---then we define the associated correlator of class $dH$ to be that number of points.
This is just the number of curves in $Q$ of homology class $d$.
But, as the following example shows, $\overline{\mathcal{M}}(Q, dH)$ is generally not zero dimensional!

\begin{example}
The moduli space $\overline{\mathcal{M}}(Q, H)$
is not a collection of finitely many points.
Indeed,
for every point $[a : b : c] \in \VV(a^{5} + b^{5} + c^{5}) \subset \PP^{2}$,
there is a line on $Q$ given by
$[\lambda : \mu] \mapsto [\lambda : -\lambda : a\mu : b\mu : c\mu]$.
These lines are clearly mutually distinct.

It is classical that a \emph{generic} quintic threefold
contains exactly $2875$ lines.
In their paper \cite{albano-katz},
Albano and Katz give an explicit description of which lines
in $\overline{\mc{M}}(Q, dH)$
deform to lines on a generic quintic.

\end{example}

This leads to a problem: \textit{define} a class \[[\overline{\mathcal{M}}(Q, dH)]^{vir} \in A_0(\overline{\mathcal{M}}(Q, dH),\]
consisting of a number of points that are ``supposed'' to comprise this moduli space (this number should also be the correct correlator value from a physical perspective). That class $[\overline{\mathcal{M}}(Q, dH)]^{vir}$ is called the \textit{virtual class} of the moduli space.

\subsection{An answer from intersection theory}

Luckily, a definition of $[\overline{\mathcal{M}}(Q, dH)]^{vir}$ can be found in Fulton's seminal work on intersection theory (though not under this name): it is the localized top Chern class in \cite[Sec~14.1]{fulton}. Let us recall this theory.

If $E$ is a vector bundle on a scheme $Y$
its \textit{top Chern class} is the self intersection of the zero section $0: Y \to E$. In a formula,
\[
c_{top}(E) \cap [Y] = 0^!0_*[Y].
\]
Here $0_*$ is proper pushforward and $0^!$ is Gysin pullback. The formula reveals that our first sentence abused grammar a bit: really, $c_{top}(E)$ is a function $A_*(Y) \to A_*(Y)$, and its value on the fundamental class $[Y]$ is the self intersection of the zero section.

Fulton's theory of localized top Chern classes says that if $E$ is a vector bundle on a pure dimensional scheme $Y$, then for any section $s$, the class $c_{top}(E) \cap [Y]$ \textit{localizes} to $\mathbb{V}(s)$. By definition this means there is a class, which we will suggestively call $[\mathbb{V}(s)]^{vir}_E$, living in $A_{\dim Y - \mathrm{rank}(E)}(\mathbb{V}(s))$, such that
\[
i_*[\mathbb{V}(s)]^{vir}_E = c_{top}(E) \cap [Y].
\]
While we can find such $[\mathbb{V}(s)]^{vir}_E$ for any $s$, if we happen to know that $s$ is regular, then we can use the fundamental class:
\[
[\mathbb{V}(s)]^{vir}_E = [\mathbb{V}(s)] \quad \quad \quad \quad \text{if $s$ is regular}.
\]

\begin{definition}
Let $X$ be a scheme. If $E$ is a vector bundle on a scheme $Y$ with a section $s$ and we are given an isomorphism $X \simeq \mathbb{V}(s)$, the associated \emph{virtual class} on $X$ is $[\mathbb{V}(s)]^{vir}_E.$
\end{definition}

There is a pleasant reinterpretation of this definition. The preceeding discussion shows that
\[
0^!0_*[Y] = i_*[\mathbb{V}(s)]
\]
when $s$ is a regular section; in other words, the self-intersection of the zero section of $E$ is exactly the class of the zero locus of a regular section.\footnote{A regular section does not always exist---for instance $\mls O_{\PP^1}(-1)$ has no regular sections---but this is a distraction from our purpose.}
This is a rigorous version of the intuition that the intersection of two varieties can be computed by ``moving'' one of them until the intersection is transverse: in our situation, we expect to compute the self-intersection of the zero section by replacing one of these sections with a ``sufficiently generic'' (in this case regular) section. The literal intersection of our regular section with the zero section is by definition equal to $\mathbb{V}(s)$. Now, the preceeding discussion is asserting that even when $s$ is \textit{not} regular, there is a class in $\mathbb{V}(s)$ that is equal (after pushing forward to $Y$) to the class of the zero locus of a regular section. Summary:
\[
c_{top}(E) \cap [Y]\text{ is the class of a regular section of $E$, or the \textit{virtual} class of \textit{any} section.}
\]

We close with some constructions of virtual classes on specific spaces $X$, including $X = \overline{\mathcal{M}}(Q, dH).$
\begin{example}
Let $X = \PP^1$. Embed $X$ in $Y = \PP^1$ via the identity map, so $X$ is the zero locus of the zero section $s$ of the zero vector bundle $E$. Then
\[i_*[X]^{vir}_E = c_{top}(0) \cap [\PP^1] = [\PP^1]\]
since the top Chern class of the zero bundle is the identity map. Moreover $i_*$ is the identity, so $[X]^{vir}_E = [X]$ for this choice of $(s, E)$.
\end{example}

\begin{example}
Let $X = \PP^1$. Embed $X$ in $Y = \PP^3$ as the twisted cubic, so $X$ is cut out by a section of $\mls O_{\PP^3}(2)^{\oplus 3}$. Then
\[
i_*[X]^{vir}_E = c_{top}(\mls O_{\PP^3}(2)^{\oplus 3}) \cap [\PP^3] = 8[pt]
\]
where $[pt]$ is the class of a point in $\PP^3$. Since $i_*: A_0(\PP^1) \to A_0(\PP^3)$ is injective we have $[X]^{vir}_E = 8[pt]$ for this choice of $(s, E)$. Note that 8 points is the expected intersection of three quadrics.
\end{example}

\begin{example}
Let $X = \overline{\mathcal{M}}(Q, dH)$ and let $Y = \overline{\mathcal{M}}(\PP^4, dH)$. Let $i: Q \to \PP^4$ denote the inclusion. The functor from $X$ to $Y$ sending a map $f: C \to Q$ to $i\circ f: C \to \PP^4$ is a closed embedding (one shows it is a proper monomorphism).
\[
\text{\underline{Claim 1:} $Y$ is smooth, hence pure-dimensional, of dimension $5d+1$.}\]

There is moreover a natural vector bundle $E$ on $Y$ with a section whose vanishing locus is $X$. This vector bundle arises as follows. On $Y$, we have a \textit{universal family}
\[
\begin{tikzcd}
C \arrow[d, "\pi"] \arrow[r, "f"] & \PP^4\\
Y = \overline{\mathcal{M}}(\PP^4, d)
\end{tikzcd}
\]
where $C \to Y$ is a family of nodal genus-zero curves and the fibers of $f$ are degree-$d$ maps.
\[\text{\underline{Claim 2:}
$\pi_*f^*\mls O_{\PP^4}(5)$ is a vector bundle of rank $5d+1$ on $Y$ with a section $s$ and $\mathbb{V}(s) = X$.}\]

This data gives us a virtual class in $A_{\dim(Y) - \mathrm{rank}(E)}(X) = A_0(X)$; i.e., we have
\[
[\overline{\mathcal{M}}(Q, dH)]^{vir}_E \in A_0(\overline{\mathcal{M}}(Q, dH)),\]
solving our problem from Section \ref{sec:problem}. This class pushes forward to the top Chern class of $E$ on $\overline{\mathcal{M}}(\PP^4, dH)$.
\end{example}

\subsection{Relationship to Behrend-Fantechi obstruction theories}

We were able to construct a virtual class on $\overline{\mathcal{M}}(Q, dH)$ because we were able to globally embed it in a pure-dimensional scheme as the zero locus of a section of a vector bundle. For a general scheme $X$, we cannot necessarily globally embed it in this way. In fact, for more general moduli spaces arising from Gromov-Witten theory such an embedding is a nontrivial result and not clearly canonical.\footnote{
For moduli spaces $\overline{\mathcal{M}}_{g, n}(\PP^r, \beta)$, global embeddings are discussed in \cite[Appendix A]{GP}, which heavily relies on \cite{FP}. The argument is extended to more general moduli of stable maps in \cite{AGOT}.
}

However, every scheme $X$ has an open cover such that every set in the cover has  such an embedding (in fact, any affine open cover works).
Then we have a problem of articulating ``transition functions between gluing data'' and using these to ``glue'' the resulting virtual fundamental classes.
Two key ingredients of the paper \cite{BF} of Behrend-Fantechi are the following.
\begin{itemize}
\item They define a \textit{perfect obstruction theory} on $X$ to be a morphism $\mathbb{E} \to \mathbb{L_X}$ in the derived category of $X$ satisfying certain properties, where $\mathbb{L_X}$ is the cotangent complex. In a sense made precise in Propositions \ref{prop:local1} and \ref{prop:local2}, this object glues local data of an embedding of $X$ as the zero locus of a section of a vector bundle.
\item Given a perfect obstruction theory $\mathbb{E} \to \mathbb{L_X}$, they define an associated virtual class $[X]^{vir}_{\EE}$. Later, we will see that this class globalizes the localized top Chern classes arising from the local data described above.
\end{itemize}

The goal of these notes is to explain how the Behrend-Fantechi construction of a virtual fundamental class globalizes the construction as Fulton's localized top Chern class, and how these ideas can be applied to Gromov-Witten theory. Indeed, a key feature of Gromov-Witten  theory is certain compatibility relations between the virtual classes of different  spaces involved. Behrend-Fantechi's construction of a virtual class on these spaces is canonical enough to make such comparisons possible (as opposed to a construction via a global embedding, where the compatibilities are not clear).

We will finish the notes by presenting a formula for the virtual class, due to Siebert \cite{siebert}, that is slightly less general than Behrend-Fantechi's construction but does apply to all moduli spaces arising in Gromov-Witten theory. This formula drastically simplifies the argument that the Gromov-Witten virtual classes are compatible with each other.

\section{Normal cones and normal spaces}
Just like one needs at least some commutative ring theory before one can discuss schemes, we need to develop something of the theory of \textit{cones} before we can discuss virtual classes.
Let $X$ be a scheme or algebraic stack over $k$. Recall from \cite[Tag 01LL]{stacks-project} the definition of the relative spectrum $\relSpec_X(-)$.
\begin{definition}\label{def:cone}
A \emph{cone} on $X$ is a scheme of the form $\relSpec_X(S^\bullet)$ for some quasi-coherent sheaf $S^\bullet$ of $\ZZ_{\geq 0}$-graded $\mls O_X$-algebras such that $S^0 = \mls O_X$ and the canonical $\mls O_X$-algebra homomorphism $Sym(S^1) \to S^\bullet$ is surjective. A \textit{morphism of cones} is a morphism of schemes $\relSpec_X(S^\bullet) \to \relSpec_X(T^\bullet)$ induced by a graded $\mls O_X$-algebra homomorphism $T^\bullet  \to  S^\bullet$.
\end{definition}
One should read the condition ``$Sym(S^1) \to S^\bullet$ is surjective'' as saying that $S^1$ locally generates $S^\bullet.$ Geometrically we are requiring the canonical map $\relSpec_X(S^\bullet) \to \relSpec_X(Sym(S^1))$ to be a closed embedding.

The grading on $S^\bullet$ equips $\relSpec_X(S^\bullet)$ with a $\Gm$-action (namely an $X$-morphism $\Gm \times \relSpec_X(S^\bullet) \to \relSpec_X(S^\bullet)$ satisfying various properties), and this is the sense in which $\relSpec_X(S^\bullet)$ is a cone.

\subsection{First example of cones}

Our first examples of cones on $X$ are those arising from coherent sheaves on $X$.

\begin{definition}\label{def:abelian}
If $\mls E$ is a coherent sheaf on $X$, its associated cone is $\relSpec_X(Sym(\mls E))$.
\end{definition}

The cones in Definition \ref{def:abelian}
have additional structure: they have a + operation making them into abelian group schemes over $X$. Specifically, the + operation on $\relSpec_X(Sym(\mls E))$ is the $X$-morphism induced from the diagonal morphism of $\mls O_X$-algebras $Sym(\mls E) \to Sym(\mls E) \otimes_{\mls O_X} Sym(\mls E)$. Moreover, this + operation is compatible with the $\Gm$ action, meaning that there is a commuting diagram equivalent to the equality $t(a+b) = ta + tb$ for $t$ a point of $\Gm$ and $a, b$ points of $\relSpec_X(S^\bullet)$.

Siebert \cite[Sec 1.1]{siebert} calls cones arising from coherent sheaves \emph{linear spaces}; Behrend-Fantechi \cite[Sec 1]{BF} call them \emph{abelian cones}. We will follow Behrend-Fantechi. The functor from coherent sheaves to abelian cones is contravariant and an equivalence of categories. Note that
if $\mls E$ is locally free, then its associated cone is a vector bundle on $X$. Hence we have containments of categories:
\[
\left\{\text{vector bundles} \right\} \subsetneq \left\{\text{abelian cones} \right\} \subsetneq \left\{\text{cones} \right\}
\]

As stated in the diagram, a general cone $\relSpec_X(S^\bullet)$ may not be abelian. However, the surjective homomorphism $Sym(S^1) \to S^\bullet$ gives this cone a canonical closed embedding into the abelian cone $\relSpec_X(Sym(S^1))$ (cf. Exercise \ref{cone1}).

\begin{warning}
There is a choice of convention in Definition \ref{def:abelian}: one could associate either the cone $\relSpec_X(Sym(\mls E))$ or $\relSpec_X(Sym(\mls E^\vee))$ to a coherent sheaf $\mls E$ on $X$. If $\mls E$ is locally free, then the sheaf of sections of $\relSpec_X(Sym(\mls E^\vee))$ is $\mls E$. For this reason, in some contexts it is more natural to associate the cone $\relSpec_X(Sym(\mls E^\vee))$ to a coherent sheaf $\mls E$. For us, $\relSpec_X(Sym(\mls E))$ is more natural since with this convention, any cone embedds into the abelian cone associated to the degree-1 part of its algebra.
\end{warning}

\subsection{More examples of cones}
To define virtual classes, the following examples of cones are essential.

\begin{definition}
Let $X \to Y$ be a closed embedding of schemes (or algebraic stacks) and let $\mls I$ be the ideal sheaf. The \emph{normal cone} $C_{X/Y}$ is a cone and the \emph{normal space} $N_{X/Y}$ is an abelian cone, given by the formulas
\[
C_{X/Y} = \relSpec_X(\bigoplus_{n\geq 0} \mls I^n/\mls I^{n+1}) \quad \quad \quad N_{X/Y} = \relSpec_X(Sym(\mls I/\mls I^2)).
\]
\end{definition}
Observe that $N_{X/Y}$ is the abelian cone corresponding to the conormal sheaf $\mls N_{X/Y} = \mls I/\mls I^2$. Behrend-Fantechi call $N_{X/Y}$ the \emph{normal sheaf} because it is the abelian cone that is the total space of the literal normal sheaf $(\mls I/\mls I^2)^\vee$. Since $N_{X/Y}$ is not actually a sheaf, we will use Siebert’s term for this object, namely the \emph{normal space} of $X$ in $Y$. Observe that there is a closed embedding
\begin{equation}\label{eq:cone1}
C_{X/Y} \hookrightarrow N_{X/Y}
\end{equation}
Sometimes the containment $C_{X/Y} \subseteq N_{X/Y}$ is equality:

\begin{lemma}\label{lem:cone3}
If $X \hookrightarrow Y$ is a quasi-regular embedding, then $C_{X/Y} = N_{X/Y}$ and these cones are vector bundles.
\end{lemma}

\begin{remark}
A quasi-regular embedding of schemes is defined in \cite[Tag 063J]{stacks-project} (but see also \cite[Tags 063D, 061P]{stacks-project}).
Due to our running hypothesis that all schemes are locally Noetherian, an embedding $X \hookrightarrow Y$ is quasi-regular if and only if it is regular \cite[Tag 063L]{stacks-project}.

Since regularity is not even \'etale local on the target \cite[Tag 06BL]{stacks-project}, it does not make sense to talk about regularity for a morphism of algebraic stacks. Fortunately quasi-regularity is fpqc-local on $Y$ \cite[Tag 068N]{stacks-project}, so we can define a quasi-regular embedding of algebraic stacks as in \cite[\S5.3]{RydhLuna}.
\end{remark}

Normal cones have functoriality properties similar to those of conormal sheaves. We will need the following lemma to define virtual classes later.
\begin{lemma}\label{lem:cone2}
Suppose we have a cartesian diagram of schemes or algebraic stacks
\[
\begin{tikzcd}
X' \arrow[r] \arrow[d] & X \arrow[d] \\
Y' \arrow[r] & Y
\end{tikzcd}
\]
where vertical maps are closed embeddings. Then there is a closed embedding $C_{X'/Y'} \hookrightarrow C_{X/Y} \times_X X' $.
\end{lemma}

\begin{remark}\label{rmk:local-picture}
 From Exercises \ref{cone4} and \ref{cone5}, you can see that if $Y$ is a $k$-scheme and $X \hookrightarrow Y$ is a $k$-point, then $C_{X/Y}$ is a ``local picture’’ of $Y$ near $X$ while $N_{X/Y}$ is a smallest vector space containing $C_{X/Y}$.
\end{remark}
\begin{remark}\label{rmk:normalcone-dim}
As explained in \cite[B.6.6]{fulton}, if $Y$ has pure dimension, then $C_{X/Y}$ also has pure dimension equal to the dimension of $Y$. Compare this with Remark \ref{rmk:local-picture}.
\end{remark}

\noindent
\textbf{Exercises}
\begin{enumerate}[label=1.\arabic*, leftmargin=*]
\item\label{cone1}Let $X$ be an algebraic stack and let $C = \relSpec_X(S^\bullet)$ be a cone over $X$. The \emph{abelian hull} of $C$ is the abelian cone $A(C) = \relSpec_X(Sym(S^1))$ (see \cite[p. 50]{BF}).
\begin{itemize}
\item[(a)] Show that there is a natural closed embedding $C \hookrightarrow A(C)$. Apply this to construct the closed embedding \eqref{eq:cone1}.
\item[(b)] Show that $C \hookrightarrow A(C)$ has the following universal property: If $E$ is an abelian cone on $X$ and $C \hookrightarrow E$ is a closed embedding, then there is a unique factorization
\[
\begin{tikzcd}
C \arrow[r] \arrow[rr, bend left] & A(C) \arrow[r, dashrightarrow] & E.
\end{tikzcd}
\]
Moreover the arrow $A(C) \to E$ is a closed embedding.
\end{itemize}

\item \label{cone2} Prove Lemma \ref{lem:cone3}.
\item \label{cone3} \begin{itemize}
\item[(a)] Prove Lemma \ref{lem:cone2}.
\item[(b)] Show moreover that the rank of $N_{X/Y}$ is the codimension of $X$ in $Y$.
\end{itemize}
\item \label{cone4} Let $k$ be a field, let $Y = \Spec(k[x]/x^2)$, and let $X \subset Y$ be the closed subvariety with ideal sheaf $(x)$. Compute $C_{X/Y} \subset N_{X/Y}$ and show that these cones are not equal.

\item \label{cone5} Let $k$ be a field, let $Z = \AA^2_k$, let $Y = \VV(y^2+x^3-x^2) \subset Z$, and let $X \subset Y$ be the origin. Describe $C_{X/Y}, N_{X/Y}, C_{Y/Z}$, and $N_{Y/Z}$.

\item \label{cone6} This exercise, based on \cite[Example 3.5]{BCM}, will show that we can have $C_{X/Y} = N_{X/Y}$ even when $X \hookrightarrow Y$ is not a regular embedding.
Let $Y = \AA^2_k$ and let $X = \VV(x^2, xy) \subset Y$.
\begin{itemize}
 \item[(a)] Show that $X$ is the $y$-axis with an embedded point at the origin in the $x$-direction.
\item[(b)] Show that there are closed embeddings $C_{X/Y} \to N_{X/Y} \to X \times \AA^2_k$ where the composition is induced by the ring map
\[
k[x,y]/(x^2, xy)[A, B] \to \bigoplus_{n \geq 0} \mls I^n/\mls I^{n+1} \quad \quad \quad A \mapsto x^2, B \mapsto xy
\]
\item[(c)] Show that the ideal sheaves of the embeddings $C_{X/Y} \to X\times \AA^2_k$ and $N_{X/Y} \to X \times \AA^2_k$ are both generated by $yA-xB$. Conclude that $C_{X/Y} = N_{X/Y}$.
\item[(d)] Show that the fiber of $C_{X/Y}$ over the origin is isomorphic to $\AA^2$, and that the fiber over the generic point of the $y$-axis is isomorphic to $\AA^1$.
\item[(e)] Show that $X \hookrightarrow Y$ is not regular. (\textit{Hint: A regularly embedded subscheme of a smooth scheme is Cohen-Macaulay \cite[Tag 00SB]{stacks-project}.})
\end{itemize}

\item \label{cone7} This exercise, based on \cite[Example 3.3]{BCM}, will show that when $X$ has multiple components, it is not necessarily true that $C_{X/Y}$ is the union of the normal cones of the components. (Note that, by \cite[Sec 3.1]{BCM}, if $Y$ has multiple components then $C_{X/Y}$ is the union of the normal cones of $X$ in each of the components.) Let $Y = \PP^3$ with homogeneous coordinates $[x:y:z:w]$ and let $X = \VV(xz, yz) \subset Y$.
\begin{itemize}
\item[(a)] Show that $X$ is the union of a projective line $L$ and a projective plane $H$.
\item[(b)] Show equalities $C_{L/Y} = N_{L/Y}$, $C_{H/Y} = N_{H/Y}$, and $C_{X/Y} = N_{X/Y}$.
\item[(c)] Show that neither of the canonical morphisms $C_{L/Y} \to C_{X/Y}$ or $C_{H/Y} \to C_{X/Y}$ is a closed embedding, and that their union $C_{L/Y} \sqcup C_{H/Y} \to C_{X/Y}$ is not set-theoretically surjective, by considering the fiber of the composition
\[
C_{L/Y} \sqcup C_{H/Y} \to C_{X/Y} \to X
\]
over the intersection point $L \cap H \in X$.

\end{itemize}

\item \label{cone8}
Let $Y$ be a smooth scheme and $E$ be a vector bundle on $Y$.
Let $s$ be a section of $E$, and let $X = \mathbb{V}(s) \subseteq Y$.
Show that $C_{X/Y}$ is the limit of the graph of $t^{-1}s$ as $t$ approaches $0$.
More precisely, let
$$
    \Gamma = \left\{(t^{-1}s(y), t) \in E \times (\AA^{1}\setminus 0) \mid
        y \in Y\right\},
$$
and let $\overline{\Gamma}$ be the closure of $\Gamma$ in $E \times \AA^{1}$.
Show that $C_{X/Y} \cong \overline{\Gamma} \times_{\AA^{1}} 0$.
Moreover, show that the embedding $C_{X/Y} \subseteq E$ coincides with the one
given by the canonical obstruction theory.
\textit{(Hint: Show that $\overline{\Gamma}$ is isomorphic to the total space of the
    deformation to the normal cone
    (see \cite[Remark~5.1.1]{fulton}).}
\end{enumerate}

\section{First construction of an obstruction theory and virtual class}

\label{sec:Lec2}

In this section we will give our first definition of a virtual class. Under this definition, a virtual class is identical to the localized Euler class defined in \cite[Sec 14.1]{fulton}. To set up the definition we recall a fundamental construction in intersection theory (see \cite[Chapter 6]{fulton}). In this section, all $X$, $Y$, etc. are $k$-schemes.

\subsection{A fundamental construction: Gysin pullback}

\begin{definition}\label{def:gysin}
Let $i: X \hookrightarrow Y$ be a regular embedding of schemes of codimension $d$, and suppose we have a fiber square
\[
\begin{tikzcd}
X' \arrow[r] \arrow[d, "f"] & Y' \arrow[d] \\
X \arrow[r, "i"] & Y.
\end{tikzcd}
\]
The  \emph{refined Gysin homomorphism} $i^!: A_*(Y') \to A_{*-d}(X')$  is defined to be the composition
\[
A_*(Y') \xrightarrow{\sigma} A_*(C_{X'/Y'}) \xrightarrow{j_*} A_*(f^*N_{X/Y}) \xrightarrow{(p^*)^{-1}} A_{*-d}(X'),
\]
where
\begin{itemize}
\item $\sigma$ sends the cycle corresponding to an (irreducible)
    subvariety $V \subset Y'$ to $[C_{V \cap X' / V}]$ (this defines a morphism of Chow groups by \cite[Prop 5.2]{fulton}).
\item $j_*$ is proper pushforward along the closed embeddings $C_{X'/Y'} \to f^*C_{X/Y} \xrightarrow{\sim} f^*N_{X/Y}$ given by Lemmas \ref{lem:cone2} and \ref{lem:cone3}, respectively.
\item $(p^*)^{-1}$ is inverse to the flat pullback along $p: f^*N_{X/Y} \to X'$. (Note that $N_{X/Y}$ is a vector bundle of rank $d$ by Lemma \ref{lem:cone3}, and so this flat pullback is an isomorphism by \cite[Thm 3.3(a)]{fulton}.
\end{itemize}
\end{definition}

\begin{remark}
We have given a formula for $\sigma$ on any subvariety of $Y'$: if $V \subseteq Y'$ is a subvariety, then $\sigma([V]) = [C_{V \times_Y X / V}].$ One extends this formula linearly to define $\sigma$ on all cycles. If $V$ is a pure dimensional scheme, then the formula
\[\sigma([V]) = [C_{V \times_Y X / V}]\]
is still true, but if $V$ doesn't have pure dimension then this formula need not hold (Exercise \ref{vc1}).
\end{remark}

Gysin homomorphisms have the following functoriality properties.
\begin{lemma}[{\cite[Thm 6.2]{fulton}}]
\label{lem:pushforward} Suppose we have a cartesian diagram
\[
\begin{tikzcd}
X'' \arrow[r] \arrow[d, "q"] & Y'' \arrow[d, "p"]\\
X' \arrow[d] \arrow[r] & Y' \arrow[d]\\
X \arrow[r, "i"] & Y
\end{tikzcd}
\]
where $i$ is a regular embedding and $p$ is proper. Then
\[
i^!p_*(\alpha) = q_*i^!(\alpha) \quad \quad \quad \quad \text{for all } \alpha \in A_*(Y'').
\]
\end{lemma}

\begin{lemma}[{\cite[Thm 6.4]{fulton}}]
Suppose we have a cartesian diagram
\[
\begin{tikzcd}
X'' \arrow[r] \arrow[d] & Y'' \arrow[d] \arrow[r] & Z \arrow[d, "j"]\\
X' \arrow[d] \arrow[r] & Y' \arrow[d] \arrow[r] & W\\
X \arrow[r, "i"] & Y
\end{tikzcd}
\]
where $i$ and $j$ are regular embeddings. Then
\[
i^!j^!(\alpha) = j^!i^!(\alpha)  \quad \quad \quad \quad \text{for all } \alpha \in A_*(Y').
\]
\end{lemma}

\begin{lemma}[{\cite[Thm 6.5]{fulton}}]
\label{lem:composition} Suppose we have a cartesian diagram
\[
\begin{tikzcd}
X' \arrow[r] \arrow[d] & Y' \arrow[r] \arrow[d] & Z' \arrow[d] \\
X \arrow[r, "i"] & Y \arrow[r, "j"] & Z
\end{tikzcd}
\]
where $i$ and $j$ are regular embeddings of codimensions $d$ and $e$, respectively. Then $j \circ i$ is a regular embedding of codimension $d+e$, and
\[
(j \circ i)^!(\alpha) = i^! j^!(\alpha)  \quad \quad \quad \quad \text{for all } \alpha \in A_*(Z').
\]
\end{lemma}

\subsection{First construction of a virtual class}
Let $X$ be a scheme. The simplest way a virtual class can arise on $X$ is from
the data of a rank-$d$ vector bundle $E$ on a scheme $Y$ cutting out $X$ as the zero locus of a section. This is done as follows.

\begin{construction}\label{const:1} Suppose $X$ is embedded as a closed subscheme of $Y$ cut out by a section $s$ (not necessarily regular) of a vector bundle $E$ of rank $d$ on $Y$. In other words, there is a fiber square
\[
\begin{tikzcd}
X \arrow[r, "i"] \arrow[d] & Y \arrow[d, "s"]\\
Y \arrow[r, "s_E"] & E
\end{tikzcd}
\]
where $s_E$ is the zero section. The closed embedding $\phi:C_{X/Y} \to E|_X$ arising from Lemma \ref{lem:cone2} is called an \emph{obstruction theory} and the Gysin homomorphism $s_E^!: A_*(Y) \to A_{*-d}(X)$ is the associated \emph{virtual pullback}. To be compatible with later notation we will also write the virtual pullback as $i^!_E$. If $Y$ has pure dimension $n$, the corresponding \emph{virtual class} is defined as
\[
[X]^{vir}_{E} := i_E^! [Y] \in A_{n-d}(X).
\]
\end{construction}
\begin{remark}\label{rmk:pure-dim}
In order for the data $s: Y \to E, X = \mathbb{V}(s)$ to define a virtual class on $Y$, we require $Y$ to be pure dimensional. This is for a few reasons.
\begin{itemize}
\item Requiring $Y$ to have pure dimension ensures that $[X]^{vir}_E$ has pure dimension.
\item By requiring $Y$ to have pure dimension we get a simpler formula for $[X]^{vir}_E$ (see Exercise \ref{vc1}).
\item Using this simpler formula, one can prove Lemma \ref{lem:vc1}. The lemma need not hold if $Y$ does not have pure dimension: see \cite[Example 2.6.4]{fulton}.
\end{itemize}
\end{remark}

\begin{remark}
    \label{rmk:bf-versus-bcm}
Let $i: X \to Y$ be a closed embedding of schemes. The reader may notice that in Construction \ref{const:1}, what we really need to define $i^!_E$ is a closed embedding $ C_{X/Y} \to E_{X/Y}$, where $E_{X/Y}$ is a vector bundle on $X$. This suggests one could define an obstruction theory for $i$ to be such a closed embedding, with associated virtual pullback
$i^!_E: A_*(Y) \to A_{*-d}(X)$ defined to be the composition
\[
A_*(Y) \xrightarrow{\sigma} A_*(C_{X/Y}) \xrightarrow{\phi_*} A_*(E_{X/Y}) \xrightarrow{(p^*)^{-1}} A_*(X),
\]
where $\sigma$ and $(p^*)^{-1}$ are defined as in Definition \ref{def:gysin}. This is indeed the approach taken in \cite[Construction 4.3]{BCM}. We de-emphasize it here because we cannot see a way that this construction is a special case of Behrend-Fantechi's construction \ref{const3}.
\end{remark}

Note that $[X]^{vir}_{E}$ is precisely the localized Euler class defined in \cite[Sec 14.1]{fulton}. In particular, by \cite[Section 14.1]{fulton} we have
\begin{equation}\label{lem:vc2}
i_*[X]^{vir}_{E} = c_{top}(E) \cap [Y].
\end{equation}
An important example of a virtual class arises when the section $s$ defining $X$ happens to be quasi-regular, as explained in the following lemma. (This lemma is a not a corollary of \eqref{lem:vc2} as $i_*$ may not be injective.)

\begin{lemma}\label{lem:vc1}
If $X$ is the zero locus of a quasi-regular section $s$ of a vector bundle $E$ on a pure-dimensional scheme $Y$, then the associated virtual class is the usual fundamental class:
\[
[X]^{vir}_E = [X].
\]
\end{lemma}

\begin{example}\label{ex:two-classes}
A space can carry various virtual classes. For example, let $X = \PP^1_k$.

On the one hand, $X \subseteq X$ is the zero locus of the zero section of the zero vector bundle on $X$. The associated virtual class is $[X]$.

On the other hand, $X$ can be embedded as the twisted cubic in $\PP^3_k = \mathrm{Proj}(k[x, y, z, w])$. Via this embedding, it is the zero locus of the section of $\mls O_{\PP^3}(2)^{\oplus 3}$ given by $(xz-y^2, yw-z^2, xw-yz)$. The associated virtual class is $8[pt]$ (see Exercises \ref{vc10} and \ref{excess3}).
\end{example}

\subsection{Compatible obstruction theories}\label{sec:compatible-OTs} Finally, virtual classes enjoy the following functoriality property. Suppose we have a sequence of vector bundles
\[
0 \to E' \to E \to E'' \to 0
\]
on a pure-dimensional scheme $Y$, and assume this sequence arises by applying $\relSpec_Y(Sym(-))$ to a short exact sequence of locally free sheaves on $Y$. Suppose $s: Y \to E$ is a section with vanishing locus $X$. Let $X''$ be the vanishing locus of the induced section of $E''$, and note that we have closed embeddings
\[
X \xrightarrow{i} X'' \rightarrow Y.
\]
\begin{lemma}\label{lem:vc3}
There is a section $s'$ of $E'|_{X''}$ such that $\mathbb{V}(s') = X$, and
\[
i^!_{E'} [X'']^{vir}_{E''} = [X]^{vir}_E.
\]
\end{lemma}

\noindent
\textbf{Exercises}
\begin{enumerate}[label=2.\arabic*, leftmargin=*]
\item\label{vc10}
\begin{enumerate}
\item Justify the claims in Example \ref{ex:two-classes}. To compute the second virtual class, use \eqref{lem:vc2}.
\item Generalize the first part of the example by showing that when $X$ is a pure dimensional scheme, the fundamental class $[X]$ can be realized as a virtual fundamental class with respect to the identity embedding $X \hookrightarrow X$. (You can appeal to Lemma \ref{lem:vc1}, but try doing it directly.)
\end{enumerate}
\item\label{vc0} The goal of this exercise is to show that if $p: E \to Y$ is a rank-$d$ vector bundle on a scheme $Y$, the inverse to flat pullback $(p^*)^{-1}: A_*(E) \to A_{*-d}(Y)$ can also be described as a Gysin homomorphism (see \cite[Cor 6.5]{fulton}).
\begin{itemize}
\item[(a)] If $s$ is any section of $E$, show that $s: Y \to E$ is a regular embedding.
\item[(b)] Show that $s^!: A_*(E) \to A_{*-d}(Y)$ is inverse to $p^*$. (\textit{Hint: Use Lemma \ref{lem:composition}, or rather its extension to l.c.i. morphisms in \cite[Prop 6.6]{fulton}.} Alternatively, use Lemma~\ref{lem:pushforward}.)
\item[(c)] If $X = \VV(s)$ and $i: X \hookrightarrow Y$ is the inclusion, and if $0_E: Y \to E$ is the zero section, show that
\[
i_*s^!(\alpha) = (p^*)^{-1}(0_{E, *}\alpha) \quad \quad \quad \quad\text{for all }\alpha \in A_*(Y).
\]
(\textit{Hint: Lemma \ref{lem:pushforward}.})
\end{itemize}
\item\label{vc1} Let $X$ be the zero locus of a section $s$ of a vector bundle $E$ on a scheme $Y$, and let $i: X \to Y$ be the embedding.
\begin{itemize}
\item[(a)] Show that if $Y$ has pure dimension, then the morphism $\sigma: A_*(Y) \to A_*(X)$ in Definition \ref{def:gysin} (taking $X=X'$ and $Y=Y'$) sends $[Y]$ to $[C_{X/Y}]$. Note that $Y$ need not be irreducible.
\item[(b)] Show that if $Y$ does not have pure dimension then the formula $\sigma([Y]) = [C_{X/Y}]$ need not hold.
\item[(c)] Use (a) to show that if $Y$ has pure dimension, then $[X]^{vir}_E$ is intersection of $C_{X/Y}$ with zero section of $i^*E$; that is, \[[X]_E^{vir} = (p^*_E)^{-1}[C_{X/Y}]\]
where $(p^*_E)^{-1}$ is inverse to flat pullback along $p_E: i^*E \to X$.

\end{itemize}

\item\label{vc11} Let $X$ be embedded as a closed subscheme of an $n$-dimensional scheme $Y$, cut out by a section of a rank-$d$ vector bundle $E$ on $Y$. Use Remark \ref{rmk:normalcone-dim} and perhaps Exercise \ref{vc1} to explain why $[X]^{vir}_E$ is an ($n-d$)-cycle.

\item \label{vc2} If $X \subset \PP^n$ is cut out by homogeneous polynomials $f_i$ of degree $d_i$, show that $X$ is the zero locus of a section of $\oplus \mls O(d_i)$. Hence such $X$ carries a virtual class (induced by the choice of equations $f_i$). Write down the map of $\mls O_X$-algebras inducing the corresponding obstruction theory.

\item \label{vc3} Use the following steps to prove Lemma \ref{lem:vc1}. Let $X$ be a closed subscheme of $Y$ equal to the zero locus of a quasi-regular section $s$ of a vector bundle $E \to Y$ of rank $d$. Recall that since $Y$ is locally Noetherian, quasi-regularity of $s$ is equivalent to Koszul-regularity.
\begin{itemize}
\item[(a)] Let $\mls E$ be the sheaf of sections of $E \to Y$, so $E = \relSpec_Y(Sym(\mls E^\vee))$. If $s: \mls O_Y \to \mls E$ is a Koszul-regular section defining $X$, show there is an exact sequence
\begin{equation}\label{eq:koszul}
0 \to \Lambda^d \mls E^\vee \ldots \to \Lambda^2 \mls E^\vee \xrightarrow{f} \mls E^\vee \to I \to 0
\end{equation}
where $I$ is the ideal sheaf. What is the map $f$? (If you are not familiar with the Koszul complex, consult the Wikipedia article.)
\item[(b)] By restricting \eqref{eq:koszul} to $X$, show that there is an isomorphism $\mls E^\vee|_X \to I/I^2$ induced by $s$.
\item[(c)] Deduce that the obstruction theory $C_{X/Y} \hookrightarrow E$ is an isomorphism in this situation. Use this and Exercise \ref{vc1} to prove Lemma \ref{lem:vc1}.
\end{itemize}

\item \label{vc4} Let $X$ be the zero locus of a section $s$ of a vector bundle $E$ on a scheme $Y$, and let $i: X \to Y$ be the embedding. Prove equation \eqref{lem:vc2}. (\textit{Hints: Lemma \ref{lem:pushforward}, Exercise \ref{vc0}, and the self intersection formula.})

\item \label{vc5}
Prove Lemma \ref{lem:vc3} using the following steps.
\begin{itemize}
\item[(a)] Show there is a fiber diagram
\[
\begin{tikzcd}
X \arrow[r] \arrow[d] & X'' \arrow[r] \arrow[d] & Y \arrow[d, "s"]\\
Y \arrow[r, "s_{E'}"] & E' \arrow[r, "f"] \arrow[d] & E \arrow[d] \\
& Y \arrow[r, "s_{E''}"] & E''.
\end{tikzcd}
\]
The morphism $X'' \to E'$ defines the desired section $s'$.
\item[(b)] Show that $f: E' \to E$ is a regular embedding and that the Gysin homomorphism $A_*(Y) \to A_*(X'')$ induced by $f$ is equal to that induced by $s_{E''}$.
\item[(c)] Use Lemma \ref{lem:composition}
to deduce that $i^!_{E'}[X'']^{vir}_{E''} = [X]^{vir}_{E}.$
\end{itemize}

\item \label{vc6}  Just as vector bundles have Chern classes, cones have Segre classes \cite[Section 4.1]{fulton}. If a pure-dimensional cone $C$ is contained in a vector bundle $E$, the example \cite[Example 4.1.8]{fulton} gives a formula for the intersection of $C$ with the zero section of $E$ in terms of Segre classes of $C$ and Chern classes of $E$. Use this and Exercise \ref{vc1} to give a formula for the virtual class.
\end{enumerate}

\section{Excess intersection formula}

In this section, all $X$, $Y$, etc. are $k$-schemes unless otherwise stated.

Let $Y$ be $\PP^3_k$ with homogeneous coordinates $[x:y:z:w]$ and consider the twisted cubic $X = \VV(xz-y^2, yw-z^2, xw-yz)$ in $Y$. The twisted cubic is isomorphic to $\PP^1$ and in particular it is smooth and 1-dimensional. We computed the virtual class on $X$ associated to this embedding in Example \ref{ex:two-classes}: it was $8[pt]$. In this section we will present a theorem (Lemma \ref{lem:excess1}) that can be used as an alternative approach to computing that class.

Before stating that result we need some preliminaries. In the following lemma, a morphism of vector bundles is defined to be a morphism of the underlying cones (see Definition \ref{def:cone}).
\begin{lemma}\label{lem:quotient}
Let $E \to F$ be a  closed embedding of vector bundles over a scheme or Deligne-Mumford stack $X$. Then there exists a short exact sequence of locally free sheaves
\[
0 \to \mls K \to \mls F \xrightarrow{f} \mls E \to 0
\]
such that applying the functor $\relSpec_X(Sym(-))$ to $f: \mls F \to \mls E$ yields the closed embedding $E \to F$.
\end{lemma}

\begin{definition}\label{def:quotient-bundle}
Let $E \to F$ be a  closed embedding of vector bundles over a scheme or Deligne-Mumford stack $X$. The quotient $F/E$ is defined to be the vector bundle on $X$ associated to the locally free sheaf $\mls K$ arising in Lemma \ref{lem:quotient}.
\end{definition}

\noindent
Our statement of the Excess Intersection Formula is copied from \cite[Thm 6.3]{fulton}.

\begin{theorem}[Excess Intersection Formula]\label{thm:excess}
Consider a fiber diagram
\[
\begin{tikzcd}
X'' \arrow[r] \arrow[d, "q"] & Y'' \arrow[d, "p"]\\
X' \arrow[r, "i'"] \arrow[d, "g"] & Y' \arrow[d] \\
X \arrow[r, "i"] & Y
\end{tikzcd}
\]
where $i$ and $i'$ are regular embeddings of codimensions $d$ and $d'$, respectively. By Lemma \ref{lem:cone2} there is a closed embedding $N_{X'/Y'} \to g^*N_{X/Y}$.
Then there is an equality of homomorphisms $A_*(Y'') \to A_{*-d}(X'')$
\[
i^!(-) = c_{d-d'}(E) \cap i'^!(-), \quad \quad \quad \text{where }E := g^*N_{X/Y}/N_{X'/Y'}.
\]
The bundle $E$ is called the \emph{excess bundle}.

\end{theorem}

The excess intersection formula has the following consequence for virtual classes.

\begin{lemma}\label{lem:excess1}
Let $Y$ be pure dimensional and let $i: X \hookrightarrow Y$ be a regular embedding cut out by a section (not necessarily regular!) of a vector bundle $E$. Then there is a closed embedding $N_{X/Y} \hookrightarrow i^*E$, and
\begin{equation}\label{eq:excess1}
[X]^{vir}_E = c_{top}(i^*E/N_{X/Y}) \cap [X].
\end{equation}
\end{lemma}

\begin{remark}
If $X \hookrightarrow Y$ is a regular embedding, there can be many possible bundle-section pairs $(E, s)$ such that $X = \mathbb{V}(s)$, and these different pairs $(E, s)$ define different virtual classes on $X$. Moreover, such a section $s$ need not be a regular section of $E$ even though $X \hookrightarrow Y$ is a regular embedding.
For example, if $X = \mathbb{V}(s)$ for some section $s$ of a bundle $E$, then the bundle $E \oplus \mls O_Y$ has a section $(s, 0)$ that is certainly not regular, but $\mathbb{V}((s, 0)) = X$.
A second example is the twisted cubic:
 in Exercise \ref{excess0} we see that the twisted cubic in $\PP^3$ is regularly embedded, but cut out by a section of $\mls O_{\PP^3}(2)^{\oplus 3}$ that is not regular.

By Lemma \ref{lem:excess1}, if the section $s$ of $E$ cutting out $X$ is actually quasi-regular, then $[X]^{vir}_E = [X]$, recovering Lemma \ref{lem:vc1}.
\end{remark}

\noindent
\textbf{Exercises}
\begin{enumerate}[label=3.\arabic*, leftmargin=*]
\item \label{excess-1}
Prove Lemma \ref{lem:quotient}.
\item \label{excess1} Use the excess intersection formula to prove Lemma \ref{lem:excess1}.
\item \label{excess2} Recall that if $i: X \hookrightarrow Y$ is a regular embedding of codimension $d$ with normal bundle $N$, then there is a \emph{self intersection formula}
\[
i^!i_*(\alpha) = c_{d}(N)\cap \alpha
\]
where $ i^!$ is the Gysin homomorphism $A_*(Y) \to A_{*-d}(X)$.
\begin{itemize}
\item[(a)] Use this formula to prove that if $p: E \to X$ is a vector bundle on $X$ and $F \subseteq E$ is a subbundle such that $E/F$ is also a vector bundle, then
\begin{equation}\label{eq:excess8}
(p_E^*)^{-1}[F] = c_{top}(E/F) \cap [X].
\end{equation}
Note that this is a special case of the formula \cite[Example 4.1.8]{fulton}.
\item[(b)] Use \eqref{eq:excess8} to prove \eqref{eq:excess1} without reference to the excess intersection formula. (Note that in \cite[Sec 6.3]{fulton} the excess intersection formula is used to prove the self intersection formula.)
\end{itemize}

\item\label{excess0} Show that inclusion of the twisted cubic $\VV(xz-y^2, yw-z^2, xw-yz) \hookrightarrow \PP^3$ is a regular embedding. Explain why a section of  $\mls O_{\PP^3}(2)^{\oplus 3}$ cutting out the twisted cubic cannot be regular.

\item \label{excess3}
Let $X$ be the twisted cubic $\VV(xz-y^2, yw-z^2, xw-yz)$ in $Y = \PP^3$.
We already computed the virtual class of $X \simeq \PP^1$ with respect to this embedding in Exercise \ref{vc10}; in this exercise we will compute it again using the excess intersection formula, following \cite[Example 5.1]{BCM}.
\begin{itemize}
\item[(a)] Construct a morphism $\mls O_{\PP^3}(2)^{\oplus 3} \to \mls O_{\PP^3}(3)$ sending a section $(A, B, C)$ to $wA + yB - zC$. Show that the pullback of this morphism to $\PP^1$ has kernel $N_{X/Y}$. (\textit{Hint: this morphism arises as the dual of a partial presentation
of the ideal sheaf of $X$ in $Y$.})

\item[(b)]
Show that there is an exact sequence
\[
0 \to N_{X/Y} \to \mls O_{\PP^1}(6)^{\oplus 3} \to \mls O_{\PP^1}(8) \to 0.
\]
(\textit{Hint: show that the cokernel of the inclusion $N_{X/Y} \to \mls O_{\PP^1}(6)^{\oplus 3}$ is a subsheaf of $\mls O_{\PP^1}(9)$, hence is torsion free.})

\item[(c)] Use \eqref{eq:excess1} to conclude that the virtual class on $\PP^1$ induced from its identification with the twisted cubic is $8[pt]$ where $[pt]$ is the class of a point.
\end{itemize}
\item \label{excess5} Let $X = \VV(x^2, xy)$ be cut out of $Y = \PP^2$ by the bundle $E = \mls O_{\PP^2}(2) \oplus \mls O_{\PP^2}(2)$.
We will compute the virtual class of $X$ (induced by the defining equations of $X$ as in see Exercise \ref{vc2}) following \cite[Example 4.8]{BCM}.

\begin{itemize}

\item[(a)] Show that $C_{X/Y}$ is the union of two subschemes: $C_0$, with projection $C_0 \to X$ supported over $x=y=0$ in $X$; and $L$, with projection $L \to X$ a line bundle on the reduced subscheme of $X$. (Hint: use the computation of the normal cone in Exercise \ref{cone6}.)
Hence
\[
(p_E^*)^{-1}[C_{X/Y}] = (p_E^*)^{-1}[L] + (p_E^*)^{-1}[C_0].
\]

\item[(b)] Show that $C_0$ is the spectrum of $k[x,y,A,B]/(x^2, xy, y^2, yA-xB)$. Hence its underlying reduced scheme is $\AA^2_k$ and its multiplicity is 2. Conclude that
\[
(p^*_E)^{-1}[C_0] = 2[P]
\]
where $[P]$ is the class of a point in $X$.
\item[(c)]
Show that the inclusion $L  \to E$ is the inclusion of the second summand of $E = \mls O_{\PP^1}(2) \oplus \mls O_{\PP^1}(2),$ so in particular $L \simeq \mls O_{\PP^1}(2).$ Hence there is an exact sequence
\[
0 \to L \to \mls O_{\PP^1}(2) \oplus \mls O_{\PP^1}(2) \to \mls O_{\PP^1}(2) \to 0.
\]
Use \eqref{eq:excess8} to conclude that
\[
(p^*_E)^{-1}[L] = 2[P].
\]

\item[(d)] Conclude that $[X]^{vir}_E = 4[P].$

\end{itemize}

\item\label{excess6}

Let $X$ be a scheme. A \underline{short exact sequence of vector bundles} is a sequence of abelian cones that arises from applying $\relSpec_X(Sym(-))$ to a short exact sequence of locally free sheaves on $X$ (compare Definition \ref{def:quotient-bundle}).  Suppose we have the following diagram of short exact sequences of vector bundles on $X$, with all vertical maps given by closed embeddings.
\[
\begin{tikzcd}
0 \arrow[r] & E \arrow[r] \arrow[d, equal] & F \arrow[d] \arrow[r] & G\arrow[d] \arrow[r] & 0 \\
0 \arrow[r] & E \arrow[r] & F' \arrow[r] & G' \arrow[r] & 0
\end{tikzcd}
\]
Show that the square with $F, G, F'$, and $G'$ is fibered.
\end{enumerate}

\section{Behrend-Fantechi obstruction theories I}\label{sec:BFI}
At this point, to define a virtual class on $X$, we need to embed $X$ as the zero locus of a section $s$ of a vector bundle $E$ on a pure-dimensional scheme $Y$: the data $(Y, E, s)$ is used to define $[X]^{\vir} \in A_*(X).$
    On the other hand, we will see in Section \ref{sec:pot} that a \textit{Behrend-Fantechi obstruction theory} on $X$ is a certain kind of morphism $\EE_X \to \LL_X$ in the derived category of $X$, and this data can also be used to define a virtual class $[X]^{vir} \in A_*(X)$. The motivation for the latter construction is that the triple $(Y, E, s)$ defining $X$ may only be available locally on $X$, and the morphism $\EE_X \to \LL_X$ in the derived category is a way of gluing together such local data.
\footnote{When $X$ is a moduli stack, this morphism often arises intrinsically from the deformation theory of $X$.}

The relationship between the two virtual class constructions is made fully precise in Propositions \ref{prop:local1} and \ref{prop:local2}. Informally, how can one see they are related? First of all, if $(Y, E, s)$ is as above and $\sE$ is the dual of the sheaf of sections of $E$, then $s$ defines a morphism $s^\vee: \sE|_X \to \mls I \subseteq \mls O_Y $ where $\mls I$ is the ideal sheaf of $X$ in $Y$. The morphism $\EE_X \to \LL_X$ defining a Behrend-Fantechi obstruction theory should be viewed as a generalization of $s^\vee$.
Second of all, the surjection $s^\vee$ induces the closed embedding $C_{X/Y} \hookrightarrow E|_X$ that defines $[X]^{\vir}$. Similarly, we will see in Section \ref{sec:pot} that the derived category morphism $\EE_X \to \LL_X$ produces a closed embedding $\mf C_X \hookrightarrow st(\EE_X)$ of algebraic stacks that defines $[X]^{\vir}\in A_*(X)$
We remark that the stack $\mf C_X$ is called the \textit{intrinsic normal cone}.

We present the Behrend-Fantechi construction of obstruction theories in Sections 5-7. Section 5 gives technical motivation for the construction, Section 6 presents the construction itself, and Section 7 discusses important features of the construction.

\subsection{A motivating computation}

To motivate the definitions of $\mf C_X$ and $st(\EE_X)$, we present the following computation. Suppose we have a commuting diagram of $k$-schemes

\begin{equation}\label{eq:BF2}
\begin{tikzcd}
X \arrow[r, "i"] \arrow[dr, "j"'] & Y \arrow[d, "f"]\\
 & Y'
\end{tikzcd}
\end{equation}
where $i$ and $j$ are closed embeddings and $f$ is smooth.

By \cite[Tag 06BB(3)]{stacks-project}, there is an exact sequence of coherent sheaves
\[
0 \to \mls N_{X/Y'} \to \mls N_{X/Y} \to i^*\Omega_{Y/Y'} \to 0
\]
where $\mls N_{X/Y}$ (resp. $\mls N_{X/Y'}$) is the conormal sheaf of the embedding $X \hookrightarrow Y$ (resp. $X \hookrightarrow Y'$).
Applying $\relSpec_X(Sym(-))$ yields a sequence of abelian cones
\begin{equation}\label{eq:cotangent}
0 \to i^*T_{Y/Y'} \to N_{X/Y} \to N_{X/Y'} \to 0.
\end{equation}
One can show that $C_{X/Y} \subseteq N_{X/Y}$ is
the preimage of $C_{X/Y'} \subseteq N_{X/Y'}$ (Exercise \ref{BF0}),
hence we have a sequence of cones
\[
    0 \to i^*T_{Y/Y'} \to C_{X/Y} \to C_{X/Y'} \to 0.
\]

\begin{lemma}\label{lem:BF1}
Let $E$ be a vector bundle on $Y$ with a section $s$ such that $i: X \to \VV(s)$ is an isomorphism. Then the the composition $i^*T_{Y/Y'} \to C_{X/Y} \to i^*E$ is a closed embedding, and if we denote the quotient bundle by $E'$ then there is an induced closed embedding $C_{X/Y'} \to E'$ such that
\begin{equation}\label{eq:BF1.1}
(p^*_{i^*E})^{-1}([C_{X/Y}]) = (p^*_{E'})^{-1}([C_{X/Y'}]).
\end{equation}
\end{lemma}

Equation \eqref{eq:BF1.1} says that ``the virtual classes associated to the closed embeddings $C_{X/Y} \to i^*E$ and $C_{X/Y'} \to E'$ are the same.'' In other words, \textit{the virtual class is determined by the closed embedding $C_{X/Y} \to i^*E$ modulo $i^*T_{Y/Y'}$}.

In Section \ref{sec:inc}, we will explain how the italicized phrase can be reformulated in the language of algebraic stacks. The idea is that $i^*T_{Y/Y'}$ is a group scheme over $X$ acting on both $C_{X/Y}$ and $i^*E$ and that the closed embedding $C_{X/Y} \to i^*E$ is equivariant. By definition, this means we have a closed embedding of global quotient stacks (which happen to be schemes)
\[
C_{X/Y'} \simeq [C_{X/Y}/i^*T_{Y/Y'}] \to [i^*E/i^*T_{Y/Y'}] \simeq E'.
\]
In fact, if $Y'$ is itself smooth over $\Spec(k)$, then $i^*T_{Y/\Spec(k)}$ acts on both $C_{X/Y}$ and $i^*E$, and we have a closed embedding of algebraic stacks (which are not schemes in general)
\begin{equation}\label{eq:BF0.2}
[C_{X/Y}/i^*T_{Y/\Spec(k)}] \to [i^*E/i^*T_{Y/\Spec(k)}].
\end{equation}
The algebraic stack $[C_{X/Y}/i^*T_{Y/\Spec(k)}]$ is the \textit{intrinsic normal cone} $\mf C_X$ and the stack $[i^*E/i^*T_{Y/\Spec(k)}]$ is the stack $st(\EE|_X)$ mentioned in the introduction to Section \ref{sec:BFI}. The implication of Lemma \ref{lem:BF1} is that we should be able to read the virtual class associated to $(Y, E, s)$ from the closed embedding \eqref{eq:BF0.2}. We will see in Section \ref{sec:pot} that this embedding can be described by a perfect obstruction theory $\EE_X \to \LL_X$ as well as by the triple $(Y, E, s)$.

\subsection{(Relative) Intrinsic normal cone}\label{sec:inc}
Let $X$ be a Deligne-Mumford stack and let $X \to B$ be a morphism to an algebraic stack $B$, all defined over $k$.\footnote{If you are not comfortable with general stacks, you can safely take $X$ and $B$ to be schemes. The reason we allow more general objects at this point is that when we construct a virtual cycle on the moduli of stable maps to projective space, $X$ will be $\overline{\mathcal{M}}_{g, n}(\PP^n, d)$ and $B$ will be the algebraic stack of prestable marked curves.}
Until now, we have taken $B = \Spec(k)$ for the base, but for applications to Gromov--Witten theory it will be convenient to develop the theory for a more general base $B$.
The \textit{relative intrinsic normal cone }$\mf C_{X/B}$ to a morphism $X \to B$ is defined in \cite[Section 7]{BF}. It is a closed substack of an algebraic stack $\mf N_{X/B}$ over $X$ called the \textit{intrinsic normal space}.\footnote{Behrend--Fantechi \cite{BF} call $\mf N_{X/B}$ the intrinsic normal sheaf.} (In particular, $\mf C_{X/B}$ is an algebraic stack over $X$.) We will discuss the true definition of $\mf C_{X/B} \to \mf N_{X/B}$ in Section \ref{sec:pot}, and here content ourselves with a description (rather than a definition). Suppose we have a commuting diagram
\begin{equation}\label{eq:diagram}
\begin{tikzcd}
U \arrow[r, hookrightarrow] \arrow[d] & \arrow[d] V \\
X \arrow[r] & B
\end{tikzcd}
\end{equation}
such that $U$ and $V$ are schemes, $U \to X$ is \'etale, $V \to B$ is smooth, and $i:U \to V$ is a closed immersion. Then the restriction of $\mf C_{X/B} \to \mf N_{X/B}$ to $U$ is equal to
\begin{equation}\label{eq:inc}
[C_{U/V}/i^*T_{V/B}] \to [N_{U/V}/i^*T_{V/B}].
\end{equation}
We can always find a diagram \eqref{eq:diagram} where moreover $U \to X$ is surjective (see Exercise \ref{BF5}).
However, this description does not uniquely determine the $X$-stacks $\mf C_{X/B} \hookrightarrow \mf N_{X/B}$ (to do that, we would need to specify some kind of descent datum).  We conclude by explaining very concretely what is meant by $[N_{U/V}/i^*T_{V/B}]$ and $[C_{U/V}/i^*T_{V/B}].$\\

\noindent
\underline{Definition of $[N_{U/V}/i^*T_{V/B}]$.} For this, we use that an abelian cone $E$ on $X$ is an abelian group scheme over $X$, and a morphism $d: E^0 \to E^1$ of abelian group schemes on $X$ defines an action of $E^0$ on $E^1$ by the rule
\[
a \cdot b = d(a) + b \quad \quad \quad a \in E^0, b \in E^1.
\]
Notice that the definition of this action uses that both $E^0$ and $E^1$ are abelian groups. Using this action, we can define a functor
\begin{equation}\label{eq:st}  \left\{\begin{array}{c}
\text{2-term complexes $[E^0 \to E^1]$}\\
\text{of abelian cones, with $E^0$ a vector bundle}
\end{array}\right\} \longrightarrow \left\{\begin{array}{c}
\text{algebraic stacks}\end{array}\right\}
\end{equation}
sending $[E^0 \to E^1]$ to $[E^1/E^0]$; i.e., to the quotient of $E^1$ by the $X$-group scheme $E^0$.\footnote{
We need $E^0$ to be locally free so that it is a \textit{smooth} group scheme over $X$. This is part of what is needed for the quotient $[E^1/E^0]$ to be \textit{algebraic} (as opposed to just a stack).
} There is a canonical morphism $i^*T_{V/B} \to N_{U/V}$ (see Exercise \ref{BF1}), and the stack $[N_{U/V}/i^*T_{V/B}]$ is simply the output of the functor  \eqref{eq:st} on the complex $i^*T_{V/B} \to N_{U/V}$. \\

\noindent
\underline{Definition of $[C_{U/V}/i^*T_{V/B}]$.} Since $C_{U/V}$ is not an abelian cone, we cannot directly use the functor \eqref{eq:st} to define $[C_{U/V}/i^*T_{V/B}]$. However, we know $C_{U/V}$ is a closed subscheme of $N_{U/V}$.
\begin{lemma}\label{lem:BF4}
The action of $i^*T_{V/B}$ on $N_{U/V}$ preserves $C_{U/V}$.
\end{lemma}
\noindent
Because of the lemma, it makes sense to quotient $C_{U/V}$ by the action of $i^*T_{V/B}$. We define $[C_{U/V}/i^*T_{V/B}]$ to be this quotient. In fact, the \textit{morphism} $[C_{U/V}/i^*T_{V/B}] \to [N_{U/V}/i^*T_{V/B}]$ is the quotient of the closed embedding $C_{U/V} \to N_{U/V}$ by the action of $i^*T_{V/B}$. Because of this, the morphism of algebraic stacks $\mf C_{X/B} \to \mf N_{X/B}$ is (by definition) a closed embedding.

\subsection{More on cone stacks}\label{sec:conestacks}
If $[E^0 \to E^1]$ is a 2-term complex of abelian cones with $E^0$ a vector bundle, the output $[E^1/E^0]$ of the functor in \eqref{eq:st} is called a \emph{cone stack}.
More generally, if $C$ is a cone (not necessarily abelian) and $E^0$ is a vector bundle with a morphism of cones $E^0 \to C$ (Definition \ref{def:cone}), then $C$ has a canonical closed embedding in an abelian cone $E^1$ and the composition $E^0 \to C \to E^1$ is a morphism of abelian cones. If the induced action of $E^0$ on $E^1$ preserves $C$ then we can take the stack quotient $[C/E^0]$, and this is also called a \emph{cone stack} (in particular, the intrinsic normal cone arises this way). The most general definition of a cone stack is in \cite[Definition 1.5]{BF}; such a stack is characterized by locally arising as a quotient in the way we just described.

If $E^1$ is also a vector bundle then $[E^1/E^0]$ is a \emph{vector bundle stack}. A morphism of 2-term complexes of abelian cones is a commuting diagram
\[
\begin{tikzcd}
{[\; E^0} \arrow[r] \arrow[d] & {E^1 \;]} \arrow[d] \\
{[\; F^0} \arrow[r] & {F^1 \;]},
\end{tikzcd}
\]
and this morphism is a \emph{quasi-isomorphism} if it is obtained by applying $\relSpec_X(Sym(-))$ to a quasi-isomorphism of 2-term complexes of finite rank locally free sheaves. A vector bundle stack $[E^1/E^0]$ only depends on the complex $[E^0 \to E^1]$ up to quasi-isomorphism, in the following sense.

\begin{lemma}\label{lem:BF88}
Suppose we have a quasi-isomorphism of complexes
\[
\phi: [E^0 \to E^1] \to [F^0 \to F^1].
\]
Then there is an induced isomorphism of algebraic stacks $[E^1/E^0] \to [F^1/F^0]$.
\end{lemma}

The lemma has the following corollary, which also furnishes some concrete examples of vector bundle stacks.

\begin{corollary}\label{cor:examples}
Let $f: E^0 \to E^1$ be a morphism of vector bundles on a scheme or Deligne-Mumford stack $X$.
\begin{enumerate}
\item If $f$ is a closed embedding, then $[E^1/E^0]$ is the quotient vector bundle $E^1/E^0$ of Definition \ref{def:quotient-bundle}.
\item If $f$ arises by applying $\relSpec_X(Sym(-))$ to an injective morphism of finite rank locally free sheaves $\mls E^1 \to \mls E^0$ with locally free cokernel $\mls K$, then $[E^1/E^0]$ is the classifying stack of $K$, where $K = \relSpec_X(Sym(\mls K))$ is viewed as an $X$-group scheme.
\end{enumerate}

\end{corollary}
The classifying stack of $K$ in Corollary \ref{cor:examples}, denoted $BK$, is equal to the stack quotient $[X/K]$ where $K$ acts trivially on $X$. As a topological space $[X/K]$ is just $X$, and the stabilizer group at a geometric point $x \to X$ is the fiber of $K$ at $x$.\\

\noindent
\textbf{Exercises.}

\begin{enumerate}[label=4.\arabic*, leftmargin=*]
\item \label{BF0}\begin{itemize}
\item[(a)] Show that $C_{X/Y} \subseteq N_{X/Y}$ is the preimage of $C_{X/Y'} \subseteq N_{X/Y'}$ under the morphism $N_{X/Y} \to N_{X/Y'}$ appearing in \eqref{eq:cotangent}.

\item[(b)]Prove Lemma \ref{lem:BF1}. (\textit{Hint: use Exercise \ref{cone1}.})
\end{itemize}

\item \label{BF0.5} The point of this exercise is to develop intuition for quotients $[E^1/E^0]$ of abelian cones.
\begin{itemize}
\item[(a)] Show that the inclusion of the $x$-axis $\AA^1_k \to \AA^2_k$ is a morphism of abelian cones over $\Spec(k)$. Show that the induced action of $\mathbb{G}_a$ on $\AA^2_k$ is free (meaning that for every $T$-point $T \to \AA^2_k$ the induced action $\AA^1_k(T) \times \AA^2_k(T) \to \AA^2_k(T)$ is free), so the stack $[\AA^2_k/\AA^1_k]$ has a chance of being a scheme (a quotient of a scheme by a free group action could also be an algebraic space). In fact by Corollary \ref{cor:examples} the stack $[\AA^2_k/\AA^1_k]$ is the vector space quotient $\AA^2_k/\AA^1_k$.
\item[(b)] Use Lemma \ref{lem:BF88} to prove Corollary \ref{cor:examples}.
\end{itemize}

\item \label{BF1} In the definition of $[N_{U/V}/i^*T_{V/B}]$, we claimed there was a canonical map $i^*T_{V/B} \to N_{U/V}$. Derive this map from the universal differential $d: \mls O_V \to \Omega_{V/B}.$
\item \label{BF2} Prove Lemma \ref{lem:BF4}. (\textit{Hint: Use the composition
$V \xrightarrow{\Delta_{V}} V \times_{B} V \xrightarrow{p_{1}} V$
to obtain an identification
$i^{\ast}T_{V/B} \times N_{U/V} \cong N_{U/V\times_{B}V}$,
and show that the post-composition of this isomorphism with the natural map
$N_{U/V \times_{B} V} \xrightarrow{p_{2\ast}} N_{U/V}$
is coincides with the action map.})
\item\label{BF3} In this exercise you will compute $\mf C_{X/B} \to \mf N_{X/B}$ in two special cases.
\begin{itemize}
\item[(a)] Describe $\mf C_{X/B} $, $\mf N_{X/B}$, and $\mf C_{X/B} \to \mf N_{X/B}$ when $X \to B$ is a closed embedding of schemes.
\item[(b)] Describe $\mf C_{X/B} $, $\mf N_{X/B}$, and $\mf C_{X/B} \to \mf N_{X/B}$ when $X \to B$ is a smooth morphism of schemes.
\end{itemize}
\item\label{BF5} Let $X$ be a Deligne-Mumford stack and let $X \to B$ be a morphism to an algebraic stack $B$.
\begin{itemize}
\item[(a)] Prove that for any geometric point $\bar x \to X$ there is a square \eqref{eq:diagram} where the image of $U$ contains the image of $\bar x$, $U$ and $V$ are schemes, $U \to X$ is \'etale, $Y \to B$ is smooth, and $U \to V$ is a closed immersion. (\textit{Hints: \cite[Tags 04II, 057G]{stacks-project}.})
\item[(b)] Show that we can moreover assume $U \to X$ is surjective.
\end{itemize}
\item\label{BF4} This requires some knowledge of global quotient stacks.
\begin{itemize}
\item[(a)] Let $\phi: [E^0 \xrightarrow{d} E^1] \to [F^0 \xrightarrow{d} F^1]$ be a morphism of complexes of vector bundles. Construct an associated morphism $[E^1/E^0] \to [F^1/F^0].$
\item[(b)] Let $\phi, \psi: [E^0 \xrightarrow{d} E^1] \to [F^0 \xrightarrow{d} F^1]$ be two morphisms of complexes. A \emph{homotopy} from $\phi$ to $\psi$ is a morphism $k: E^1 \to F^0$ satisfying $\psi^0 = kd + \phi^0$ and $\psi^1 = dk + \phi^1$. Show that a homotopy $\phi \to \psi$ induces a 2-isomorphism of morphisms of algebraic stacks $\phi, \psi:[E^1/E^0] \to [F^1/F^0]$.
\item[(c)] Use part (b) to prove Lemma \ref{lem:BF88}. (\textit{Hint: \cite[Thm 10.4.8]{weibel}.}
\end{itemize}
\end{enumerate}

\section{Behrend-Fantechi obstruction theories II}

In this section we give Behrend-Fantechi's definition of a (perfect) obstruction theory (Definition \ref{def:BF}) and associated virtual class (Construction \ref{const3}) after a brief summary of a prerequisite concept, namely the cotangent complex. The benefit of these newest definitions of an obstruction theory and virtual class is that they do not depend on an embedding $X \to Y$. We will see in Propositions \ref{prop:local1} and \ref{prop:local2} are in fact direct generalizations of Construction \ref{const:1}.

\subsection{Cotangent complex}
Let $X$ be a Deligne-Mumford stack and let $B$ be an algebraic stack over $k$. Every quasicompact and quasiseparated morphism $X \to B$ of algebraic stacks has a \textit{cotangent complex} $\LL_{X/B} \in D_{qc}(X)$ with vanishing cohomology in positive degrees.\footnote{
The category $D_{qc}(X)$ is the derived category of $\mls O_X$-modules in the \'etale topology on $X$ with quasi-coherent cohomology sheaves.
}
In these notes, we will only make use of $\LL_{X/B}$ when $X$ is Deligne-Mumford and $B$ is algebraic, so we restrict ourselves to this case.

This complex should be viewed as a sort of left derived functor to taking K\"ahler differentials and it has functoriality properties analogous to those of the sheaf of differentials. Here are some key properties of the cotangent complex (see \cite[Thm 8.1]{Olsson07}).

\begin{enumerate}[label=(C\arabic*)]
\item \label{eq:C1}
There is a canonical morphism $\LL_{X/B} \to \Omega_{X/B}$ (where we view $\Omega_{X/B}$ as a complex concentrated in degree 0). This morphism is a quasi-isomorphism if $X \to B$ is smooth, and in general it induces an isomorphism $h^0(\LL_{X/B}) \simeq \Omega_{X/B}$.
\item \label{eq:C2} If $X \to B$ admits a factorization $X \xrightarrow{i} B' \xrightarrow{p} B$ where $i$ is a closed immersion with ideal sheaf $\mls I$ and $p$ is smooth, then
\[
\tau_{\geq -1}\LL_{X/B} = [0 \to \mls I/\mls I^2 \xrightarrow{d} i^*\Omega_{B'/B} \to 0]
\]
where $\mls I/\mls I^2$ is in degree -1, $i^*\Omega_{B'/B}$ is in degree 0, $d$ is induced by the differential $\mls O_B \to \Omega_{B/B'}$, and $\tau_{\geq -1}$ is the canonical truncation functor defined in \cite[Tag 0118]{stacks-project}. If $i$ is regular then $\LL_{X/B}$ (no truncation) is equal to $[\mls I/\mls I^2 \xrightarrow{d} i^*\Omega_{B'/B}]$.
\item \label{eq:C3} Given a commuting square
\[
\begin{tikzcd}
X \arrow[r, "f"] \arrow[d] & X' \arrow[d] \\
B \arrow[r] & B'
\end{tikzcd}
\]
there is a canonical map $Lf^* \LL_{X'/B'} \to \LL_{X/B}$. If the square is fibered and either $X' \to B'$ or $B \to B'$ is flat, then this map is a quasi-isomorphism.
\item \label{eq:C4} Given a factorization $X \xrightarrow{h} Y \to B$ of $f$, there is a distinguished triangle in $D_{qc}(X)$:
\[
Lh^* \LL_{Y/B} \to \LL_{X/B} \to \LL_{X/Y} \to.
\]

\end{enumerate}
\begin{example}
Let $p\colon C \to B$ be a family of nodal curves.
Then the canonical morphism $\LL_{C/B} \to \Omega_{C/B}$ is an isomorphism (see Exercise \ref{POT10}).
\end{example}

\subsection{Perfect obstruction theory}\label{sec:pot}
Let $X$ be a scheme or Deligne-Mumford stack over an algebraic stack $B$.

\begin{definition}\label{def:BF}
A \emph{perfect obstruction theory} on $X$ relative to $B$ is a morphism $\phi: \EE_{X/B} \to \LL_{X/B}$ in $D^{\leq 0}_{qc}(X)$ such that
\begin{enumerate}
\item $\EE_{X/B}$ is locally quasi-isomorphic to a complex of finite rank locally free sheaves supported in degrees $[-1, 0]$ (we say that $\EE_{X/B}$ is \emph{perfect}).
\item $\phi$ induces an isomorphism $h^0(\EE_{X/B}) \to h^0(\LL_{X/B})$ and a surjection $h^{-1}(\EE_{X/B}) \to h^{-1}(\LL_{X/B})$ (we say that $\phi$ is an \emph{obstruction theory}).
\end{enumerate}
\end{definition}

Just as the obstruction theory in Construction \ref{const:1} endcodes the embedding of a normal cone in a vector bundle, so also Definition \ref{def:BF} encodes a closed embedding of $\mf C_{X/B}$ in a vector bundle stack as follows.

If $\EE$ is a complex in $D_{qc}(X)$, let $\tau_{[-1,0]}(\EE)$ denote the canonical truncation of $\EE$ to a 2-term complex $[-1,0]$. Writing $\tau_{[-1,0]}(\EE) = [\mls E^{-1} \to \mls E^0],$ we define
\[
st(\EE) := [E^1/E^0], \quad \quad \quad \quad \quad \text{where }E^i = \Spec(Sym(\mls E^{-i})),
\]
and $[E^1/E^0]$ is defined as in \eqref{eq:st}.\footnote{As we've written it here, $E^0$ needn't be a vector bundle. We will only apply this construction in situations where $\EE$ is locally on $X$ represented by a 2-term complex $[\mls E^{-1} \to \mls E^0]$ with $\mls E^0$ locally free, and hence the stack $st(\EE)$ is algebraic, as this can be verified locally on $X$. } (Note that by Lemma \ref{lem:BF88}, the stack $st(\EE)$ only depends on the quasi-isomorphism class of $\EE$.)  Recall that in Section \ref{sec:inc} we did not give a true definition of $\mf N_{X/B}$; we can do this now.
\begin{definition}\label{def:inc}
If $X$ is a Deligne-Mumford stack and $B$ is an algebraic stack, the \textit{intrinsic normal space} to a morphism $X \to B$ is $\mf N_{X/B} := st(\LL_{X/B})$.
\end{definition}

\noindent
The definition of a perfect obstruction theory is chosen so that $st(\EE_{X/B})$ is a vector bundle stack and so that the following lemma holds.
\begin{lemma}\label{lem:BF5}
A perfect obstruction theory $\EE_{X/B} \to \LL_{X/B}$ induces a closed embedding
\[\mf  N_{X/B} = st(\LL_{X/B}) \to st(\EE_{X/B}).\]
\end{lemma}

\noindent
Showing the existence of a closed substack $\mf C_{X/B} \subseteq \mf N_{X/B}$ recovering the local description in Section \ref{sec:inc} is hard work, and a main result of \cite{BF}.
Nevertheless, combining that result with Lemma \ref{lem:BF5} we get a composition of closed embeddings
\begin{equation}\label{eq:closed}
\mf C_{X/B} \to \mf N_{X/B} \to st(\EE_{X/Y}).
\end{equation}
This composition $\mf C_{X/B} \hookrightarrow st(\EE_{X/B})$ is the analog of the closed embedding $C_{X/Y} \hookrightarrow E|_X$ in Construction \ref{const:1},
and we can use this embedding to define a virtual class on $X$.

\subsection{Virtual class}
The definition of a virtual class associated to \eqref{eq:closed} uses flat pullback along the structure morphism of the vector bundle stack $p: st(\EE_{X/Y}) \to X$. In particular, it uses the Chow group of the Artin stack $\EE_{X/Y}$. The reader unfamiliar with Kresch's theory \cite{kresch} for Chow groups of algebraic stacks stratified by global quotient stacks can either take the existence and properties of $p^*$ as a black box or consult Section \ref{sec:globalres} below. Before giving the construction, we define the \textit{rank} of a perfect obstruction theory.

\begin{definition}\label{def:rank}
If $\EE$ is perfect, then locally on $X$ we can write $\EE = [\mls E^{-1} \to \mls E^0]$ with $\mls E^i$ a finite rank locally free sheaf, and the \textit{rank} of $\EE$ is defined as $\mathrm{rank}(\mls E^0) - \mathrm{rank}(\mls E^{-1})$.
\end{definition}

\begin{remark}
Our definition of the rank of $\EE$ agrees with that in \cite{BF}, but is the negative of that in \cite{BCM}.
\end{remark}

\begin{construction}\label{const3}
Let $f: X \rightarrow B$ be a morphism from a scheme or Deligne-Mumford stack $X$ to an algebraic stack $B$, and suppose $\phi: \EE_{X/B} \to \LL_{X/B}$ is a perfect obstruction theory of rank $d$.
The  \emph{virtual pullback} $f_\Phi^!: A_*(B) \to A_{*+d}(X)$  is defined to be the composition
\[
A_*(B) \xrightarrow{\sigma} A_*(\mf C_{X/B}) \xrightarrow{j_*} A_*(st(\EE_{X/B})) \xrightarrow{(p^*)^{-1}} A_{*+d}(X),
\]
where
\begin{itemize}
\item $\sigma$ is defined via specialization to the  normal cone as in \cite[Section 5.1]{kresch}; in particular, it sends the cycle corresponding to a pure-dimensional integral closed substack $V \subset B$ to $[\mf C_{V \times_B X / V}]$.
\item $j_*$ is proper pushforward along \eqref{eq:closed}.
\item $(p^*)^{-1}$ is inverse to pullback along $p: st(\EE_{X/B}) \to X$. (The stack $st(\EE_{X/B})$ is a vector bundle stack, so $p^*$ is an isomorphism by \cite[Prop 5.3.2]{kresch}.)
\end{itemize}
If moreover $B$ has pure dimension $n$, the induced \emph{virtual class} on $X$ is defined to be
\[
[X]^{vir}_\phi := f^!_{\phi}[B] = (p^*)^{-1}[\mf C_{X/B}] \quad \in A_{n+d}(X).
\]

\end{construction}

\begin{remark}
If $B$ is an algebraic stack, then the Chow group $A_*(B)$ is defined in \cite{kresch}. Not all elements of this group correspond  to integral closed substacks of $B$, so our formula for $\sigma$ in Construction \ref{const3} is not complete. A complete description of how $\sigma$ acts on elements of $A_*(B)$ can be found in \cite{LW2}.
\end{remark}

\subsection{Relationship to Fulton's localized top Chern class}

Fulton's localized top Chern class can be recovered as a particular example of a Behrend-Fantechi virtual class:

\begin{proposition}\label{prop:local1}
Let $X$ be the zero locus of a section $s$ of a vector bundle $E$ on a pure-dimensional scheme $Y$ and let $i: X \to Y$ be the inclusion.
If $\mls E$ is the dual of the sheaf of sections of $E$, then setting $\EE_{X/Y} = [\mls E \to 0]$ there is a perfect obstruction theory
$\phi: \EE_{X/Y} \to \LL_{X/Y}$
so that
\[i^{!}_{\phi} = i^{!}_{E},\]
where the left side was defined in Construction \ref{const3} and the right side in Construction \ref{const:1}.
\end{proposition}

If $X$ is affine, there is a relative version of the perfect obstruction theory in Proposition \ref{prop:local1} as follows. Let $X \to Y \to B$ be morphisms of schemes with $B$ pure dimensional, assume $Y \to B$ is smooth, and let $X \to Y$ be the inclusion of the zero locus of a section $s$ of a vector bundle $E$ on $Y$. Let $\mls E$ be the dual of the sheaf of sections of $E$. Then $s$ defines a homomorphism $\mls E \to \mls O_Y$ with image the ideal sheaf $\mls I$ of $X$ in $Y$, and hence its restriction is a morphism $\mls E|_X \to \mls I/\mls I^2$. Composing with the canonical homomorphism $\mls I/\mls I^2 \to \Omega_{Y/B}|_X$ we get a perfect complex
\[
\EE_{X/B} = [\mls E|_X \to \Omega_{Y/B}|_X].
\]

\begin{lemma}\label{lem:lift}
If $X$ is affine, the natural complex map $\EE_{X/B} \to \tau_{\geq -1} \LL_{X/B}$ lifts uniquely to a perfect obstruction theory $\EE_{X/B} \to \LL_{X/B}$.
\end{lemma}

We call the perfect obstruction theory in Lemma \ref{lem:lift} the perfect obstruction theory for $X \to B$ associated to $(Y, E, s)$. Note that if $Y = B$ this perfect obstruction theory recovers the one in Proposition \ref{prop:local1}. The next result says that every perfect obstruction theory is locally one associated to a vector bundle-section pair.

\begin{proposition}\label{prop:local2}
Let $\phi\colon \EE_{X/B} \to \LL_{X/B}$ be a perfect obstruction theory. Then there exist an affine \'etale cover $U \to X$,
a smooth morphism $Y \to B$,
and a vector bundle-section pair $(E, s)$ on $Y$
such that (i) $U\to B$
factors as
$U = \mathbb{V}(s) \to Y \to B$, and (ii) the obstruction theory for $U \to B$ associated to $(Y, E, s)$ is isomorphic to $\phi|_U$.
\end{proposition}

\begin{example}\label{ex:gysin}
The virtual pullback in Construction \ref{const3} generalizes the Gysin homomorphism in Definition \ref{def:gysin}. Let $i: X \to Y$ be a regular embedding of schemes of codimension $d$. Then $\LL_{X/Y}$ is perfect of tor amplitude $[-1, 0]$, and the identity map $\LL_{X/Y} \to \LL_{X/Y}$ is a perfect obstruction theory with the property that $i^!_\phi$ agrees with the Gysin homomorphism $i^!$. If $Y$ has pure dimension then $[X]^{vir}_\phi = [X]$.
\end{example}

The following is a generalization of Lemma \ref{lem:vc1} (see \cite[Prop 7.3]{BF}).
\begin{lemma}\label{lem:BF9}
If $X \to B$ is locally of finite presentation and $\phi:\EE_{X/B} \to \LL_{X/B}$ is a perfect obstruction theory such that $h^{-1}(\EE_{X/B})$ vanishes and $h^0(\EE_{X/B})$ is locally free, then $X \to B$ is smooth and
\[[X]^{vir}_{\phi} = [X].\]
\end{lemma}
\noindent
\textbf{Exercises}
\begin{enumerate}[label=5.\arabic*, leftmargin=*]

\item  \label{POT1} Let $X \to B$ be a morphism and let $U \to X$ be an \'etale morphism (for example an open immersion).
\begin{enumerate}
\item[(a)] Prove that there is a canonical isomorphism $\LL_{X/B}|_U \simeq \LL_{U/B}$. (Hint: use \ref{eq:C1} and \ref{eq:C4}.)
\item[(b)] If $\phi: \EE_{X/B} \to \LL_{X/B}$ is a perfect obstruction theory, prove that $\phi|_U: \EE_{X/B}|_U \to \LL_{X/B}|_U \simeq \LL_{U/B}$ is also an obstruction theory.
\end{enumerate}

\item \label{POT10}
Let $p\colon C \to B$ be a family of nodal curves.
Show that $\LL_{C/B}$ is isomorphic to $\Omega_{C/B}[0]$,
where $\Omega_{C/B}$
is the sheaf of relative differentials.
Bonus: generalise this to families of curves with
planar singularities.
\textit{({Hint: }Reduce this to an \'etale-local computation,
and write down what happens near a node using \cite[Tag 0CBY]{stacks-project}.)}

\item \label{POT2} Show that our construction in \S \ref{sec:inc}
    does indeed give a local description of
    $st(\LL_{X/B}) = \mathfrak{N}_{X/B}$.
\item \label{POT3}
\begin{itemize}
    \item[(a)] Show that if $\phi: \EE_{X/B} \to \LL_{X/B}$ is a perfect obstruction theory, then locally on $X$ we can find a diagram \eqref{eq:diagram} such that $\EE_{X/B}|_U$ is quasi-isomorphic to $\mls E^{-1} \to \mls E^0$ a sequence of locally free sheaves, $(\tau_{\geq -1}\LL_{X/B})|_U$ is quasi-isomorphic to $[\mls I/\mls I^2 \to \Omega_{V/B}|_U]$ where $\mls I$ is the ideal sheaf of $U \to V$, and $\tau_{\geq -1}( \phi|_U)$
    is given by a commuting square
\[
\begin{tikzcd}
\mls E^{-1} \arrow[r] \arrow[d] & \mls E^0 \arrow[d] \\
\mls I/\mls I^2 \arrow[r] & \Omega_{V/B}
\end{tikzcd}
\]
where $\mls E^0 \to \Omega_{V/B}$ is an isomorphism and $\mls E^{-1} \to \mls I/\mls I^2$ is surjective (not just on cohomology!).
\item[(b)] Use (a) to prove Lemma \ref{lem:BF5}.
\end{itemize}

\item \label{POT0} Let $X \to B$ be a morphism and assume $B$ has pure dimension.
\begin{enumerate}
    \item[(a)] Show that the intrinsic normal cone $\mf C_{X/B}$ has dimension $\dim B$. You may use the fact that if $Y$ is a $B$-scheme with pure relative dimension $d$ and $G$ is an algebraic group of relative dimension $e$ over $B$, then the dimension of $[Y/G]$ is $d-e$ (see \cite[Tag 0AFL]{stacks-project}). (Observe that $d-e$ could be negative!) (\textit{Hint: Remark \ref{rmk:normalcone-dim}.})
    \item[(b)] Show that if $\EE$ is perfect (of amplitude $[-1, 0]$), then $st(\EE)$ is a vector bundle stack of relative dimension $-\rank(\EE)$ over $X$.
\item[(c)] Use (a) and (b) to show that if $\phi: \EE_{X/B} \to \LL_{X/B}$ is a perfect obstruction theory, then $[X]^{vir}_\phi$ lives in $A_{\dim B + \rank(\EE_{X/B})}(X).$
\end{enumerate}

\item\label{POT3.5}
Prove the claims in Example \ref{ex:gysin}.

\item \label{POT100} Prove Lemma \ref{lem:lift}.

\item \label{POT5} Prove Propositions \ref{prop:local1} and \ref{prop:local2}. (\textit{Hint: use Exercise \ref{POT3}(a). })

\item \label{POT9} Prove Lemma \ref{lem:BF9}. (\textit{Hint: \cite[Tag 0D0L]{stacks-project}, but note that this is stated for schemes. Note also that $\tau_{\geq -1}\LL_{X/B}$ would be written $NL_{X/B}$ in the notation of the reference.})

\end{enumerate}
\section{Behrend-Fantechi obstruction theories III}

This section discusses various useful features of Behrend-Fantechi's construction.

\subsection{Global resolutions}\label{sec:globalres}
Let $\phi: \EE_{X/B} \to \LL_{X/B}$ be a perfect obstruction theory. We explain an important technical assumption on $\phi$ that in the original paper \cite{BF} of Behrend-Fantechi was necessary to define $[X]^{vir}_\phi$, and even today yields a clean formula for the virtual class (see Section \ref{sec:siebert}).

A \textit{global resolution} for $\EE_{X/B}$ is a 2-term complex $\mathbb{F} = [\mls F^{-1} \to \mls F^0]$ of locally free sheaves and a quasi-isomorphism
\begin{equation}\label{eq:globalres}
\EE_{X/B} \xrightarrow{\sim} \FF.
\end{equation}
These data induce a morphism over $X$
\[
\pi: \relSpec_X(Sym(\mls F^{-1})) \to st(\EE_{X/B})
\]
that is a smooth cover of the stack $st(\EE_{X/B})$ by the scheme $F^1 := \relSpec_X(Sym(\mls F^{-1})).$

In the presence of a global resolution $[X]^{vir}_\phi$ may be computed as follows. Let $C(F)$ denote the fiber product
\[
\begin{tikzcd}
C(F) \arrow[r] \arrow[d] & F^1 \arrow[d, "\pi"]\\
\mf C_{X/B} \arrow[r] & st(\EE_{X/B}).
\end{tikzcd}
\]
Let $p_{F^1}: F^1 \to X$ and $p_{st(\EE_{X/B})}: st(\EE_{X/B}) \to X$ be the structure morphisms. Kresch's paper \cite{kresch} shows that since $p_{F^1} = p_{st(\EE_{X/B})} \circ \pi$, we have an equality $p_{F^1}^* = \pi^* \circ p_{st(\EE_{X/B})}^*$. It follows that
\begin{equation}\label{eq:no-artin}
[X]^{vir}_\phi = (p^*_{st(\EE_{X/B})})^{-1}[\mf C_{X/B}] = (p_{F^1}^*)^{-1} \pi^*[\mf C_{X/B}] = (p_{F^1}^*)^{-1} [C(F)].
\end{equation}
In other words, we can compute $[X]^{vir}_{\phi}$ by intersecting $C(F)$ with the zero section of the vector bundle $F^1$.
Note that the formula \eqref{eq:no-artin} for $[X]^{vir}_\phi$ does not require any intersection theory on Artin stacks; this is the reason why global resolutions were originally an essential technical assumption in \cite{BF}.

Using global resolutions, we can state the following generalization of Lemma \ref{lem:excess1} (see \cite[Prop 7.3]{BF}).

\begin{lemma}\label{lem:BF8}
If $X \to B$ is smooth and $\phi:\EE_{X/B} \to \LL_{X/B}$ is a perfect obstruction theory with a global resolution, then $E := \relSpec_X(Sym(h^{-1}(\EE_{X/B})))$ is a vector bundle on $X$ and
\[
[X]^{vir}_{\phi} = c_{top}(E) \cap [X].
\]
\end{lemma}

\subsection{A practical criterion to be an obstruction theory}

The fundamental theorem for the cotangent complex $\LL_{X/B}$ says that this complex controls the deformation theory of maps to $X$. The following result says that $\phi: \EE_{X/B} \to \LL_{X/B}$ is an obstruction theory if and only if $\EE_{X/B}$ also controls the deformation theory of maps to $X$ (in an appropriate sense). This criterion is often used to show that a moduli stack $X$ has an obstruction theory. To state the criterion, recall that given a solid commuting diagram
\begin{equation}\label{eq:deform}
\begin{tikzcd}
T \arrow[r, "f"] \arrow[d] & X \arrow[d] \\
T' \arrow[r] \arrow[ur, dashrightarrow] & B
\end{tikzcd}
\end{equation}
with $T$ a scheme and $T \to T'$ a square zero extension by a quasicoherent sheaf of ideals $I$ on $T$, there is an obstruction $o(f) \in \Ext^1(Lf^*\LL_{X/B}, I)$ whose vanishing is necessary and sufficient for the existence of a dotted arrow (and appropriate 2-morphisms) making \eqref{eq:deform} (2-)commute.
Explicitly, the obstruction $o(f)$ is the following composition of morphisms of cotangent complexes induced by the square \eqref{eq:deform}:
\[
    Lf^*\LL_{X/B} \to \LL_{T/B} \to \LL_{T/T'} \to \tau_{\geq -1}\LL_{T/T'} = I[1].
\]

When $o(f)=0$ the set of such lifts is a torsor under $\Ext^0(Lf^*\LL_{X/B}, I)$. Finally, observe that an obstruction theory $\phi: \EE_{X/B} \to \LL_{X/B}$ induces morphisms
\[
\Phi^i: \Ext^i(Lf^*\LL_{X/B}, I) \to \Ext^i(Lf^*\EE_{X/B}, I).
\]

\begin{proposition}\label{prop:criterion}
The following are equivalent.
\begin{enumerate}
\item $\phi$ is an obstruction theory.
\item For every square of the form \eqref{eq:deform},
\begin{itemize}
    \item[(a)] $\Phi^1(o(f))=0$ if and only if the set of lifts in \eqref{eq:deform} is nonempty, and
\item[(b)] if $\Phi^1(o(f))=0$ then $\Phi^0$ is an isomorphism.
\end{itemize}
\end{enumerate}
\end{proposition}

For a proof, see e.g. \cite[Thm~4.5]{BF} or \cite[Lem~4.3.2]{webb}.

\subsection{Functoriality}

The virtual pullbacks defined in Construction \ref{const3} are a generalization of the Gysin homomorphisms defined in Definition \ref{def:gysin} (see Exercise \ref{POT3.5}).
Thus, it is natural and desirable to develop a
calculus for virtual pullbacks parallel
to the functoriality properties of Gysin pullbacks
in \S\ref{sec:Lec2}.
This was done in \cite{manolache-pullback},
and we collect some of the main results below.

\begin{remark}\label{rmk:POT2.5}
    Suppose that we have a cartesian diagram of algebraic stacks
    \begin{equation}\label{eq:this}\begin{tikzcd}
        X^{\prime} \arrow[r, ] \arrow[d, "q"'] & B^{\prime} \arrow[d, "p"]\\
        X \arrow[r, "f"] & B
    \end{tikzcd}\end{equation}
    with $X$ and $X'$ Deligne-Mumford, and let $\phi: \EE_{X/B} \to \LL_{X/B}$ be a perfect obstruction theory for $f$. Then the composition $Lp^*\EE_{X/B} \to Lp^*\LL_{X/B} \to \LL_{X'/B'}$ is a perfect obstruction theory (Exercise \ref{POT14}).
    It is customary to denote the induced virtual pullback $A_*(B') \to A_*(X')$ by $f^!_\phi$.
\end{remark}

\begin{lemma}[Pushforward, {\cite[Theorem 4.1(i)]{manolache-pullback}}]\label{lem:functoriality1}
    Suppose we have a cartesian diagram \eqref{eq:this} with $X$ and $X'$ Deligne-Mumford and $p$ projective. Then
    \[
        f_\phi^{!}p_\ast(\alpha) = q_{\ast}f_{\phi}^{!}(\alpha)
    \]
    for all $\alpha \in A_{\ast}(B^{\prime})$.
\end{lemma}

\begin{lemma}[Commutativity, {\cite[Theorem 4.3]{manolache-pullback}}]
    \label{lem:functoriality pull pull}
    Consider a fibre diagram of algebraic stacks
    \[
    \begin{tikzcd}
       Z \arrow[r] \arrow[d] & Y \arrow[d, "g"]\\
       X \arrow[r, "f"] & B
    \end{tikzcd}
    \]
    where $X, Y$, and $Z$ are Deligne-Mumford, $f$ has a perfect obstruction theory $\phi$, and $g$ has a perfect obstruction theory $\psi$. Then
    \[
    f_\phi^{!}g_\psi^{!}(\alpha) = g_\psi^{!}f_\phi^{!}(\alpha)
    \]
    for all $\alpha \in A_{\ast}(Z)$.
\end{lemma}

Finally, suppose $X$ and $Y$ are Deligne-Mumford stacks and $B$ is an algebraic stack and we have a sequence of morphisms
$X \xrightarrow{f} Y \xrightarrow{g} B$.
If we have perfect obstruction theories for $f$, $g$, and $g \circ f$ that are ``compatible'' then we expect $f^! \circ g^! = (g \circ f)^!$. The correct notion of compatibility is as follows.

\begin{definition}[{\cite[Definition 4.5]{manolache-pullback}}]
    Let $X \xrightarrow{f} Y \xrightarrow{g} B$
    be as above,
    and let $\phi_{X/Y}, \phi_{X/B}, \phi_{Y/B}$
    be perfect obstruction theories for $f, g\circ f,$ and $g$ respectively.
    They form a \emph{compatible triple}
    if there exists a morphism of distinguished triangles
    \[
    \begin{tikzcd}
        f^{\ast}\EE_{Y/B} \arrow[r]\arrow[d, "\phi_{Y/B}"] & \EE_{X/B} \arrow[r] \arrow[d, "\phi_{X/B}"] & \EE_{X/Y} \arrow[r, "+1"]\arrow[d, "\phi_{X/Y}"] & {}\\
        Lf^{\ast}\LL_{Y/B} \arrow[r] & \LL_{X/B} \arrow[r] & \LL_{X/Y}\arrow[r, "+1"] & .{}
    \end{tikzcd}
    \]
\end{definition}

\begin{remark}
    When the obstruction theories arise from
    descriptions as vanishing loci of sections,
    the compatibility condition above translates to the
    condition in Section \ref{sec:compatible-OTs}.
\end{remark}

\begin{theorem}[Functoriality, {\cite[Theorem 4.8]{manolache-pullback}}]\label{thm:compat-triple}
    Let $X \xrightarrow{f} Y \xrightarrow{g} B$ be morphisms of algebraic stacks such that $X$ and $Y$ are Deligne-Mumford.
    Let $(\phi_{X/Y}, \phi_{X/B}, \phi_{Y/B})$
    be a compatible triple of obstruction theories.
    Then, for any $\alpha \in A_{\ast}(B)$,
    we have
    \[
    f^!_\phi g_\phi^!(\alpha) = (g \circ f)^!_\phi(\alpha).
    \]
\end{theorem}

\begin{corollary}[{\cite[Corollary 4.9]{manolache-pullback}}]\label{cor:functoriality}
Let $X \xrightarrow{f} Y \xrightarrow{g} B$ be morphisms of algebraic stacks such that $X$ and $Y$ are Deligne-Mumford and $B$ has pure dimension. Let $(\phi_{X/Y}, \phi_{X/B}, \phi_{Y/B})$ be a compatible triple of obstruction theories. Then
\[
f^!_\phi [Y]^{vir}_\phi = [X]^{vir}_\phi.
\]
\end{corollary}

\subsection{Induced obstruction theories}
\label{rmk:BF3}
We used Lemma \ref{lem:BF1} to motivate the definition of a Behrend-Fantechi perfect obstruction theory. However, now we have that definition, we can recognize Lemma \ref{lem:BF1} as a special case of the following more general statement (see \cite[Prop 3]{KKP}):

Assume we have a diagram of $k$-schemes
\[
\begin{tikzcd}
X \arrow[r, "i"] \arrow[dr, "j"'] & Y \arrow[d, "f"]\\
 & Y'
\end{tikzcd}
\]
where $i$ and $j$ are morphisms and $f$ is smooth. Let $\phi: \EE_{X/Y} \to \LL_{X/Y}$ be a perfect obstruction theory. If $h$ is the composition $\EE_{X/Y} \to \LL_{X/Y} \to i^*\LL_{Y/Y'}[1]$ and $\EE_{X/Y'}$ is a complex fitting into a distinguished triangle $\EE_{X/Y'} \to \EE_{X/Y} \xrightarrow{h} i^*\LL_{Y/Y'}[1]$, then the induced map $\phi': \EE_{X/Y'} \to \LL_{X/Y'}$ is a perfect obstruction theory with the property that $[X]^{vir}_{\phi} = [X]^{vir}_{\phi'}.$\\

\noindent
\textbf{Exercises}
\begin{enumerate}[label=6.\arabic*, leftmargin=*]

\item \label{POT8} Prove Lemma \ref{lem:BF8} as follows.
\begin{enumerate}
\item[(a)] Show that the abelian cone $E$ in the lemma statement is indeed a vector bundle. (This can be done even without assuming $\EE_{X/B}$ has a global resolution.)

\item[(b)] Finish the proof of the lemma. (\textit{Hint: \cite[Prop 5.6]{BF}}.)

\end{enumerate}

\item \label{POT14}
\begin{itemize}
\item[(a)] Consider a Cartesian diagram as in Remark \ref{rmk:POT2.5}. Show that the canonical map $Lp^*\LL_{X/B} \to \LL_{X'/B'}$ is an obstruction theory (not necessarily perfect). (\textit{Hint: Proposition \ref{prop:criterion}.)}
\item[(b)] Show that the composition $Lp^*\EE_{X/B} \to Lp^*\LL_{X/B} \to \LL_{X'/B'}$ in Remark \ref{rmk:POT2.5} is a perfect obstruction theory.
\end{itemize}
\item \label{POT15} Prove Corollary \ref{cor:functoriality}.

\item \label{POT7} Prove the more general statement of Lemma \ref{lem:BF1} given in Section \ref{rmk:BF3}. (\textit{Hint: first construct the morphism of distinguished triangles below.})
\[
\begin{tikzcd}
\EE_{X/Y'} \arrow[r] \arrow[d] & \EE_{X/Y} \arrow[r, "h"] \arrow[d] & i^*\LL_{Y/Y'}[1] \arrow[r] \arrow[d, equal]& {}\\
\LL_{X/Y'} \arrow[r] & \LL_{X/Y} \arrow[r] & i^*\LL_{Y/Y'}[1] \arrow[r] & {}
\end{tikzcd}
\]

\end{enumerate}

\section{Virtual class for moduli of maps: definition}
In this section we assume the base field $k$ is the complex numbers.

\subsection{Moduli stacks}

We define various moduli stacks relevant to Gromov-Witten theory. Let $X$ be a scheme.
\begin{definition}
The \textit{moduli stack of $n$-marked genus-$g$ prestable curves} is the algebraic stack $\mf M_{g, n}$ whose fiber over a scheme $S$ is the groupoid of tuples
\[(\pi: C \to S, \{s_i: S \to C\}_{i=1}^n),\] where $C$ is an algebraic space, the $s_i$ are disjoint sections of $\pi$, and $\pi$ is flat, proper, locally finitely presented, of relative dimension 1, and has geometric fibers equal to reduced connected nodal curves of genus $g$ with the images of the $s_i$ contained in the smooth locus. An isomorphism in this groupoid is an isomorphism of the algebraic spaces $C$ over $S$ commuting with the sections.
\end{definition}

\begin{definition}The \textit{moduli stack of prestable maps} to $X$ is the stack $\mc M_{g, n}(X)$ whose fiber over a scheme $S$ is the groupoid of tuples
\[(C \to S, \{s_i: S \to C\}_{i=1}^n, f: C \to X),\] where $(C \to S, \{s_i\}_{i=1}^n)$ is a family of $n$-marked genus-$g$ prestable curves and $f$ is a morphism. An isomorphism in this groupoid is an isomorphism of $n$-marked genus-$g$ prestable curves commuting with the maps to $X$.
\end{definition}

The stack $\mc M_{g, n}(X)$ has a morphism to $\fM_{g, n}$ sending $(C \to S, \{s_i\}_{i=1}^n, f: C \to X)$ to $(C \to S, \{s_i\}_{i=1}^n)$.

Certain open substacks of $\mc M_{g, n}(X)$ are particularly important.
\begin{itemize}
\item For any curve class $\beta \in H_2(X, \ZZ)$ there is an open and closed substack $\mc M_{g, n}(X, \beta)$ of $\mc M_{g, n}(X)$ whose geometric points are prestable maps $(C\to \Spec(k) \{s_i\}_{i=1}^n, f: C \to X)$ such that $f_*[C] = \beta$ (here $[C]$ refers to the fundamental class in homology and $f_*$ is pushforward in homology, not Chow).
\item There is an open substack $\overline{\mc M}_{g, n}(X, \beta) \subseteq \mc M_{g, n}(X, \beta)$ of \emph{stable} maps: an object of $\mc M_{g, n}(X, \beta)(S)$ is stable if geometric fibers over $S$ have finitely many automorphisms (as prestable maps).
\end{itemize}

\subsection{Obstruction theory}
The morphism $\mc M_{g, n}(X) \to \mf M_{g, n}$ has a canonical obstruction theory constructed as follows.
There is a universal curve $\mls C \to \mf M_{g, n}$ and its pullback to $\mc M_{g, n}(X)$ fits into a diagram
\[
\begin{tikzcd}
\mls C_{\mc M_{g, n}(X)} \arrow[d, "\pi"] \arrow[r, "f"] & X\\
\mc M_{g, n}(X).
\end{tikzcd}
\]
The morphism $f$ is called the universal map to $X$. There are canonical maps of cotangent complexes
\begin{equation}\label{eq:maps2}
f^*\LL_X \to \LL_{\mls C_{\mc M_{g, n}(X)}/\mls C} \xleftarrow{\sim} \pi^* \LL_{\mc M_{g, n}(X)/\mf M_{g, n}}
\end{equation}
hence we have a morphism of complexes $f^*\LL_X \to \pi^* \LL_{\mc M_{g, n}(X)/\mf M_{g, n}}$. The map $\pi^*$ has a left adjoint given by $R\pi_*(- \otimes \omega_{\pi}[1])$, where $\omega_{\pi}$ is the dualizing sheaf for the family of curves (see \cite[Tags 0E6Q, 0E6R]{stacks-project})). Applying adjunction to $f^*\LL_X \to \pi^* \LL_{\mc M_{g, n}(X)/\mf M_{g, n}}$ we get a morphism
\[
\phi: R\pi_*(f^*\LL_X \otimes \omega_{\pi}[1]) \to \LL_{\mc M_{g, n}(X)/\mf M_{g, n}}.
\]
It is a nontrivial result that this morphism is an obstruction theory, known as the canonical obstruction theory for $\mc M_{g, n}(X) \to \mf M_{g, n}$ (see e.g. \cite[Section 6]{BF} or \cite{webb}).

\begin{lemma}\label{lem:maps1}
The obstruction theory on $\mc M_{g, n}(X)$ is perfect if $X$ is smooth. In this case the obstruction theory is a morphism
\begin{equation}\label{eq:pot}
\phi: R\pi_*(f^*\Omega_X \otimes \omega_{\pi}[1]) \to \LL_{\mc M_{g, n}(X)/\mf M_{g, n}}.
\end{equation}

\end{lemma}

\subsection{Virtual class}\label{sec:vc}
By Exercise \ref{BF4} we can restrict the perfect obstruction theory $\phi$ to any of the open substacks $\mc M_{g, n}(X, \beta) \subseteq \mc M_{g, n}(X)$ to get a perfect obstruction theory for $\mc M_{g, n}(X, \beta)$. Similarly, we get a perfect obstruction theory for any $\overline{\mc M}_{g, n}(X, \beta)$.
From this perfect obstruction theory and Construction \ref{const3} we get a virtual class
\[
[\overline{\mc M}_{g, n}(X, \beta)]^{vir}_\phi = (p^*)_{\mf E_\phi}^{-1} [\mf C_{\overline{\mc M}_{g, n}(X, \beta)/\mf M_{g, n}}] \in A_{*}(\overline{\mc M}_{g, n}(X, \beta))
\]
where $\mf E_{\phi} = st(R\pi_*(f^*\LL_X \otimes \omega_\pi[1]))$ and $p_{\mf E_\phi}: \mf E_{\phi} \to \overline{\mc M}_{g, n}(X, \beta)$ is the structure morphism.
 The class $[\overline{\mc M}_{g, n}(X, \beta)]^{vir}_\phi$ has pure dimension.
\begin{lemma}\label{lem:maps6}
The virtual class $[\overline{\mc M}_{g, n}(X,\beta)]^{vir}_\phi$ is an element of $A_d(\overline{\mc M}_{g, n}(X, \beta))$, where
\[
d = (1-g)(\dim X - 3) - K_X \cdot \beta + n
\]
and $K_X = c_1(\Lambda^{\dim X}\Omega_X)$. The integer $d$ is called the \emph{expected dimension} of $\overline{\mc M}_{g, n}(X,\beta)$.
\end{lemma}

\noindent
\textbf{Exercises}
\begin{enumerate}[label=7.\arabic*, leftmargin=*]
\item \label{maps1} Explain where the morphisms in \eqref{eq:maps2} come from and why one of them is an isomorphism. Use the properties of the cotangent complex \ref{eq:C1}-\ref{eq:C4}.

\item \label{maps2} Prove Lemma \ref{lem:maps1}. (\textit{Hint: \cite[Tag 08EV]{stacks-project} and \cite[Prop 8.3.6.4]{fgaexplained}}.)

\item \label{maps3} Prove Lemma \ref{lem:maps6} as follows. We continue with the notation in \S \ref{sec:vc}.
\begin{itemize}
\item[(a)] If $f_K: C \to X$ is any stable map defined over an algebraically closed field $K$, show that the rank of $R\pi_*(f^*\Omega_X \otimes \omega_\pi[1])$ is $\chi(f_K^*\Omega_X^\vee)$.
Hence $(p^*)^{-1}_{\mf E_\phi}$ is a morphism
\[
(p^*)^{-1}_{\mf E_\phi}: A_{*}({\mf E_\phi}) \to A_{*+\chi(f_K^*\Omega_X^\vee)}(\overline{\mc M}_{g, n}(X, \beta)).
\]
\item[(b)] Use Riemann-Roch and \cite[Remark 3.2.3(c)]{fulton} to show
\[
\chi(f_K^*\Omega_X^\vee) = \dim X(1-g) - K_X \cdot \beta.
\]
\item[(c)] Use parts (a), (b), and Exercise \ref{POT0} to prove Lemma \ref{lem:maps6}. You may use that $\mf M_{g, n}$ is smooth of dimension $3g-3 + n$ (see \cite[Proposition~2]{BehrendGW}).
\end{itemize}

\end{enumerate}

\section{Virtual class for moduli of stable maps: examples}
In this section we continue to assume the base field $k$ is the complex numbers.

The virtual class on $\overline{\mc M}_{g, n}(X, \beta)$ can be calculated in some examples. We consider three. Let $\phi: \EE \to \LL_{\overline{\mc M}_{g, n}(X, \beta)/\mf{M}_{g, n}}$ denote the canonical perfect obstruction theory.

First, if $h^{-1}(\EE) = 0$ and $h^0(\EE)$ is locally free, then by Lemma \ref{lem:BF9} we have that $\overline{\mc M}_{g, n}(X, \beta)$ is smooth and
\[
[\overline{\mc M}_{g, n}(X, \beta)]^{vir}_\phi = [\overline{\mc M}_{g, n}(X, \beta)].
\]

Motivated by this, $X$ is called \textit{convex} if $H^1(C, f^*\Omega_X^\vee)=0$ for every genus-zero stable map $f: C \to X$.

\begin{lemma}\label{lem:maps3}For each of the following moduli spaces of stable maps, we have that $h^{-1}(\EE)=0$ and $h^0(\EE)$ is locally free:
\begin{itemize}
\item $\overline{\mc M}_{0, n}(X, 0)$,
\item $\overline{\mc M}_{g, n}(\Spec(k), 0)$
\item $\overline{\mc M}_{0, n}(X, \beta)$ for $X$ convex.
\end{itemize}
\end{lemma}

\noindent
The following lemma implies that projective spaces, Grassmannians, and more generally varieties of the form $G/P$ are convex.

\begin{lemma}\label{lem:maps4}If the tangent bundle of $X$ is globally generated, then $X$ is convex.
\end{lemma}

Second, if $\overline{\mc M}_{g, n}(X, \beta)$ is smooth we can compute its virtual class using Lemma \ref{lem:BF8}. Since $\overline{\mc M}_{g, n}(X, 0) \simeq \overline{\mc M}_{g, n} \times X$, where $\overline{\mc M}_{g, n}$ is the moduli stack of $n$-marked genus-$g$ stable curves, we can apply Lemma \ref{lem:BF8},
noting that $\EE$ has a global presentation by
\cite[Lemma~5.0.3]{CJW}.

\begin{lemma}\label{lem:maps5}
The stack $\overline{\mc M}_{g, n}(X, 0) $ is smooth and its virtual cycle is given by
\begin{equation}\label{eq:maps0}
[\overline{\mc M}_{g, n}(X, 0)]^{vir}_{\phi} = c_{g \dim X}(E) \cap [\overline{\mc M}_{g, n}(X, 0)]
\end{equation}
where $E \to \overline{\mc M}_{g, n} \times X$ is the vector bundle whose sheaf of sections is $(\pi_*\omega_\pi)^\vee \boxtimes \Omega_X^\vee$.
\end{lemma}

The third and final example is a \textit{genus-0 quantum Lefshetz formula.} Let $X$ be a hypersurface of degree $\ell$ in $\PP^r$. Recall that $H_2(\PP^r, \ZZ)$ is isomorphic to $\ZZ$, generated by the class of a line $L$.
For an integer $d$, it is customary to write $\overline{\mc M}_{g, n}(\PP^r, d)$ for $\overline{\mc M}_{g, n}(\PP^r, d[L]).$
 Since $\PP^r$ is convex, the moduli stack $\overline{\mc M}_{0, n}(\PP^r, d)$ is smooth and its virtual class equals the usual fundamental class. The inclusion $X \to \PP^r$ induces a closed embedding $j: \overline{\mc M}_{0, n}(X, \beta) \to \overline{\mc M}_{0, n}(\PP^r, d)$ where $d[L] = i_*\beta$.\footnote{There is an isomorphism $H_2(\PP^r, \ZZ) \simeq \ZZ$ such that class of a stable map to $X$ is the integer $d$ if the induced map to $\PP^r$ has degree $d$.} For the following result we refer to \cite{KKP}, or to \cite[Thm 8.11]{BCM} for a summary.

\begin{lemma}\label{lem:maps66}
Let $d$ be a positive integer.
Then the coherent sheaf
$\mls E := \pi_* f^* \mls O_{\PP^r}(\ell)$ on $\mc M_{0, n}(\PP^r, d)$ is locally free of rank $\ell d + 1$.
Moreover, if $E := \Spec(Sym(\mls E^\vee))$
and $\overline{\mc{M}}_{0, n}(X, d) :=
\bigsqcup_{i_{\ast}\beta = d[L]} \overline{\mc{M}}_{0,n}(X, \beta)$,
then
\begin{equation}\label{eq:maps-1}
j_*[\overline{\mc M}_{0, n}(X, d)]^{vir}_{\phi} = c_{\ell d + 1}(E) \cap [\overline{\mc M}_{0, n}(\PP^r, d)].
\end{equation}
\end{lemma}

\noindent
\textbf{Exercises}
\begin{enumerate}[label=8.\arabic*, leftmargin=*]
\item \label{GW1}
 Let $X$ be a smooth scheme.
 \begin{itemize}
\item[(a)] Let $\phi: \EE \to \LL_{\mc M_{g, n}(X)}$ denote the canonical perfect obstruction theory. Show that if $h^{-1}(\EE)=0$ then $h^0(\EE)$ is locally free.
\item[(b)] Prove Lemma \ref{lem:maps3}.
\end{itemize}

\item \label{GW2} Prove Lemma \ref{lem:maps4}.

\item \label{GW3} This problem is about Lemma \ref{lem:maps5}.
\begin{itemize}
\item[(a)] Show the dimensions of the Chow classes on either side of \eqref{eq:maps0} agree.

\item[(b)] Let $X$ be a scheme (you may assume $X$ is smooth if you like, but for this part it is not necessary). Construct an isomorphism $\overline{\mc M}_{g, n}(X, 0) \simeq \overline{\mc M}_{g, n} \times X.$

\item[(c)] Prove Lemma \ref{lem:maps5}. (\textit{Hints: Lemmas \ref{lem:BF8} and \ref{lem:two-term}}.)
\end{itemize}

\item \label{GW4} This problem is about Lemma \ref{lem:maps66}. For simplicity we write $d$ instead of $d[L]$.
\begin{itemize}
\item[(a)] Show that the dimensions of the Chow classes on either side of \eqref{eq:maps-1} agree.
\item[(b)] Show that $\mls E:= \pi_*f^*\mls O_{\PP^r}(\ell)$ is a locally free sheaf on $\overline{\mc M}_{0, n}(\PP^r, d)$ of rank $\ell d+1$.
\item[(c)] Let $s \in \Gamma(\PP^r, \mls O(\ell))$ be the section defining $X$. Show that this section induces a section $\sigma$ of $\mls E$ such that $\mathbb{V}(\sigma)$ is precisely $\overline{\mc M}_{0, n}(X, d)$.
\end{itemize}
We will complete the proof of Lemma \ref{lem:maps66} in Exercise \ref{siebert101}.

\item \label{GW5} Let $X$ be a K3 surface (this means $X$ is a complete nonsingular variety over $k$ with $\Lambda^2 \Omega_X \simeq \mls O_X$ and $H^1(X, \mls O_X)=0$).
It is a classical fact that Gromov--Witten invariants
of $X$ vanish
when the curve class $\beta$ is non-zero.
Here, we will see what happens when $\beta = 0$.
Show that $[\overline{\mc M}_{g, n}(X, 0)]^{\mathrm{vir}} = 0$
when $g \geq 2$.
What is the virtual cycle if $g = 0, 1$?
\textit{(Hint: Apply Mumford's relations for
$\lambda$-classes.)}

\end{enumerate}

\section{Siebert's formula for virtual classes}\label{sec:siebert}

In this section we once again allow the base field $k$ to be arbitrary.

Let $\phi: \EE_{X/B} \to \LL_{X/B}$ be a perfect obstruction theory on a Deligne-Mumford stack $X$. In this section we present Siebert's formula \cite[Theorem 4.6]{siebert} for $[X]^{vir}_\phi$, which assumes the existence of a global resolution of $\EE_{X/B}$ (Section \ref{sec:globalres}). Critically, this formula will \textit{not} depend on the morphism $\phi$, only on $X$ and the perfect complex $\EE_{X/B}.$ This independence is a crucial tool for proving functoriality properties of the virtual class for moduli of stable maps (Section \ref{sec:applications}).

Just as vector bundles have Chern classes, cones have Segre classes. If a pure dimensional cone $C$ is contained in a vector bundle $E$, the example \cite[Example 4.1.8]{fulton} gives a formula for the intersection of $C$ with the zero section of $E$ in terms of Segre classes of $C$ and Chern classes  of $E$. We used this formula in Exercise \ref{vc6} to give a formula for the localized top Chern class, a basic example of a virtual class. Siebert's formula is a generalization of that exercise.

\subsection{Fulton's class}

Let $X$ be a $k$-scheme with a global embedding into a scheme smooth over $k$. In \cite[Example 4.2.6]{fulton} Fulton defines a canonical class in $A_*(X)$ using such an embedding, but independent of the choice of embedding. Siebert calls this \textit{Fulton's canonical class}, denoted $c_F(X)$. The construction generalizes immediately to Deligne-Mumford stacks and the relative setting.

Let $X$ be a Deligne-Mumford stack, $B$ an algebraic stack, and fix a morphism $X \to B$. Assume there is an algebraic stack $M$ with $M \to B$ smooth and a closed immersion $i: X \to M$ over $B$. We define
\begin{equation}\label{eq:fultonclass}
c^{X \to M}_F(X/B) = c(i^*T_{M/B}) \cap s(C_{X/M}),
\end{equation}
where $c(-)$ is the total Chern class and $s(-)$ is the  total Segre class. Note that the normal cone $C_{X/M} \to X$ is representable by schemes, so its Segre class is defined. We will see in a moment that
the class $c^{X \to M}_F(X/B)$ depends only on $X \to B$.
Hence we  write
\[c_F(X/B) := c^{X \to M}_F(X/B).\]

To understand why $c_F(X/B)$ is independent of the choice of embedding, let us consider its definition \eqref{eq:fultonclass} more closely: motivated by \cite[Example 4.1.6]{fulton}, it seems that the right hand side should be the meaning of the Segre class $s(\mathfrak{C}_{X/B})$ of the intrinsic normal cone. Segre classes of cone stacks are not, to our knowledge, defined in general, but we can state the following definition and lemma.

\begin{definition}\label{def:conestack}
Let $\mf C$ be a cone stack on $X$. A \emph{global presentation} of $\mf C$ is an isomorphism $\mf C \simeq [C/T]$ where $C \to X$ is a cone, $T \to X$ is a vector bundle, and the action of $T$ on $C$ is induced by a morphism of cones $T \to C$ (Section \ref{sec:conestacks}). Then the \emph{total Segre class} of $\mf C$ is defined as
\[
s(\mf C) := c(T) \cap s(C).
\]
\end{definition}
The following lemma is proved in the exercises.
\begin{lemma}\label{lem:independent}
The Segre class in Definition \ref{def:conestack} is independent of the global presentation; that is, if $[C/T] = [C'/T']$ with $T \to C$ and $T' \to C'$ having the properties described in Definition \ref{def:conestack}, then
\[
c(T) \cap s(C) = c(T') \cap s(C').
\]
\end{lemma}
\noindent
In particular this lemma implies our claim that $c_F^{X \to M}(X/B)$ is independent of the embedding $X \to M$. Because of this lemma, we say \textit{the Fulton class of $X \to B$ is defined} if the relative intrinsic normal cone $\mf C_{X/B}$ has a global presentation as the quotient of a cone on $X$ by a vector bundle on $X$. In this situation we define
\[c_F(X/B) := s(\mf C_{X/B}).\]

The requirement for $X$ to have a global immersion in a smooth stack over $B$ is not always satisfied \cite{example}, although it is satisfied for stable maps spaces to smooth projective varieties \cite{AGOT}. On the other hand, there is at least one other way to get a global presentation for $\mf C_{X/B}$: such a presentation arises from a perfect obstruction theory for $X \to B$ with a global resolution, as we now explain.

Let $X$ be a Deligne-Mumford stack, $B$ an algebraic stack, and fix a morphism $f:X \to B$. Assume  now that $\phi: \EE_{X/B} \to \LL_{X/B}$ is a perfect obstruction theory with a global resolution. Recall from Section \ref{sec:globalres} that this means we have a quasi-isomorphism $\EE_{X/B} \to [\mls E^{-1} \to \mls E^0]$ with $\mls E^i$ a finite rank locally free sheaf on $X$. Writing $E^i = \relSpec_X(Sym(\mls E^{-i}))$, we have a fiber diagram
\[
\begin{tikzcd}
C(E) \arrow[d] \arrow[r] & E^1 \arrow[d]  \\
\mathfrak{C}_{X/B} \arrow[r] & st(\EE_{X/B})
\end{tikzcd}
\]
where vertical maps are quotients by $E^0$.
Since the right vertical arrow is a torsor for the $X$-group scheme $E^0$, so is the left vertical arrow; that is,
\begin{equation}\label{eq:nuther}
\mf C_{X/B} = [C(E)/E^0],
\end{equation}
and in fact this is a presentation for $\mf C_{X/B}$ in the sense of Definition \ref{def:conestack} (see Exercise \ref{siebert1.5}).
The following result is new (cf. \cite[Remark 4.5]{siebert}).

\begin{corollary}\label{cor:siebert}
Let $X$ be a Deligne-Mumford stack, $B$ an algebraic stack, and fix a morphism $X \to B$. Assume that $\phi: \EE_{X/B} \to \LL_{X/B}$ is a perfect obstruction theory with a global resolution. Then the Fulton class of $X \to B$ is defined and equals $c(E^0) \cap s(C(E)).$
\end{corollary}

We note that by Lemma \ref{lem:independent}, in situations where $X$ has both a global embedding into a stack smooth over $B$ and a perfect obstruction theory with a global resolution, either collection of data can be used to compute the Fulton class of $X \to B$.

\subsection{Siebert's formula and applications}\label{sec:applications}

Recall from Definition \ref{def:rank} the notion of the rank of a perfect obstruction theory. If $\EE = [\mls E^{-1} \to \mls E^0]$ is a 2-term complex of finite rank locally free sheaves, then writing $E^i = \relSpec_X(Sym(\mls E^i))$ we define the total Chern class of its dual
\[
c(\EE^{\vee}) := c(E^0) \cup c(E^1)^{-1}.
\]
We are now ready to state Siebert's formula. Note that our Corollary \ref{cor:siebert} allows us to remove Siebert's hypothesis of a global embedding in a smooth stack over $B$.

\begin{remark}
Note that the above definition of $c(\EE^\vee)$ is compatible with the chern class of a finite rank locally free sheaf as defined in \cite[Tag 0FE1]{stacks-project}. That is,
if $\sE$ is a finite rank locally free sheaf and $\EE = \sE[0]$,
then $c(\EE^{\vee}) = c(\sE^{\vee})$ agrees with $c(E)$,
where $E = \relSpec_{X}(\Sym(\sE))$.
\end{remark}

\begin{theorem}[Siebert's formula, {\cite[Theorem~4.6]{siebert}}]\label{thm:siebert}
Let $X$ be a Deligne-Mumford stack, $B$ an algebraic stack of pure dimension, and fix a morphism $X \to B$. Assume  that $\phi: \EE_{X/B} \to \LL_{X/B}$ is a perfect obstruction theory with a global resolution $\EE_{X/B} \simeq [\mls E^{-1} \to \mls E^0]$. Then
\[
[X]^{vir}_\phi = \{ \; c(\EE_{X/B}^{\vee})^{-1} \cap c_F(X/B)\; \}_{\dim(B) + \mathrm{rank}(\EE_{X/B})},
\]
where $\{-\}_{d}$ means to take the dimension-$d$ part of the expression in $\{-\}$.
\end{theorem}

The following result shows that this formula is applicable to all stable maps moduli spaces, as well as many generalizations that arise naturally in enumerative geometry.

\begin{lemma}[{\cite[Lem 5.0.3]{CJW}}]
    \label{lem:two-term}
    Let $M$ be a finite type Deligne-Mumford stack and let $\pi: C \to M$ be a family of prestable curves. If there exists a $\pi$-relatively ample line bundle on $C$, then for any vector bundle $T$ on $C$, the complex $R\pi_*T$ is globally isomorphic to a 2-term complex of vector bundles.
\end{lemma}

\noindent
Note that in \cite{CJW} the lemma is in fact stated for families  of \textit{twisted} curves.

While explicitly computing the quantities in Siebert's formula requires more knowledge of the Chow ring of $X$ than is typically available, Siebert's formula has an important corollary: the virtual class associated to an obstruction theory $\phi: \EE_{X/B} \to \LL_{X/B}$ with a global resolution depends only on $\EE_{X/B}$ and the morphism $X \to B$, not on the morphism $\phi$. Even more useful is the generalization of this result to virtual pullbacks. While there is a generalization of Theorem \ref{thm:siebert} for computing $f^!_\phi$ (see \cite{LW2}), it is slightly complicated to state, due to the complicated definition of cycles in $A_*(B)$ when $B$ is algebraic. Here we content ourselves with the following.

\begin{corollary}[{\cite{LW2}}]\label{cor:siebert2}
Let $X$ be a Deligne-Mumford stack, $B$ an algebraic stack, and fix a morphism $X \to B$. Assume  that $\phi: \EE_{X/B} \to \LL_{X/B}$ is a perfect obstruction theory with a global resolution.  Then $f^!_\phi$ depends only on $\EE_{X/B}$ and the morphism $X \to B$, and not on $\phi$.
\end{corollary}

As an example of the usefulness of Corollary \ref{cor:siebert2} we present the following simplification of Corollary \ref{cor:functoriality}.

\begin{definition}
    Let $X \xrightarrow{f} Y \xrightarrow{g} B$
    be morphisms of algebraic stacks such that $X$ and $Y$ are Deligne-Mumford and $B$ has pure dimension,
    and let $\phi_{X/Y}, \phi_{X/B}, \phi_{Y/B}$
    be perfect obstruction theories for $f, g\circ f,$ and $g$ respectively.
    They form a \emph{weakly compatible triple}
    if
    \begin{itemize}
    \item[(i)] There is a distinguished triangle
    \[
    f^*\EE_{Y/B} \to \EE_{X/B} \to \EE_{X/Y} \xrightarrow{+1},
    \]
    \item[(ii)] \underline{one} (at least) of the three squares in the diagram
    \[
    \begin{tikzcd}
        f^{\ast}\EE_{Y/B} \arrow[r]\arrow[d, "\phi_{Y/B}"] & \EE_{X/B} \arrow[r] \arrow[d, "\phi_{X/B}"] & \EE_{X/Y} \arrow[r, "+1"]\arrow[d, "\phi_{X/Y}"] & {}\\
        Lf^{\ast}\LL_{Y/B} \arrow[r] & \LL_{X/B} \arrow[r] & \LL_{X/Y}\arrow[r, "+1"] & .{}
    \end{tikzcd}
    \]
    commutes, and
    \item[(iii)] the obstruction theory not contained in the commuting square in (ii) has a global resolution.
    \end{itemize}
\end{definition}

\begin{corollary}\label{cor:functoriality2}
    Let $X \xrightarrow{f} Y \xrightarrow{g} B$ be morphisms of algebraic stacks such that $X$ and $Y$ are Deligne-Mumford and $B$ has pure dimension.
    Let $\phi_{X/Y}, \phi_{X/B},$ and $ \phi_{Y/B}$ be a weakly compatible triple.
    Then
    \[
    f^!_\phi [Y]^{vir}_{\phi_{Y/B}} = [X]^{vir}_{\phi_{X/B}}.
    \]
\end{corollary}
\noindent
Corollary \ref{cor:functoriality2} is used in \cite{LW2} to prove a certain functoriality property of the Gromov-Witten virtual class. In particular, this property is part of the statement that Gromov-Witten invariants form a \textit{cohomological field theory} (see e.g. \cite{pand} for an introduction).

\noindent
\textbf{Exercises}
\begin{enumerate}[label=9.\arabic*, leftmargin=*]

\item \label{siebert-1} Let $X$ be a Deligne-Mumford stack, let $C, D,$ and $E$ be cones over $X$, and let $C \to D$ and $E \to D$ be morphisms of cones.
\begin{itemize}
\item[(a)]
Show that the fiber product $C \times_D E$ is also a cone over $X$.
\item[(b)] If $F$ is another cone over $X$ and $F \to C$ and $F \to E$ are morphisms of ones such that the compositions $F \to C \to D$ and $F \to E \to D$ agree, then the induced morphism $F \to C \times_D E$ is a morphism of cones.
\end{itemize}

\item \label{siebert0} The following result is a lemma for solving the later exercises. Let $[C/T]$ and $[D/U]$ be global presentations for two cone stacks over a base $S$ in the sense of Definition \ref{def:conestack}, let $[D/U] \to [C/T]$ be a morphism of cone stacks induced by a commuting square of morphisms of cones
\[
\begin{tikzcd}
U \arrow[r] \arrow[d] & T \arrow[d] \\
D \arrow[r, "f"] & C,
\end{tikzcd}
\]
and define $C'$ to be the fiber product
\[
C' = C \times_{[C/T]} [D/U].
\]
Assume that either (a) $U$ is trivial, or (b) $U \to T$ is an isomorphism. In either case, prove that $C'$ is a cone over $S$ and that there is an induced morphism of cones $T \to C'$ with a canonical presentation $D \simeq [C'/T]$.

\item \label{siebert1} Prove Lemma \ref{lem:independent} by following these steps (cf. \cite[Proposition 1.4]{BF}).
\begin{enumerate}
\item Let $\fX$ be an algebraic stack over a base stack $S$ with two global quotient presentations:
\[
\fX = [X/G] = [Y/H],
\]
where $X$ and $Y$ are $S$-schemes and $G$ and $H$ are $S$-group schemes. Show that there is an $S$-scheme $Z$ such that $X = [Z/H]$ and $Y = [Z/G]$. (\textit{Hint: show there is a fiber diagram}
\[
\begin{tikzcd}
Z \arrow[r] \arrow[d] & X \arrow[r] \arrow[d] & \fX \arrow[d, "\delta"] \\
X \times_S Y \arrow[r] & {[(X \times_S Y)/H]} \arrow[r] \arrow[d] & \fX \times_S \fX \arrow[d, "pr_1"]\\
& X \arrow[r] & \fX
\end{tikzcd}
\]
\textit{where $\delta$ is the diagonal and $pr_1$ is projection to the first factor.})
\item Use Exercise \ref{siebert0} to show that if $\fX$ is a cone stack and $[X/G]$, $[Y/H]$ are two presentations, then the identifications $X = [Z/H]$ and $Y = [Z/G]$ in part (a) are also presentations in the sense of Definition \ref{def:conestack}.

\item Conclude from (b) that
\[
0 \to H \to Z \to X \to 0
\]
is a short exact sequence of cones in the sense of \cite[Example 4.1.6]{fulton}
(possibly by using \cite[Lemma 1.3]{BF}), and hence by \cite[Example 4.1.6(c)]{fulton} there is an equality
\[
s(X) = c(H) \cap s(Z).
\]

\item Conclude the proof of Lemma \ref{lem:independent}.

\end{enumerate}

\item \label{siebert1.5} Let $X$ be a Deligne-Mumford stack, $B$ an algebraic stack, and fix a morphism $X \to B$. Let $\phi: \EE_{X/B} \to \LL_{X/B}$ is a perfect obstruction theory with a global resolution $\EE_{X/B} = [\mls E^{-1} \to \mls E^0]$. Show that the identification $\mf C_{X/B} = [C(E)/E^0]$ in \eqref{eq:nuther} is a presentation in the sense of Definition \ref{def:conestack}. (\textit{Hint: Use \cite[Remark on p.55]{BF}  to argue that $C(E)$ is a cone.})

\item\label{siebert4} Prove Theorem \ref{thm:siebert}. (\textit{Hints: use the formula for $[X]^{vir}_\phi$ in Section \ref{sec:globalres} and \cite[Example 4.1.8]{fulton}. Show that $C(E)$ is cone of pure dimension $\dim B + \rank(E^0)$.})

\item\label{siebert5} Prove Corollary \ref{cor:functoriality2}. (\textit{Hint: Use Corollary \ref{cor:functoriality}.})

\item\label{siebert101}
This problem is a continuation of Exercise \ref{GW4}.
\begin{enumerate}
\item[(a)]
 Granting parts (b) and (c) of Exercise \ref{GW4}, we have two natural virtual classes on $\overline{\mc M}_{0, n}(X, d)$: there is the canonical virtual class $[\overline{\mc M}_{0, n}(X, d)]^{vir}_\phi$ arising from the perfect obstruction theory \eqref{eq:pot}, but since $\overline{\mc M}_{0, n}(\PP^r, d)$ is smooth we also have a virtual class $[\overline{\mc M}_{0, n}(X, d)]^{vir}_\mls E$ arising as in Construction \ref{const:1}. Show that
 \[
 [\overline{\mc M}_{0, n}(X, d)]^{vir}_\phi = [\overline{\mc M}_{0, n}(X, d)]^{vir}_\mls E.
 \]
 (\textit{Hints: Use Theorem \ref{thm:siebert} to compute the left hand side and Exercise \ref{vc6} to compute the right hand side.})
\item[(b)] Prove Lemma \ref{lem:maps66}. (\textit{Hint: use \eqref{lem:vc2}.})
\end{enumerate}

\end{enumerate}

\newpage

\section{Solution notes}

\noindent
\textbf{Section 1}

\begin{enumerate}[leftmargin=.24in]

\item[\ref{cone3}] This is \cite[Tag 0635]{stacks-project}.

\item[\ref{cone4}] {One can show that $N_{X/Y} = \Spec(k[x])$ and $C_{X/Y} = Y = \Spec(k[x]/(x^2)).$}

\item[\ref{cone5}]
    {One can show that $C_{Y/Z} = N_{Y/Z}$ is a rank 1 vector bundle on $Y$, that $C_{X/Y} = \VV(y^2-x^2)$, and that $N_{X/Y}$ is $\AA^2$.}

\item[\ref{cone6}]
\begin{itemize}
\item[(a)] The reduced subscheme of $X$ is computed by the radical ideal $(x, xy)$ and is equal to the $y$-axis. In the original nonreduced ring, the ideal $(x, y)$ annihilates $x$, showing that $(x, y)$ is an associated prime.
\item[(b)] The displayed map clearly
factors through the coordinate ring
of $N_{X/Y}$.
It is also clearly surjective.
\item[(c)]
Let $I$ be the ideal $(x^2, xy)$ in $k[x,y]$ and let $R = k[x,y]/I$. We claim that the kernel of the graded homomorphism
\[
R[A, B] \to \oplus_n I^n/I^{n+1}
\]
sending $A$ to $x^2$ and $B$ to $xy$ is equal to $(yA - xB)$. Clearly $yA - xB$ is in the kernel.

For the converse, consider an element $f = \sum_{i=0}^n [p_i(x, y)] A^i B^{n-i}$ of $R[A, B]$ that is homogeneous of degree $n\geq 1$ (in degree zero, both $(yA - xB)$ and the kernel of the above map are just $(0)$). We can choose a representative $p_i(x,y)$ of the coset $[p_i(x, y)] \in R$ equal to $a_i + b_i x + c_i y + y^2 q_i(y)$. Then the image of $f$ is
\begin{equation}\label{eq:sol-to-cone}
\sum_i (a_i + b_i x + c_i y + y^2 q_i(y)) x^{2i}(xy)^{n-i} \in I^n
\end{equation}
and $f$ is in the kernel if and only if this quantity is actually in $I^{n+1}$.

Observe that $I^{n+1}$ is a homogeneous ideal and that every term of every element of $I^{n+1}$ has degree at least $2n+2$. Therefore, assuming that $f$ is in the kernel and equating terms of like homogeneous degree in \eqref{eq:sol-to-cone}, we find
\begin{align*}
\sum_{i=0}^n a_i x^{2i}(xy)^{n-i} &= 0 &\text{in degree }2n\\
\sum_{i=0}^n (b_ix + c_iy) x^{2i}(xy)^{n-i} &= 0 &\text{in degree }2n+1\\
\sum_{i=0}^n y^2q_i(y) x^{2i}(xy)^{n-i} & \in I^{n+1} &\text{in degree }\geq 2n+2.
\end{align*}
The first equation implies $a_i=0$ for all $i$, so the corresponding part of $f$ is 0 (and in $(yA-xB)$).

The second equation implies $b_n=c_0=0$ and $b_i + c_{i+1}=0$ for $i=0, \ldots, n-1$. It follows that the corresponding part of $f$ is in $(yA - xB)$:
\[
\sum_{i=0}^n (b_ix + c_iy) A^i B^{n-i} = \sum_{i=1}^n c_i (yA^iB^{n-i} - xA^{i-1}B^{n-i+1}) = \sum_{i=1}^n c_i(yA - xB)A^{i-1}B^{n-i}.
\]

In the third equation, observe that every monomial of an element of $I^{n+1}$ has $x$ appearing with degree at least $n+1$. Since $x$ appears in a monomial of $y^2q_i(y)x^{2i}(xy)^{n-i}$ with degree $n+i$, we see that this third part of $f$ is a sum over $i \geq 1$:
\[
\sum_{i =1}^n y^2q_i(y)A^iB^{n-i}.
\]
  Now,
    for every $i \geq 1$,
    we can show that the corresponding term of $f$ is in $(yA-xB)$:
    \[
        [y^{2}]A^{i}B^{n-i} = [y](yA-xB)A^{i-1}B^{n-i} +
            [xy]A^{i-1}B^{n-i+1}
            = [y](yA-xB)A^{i-1}B^{n-i}.
    \]

Finally, the kernel of $R[A, B] \to Sym(I/I^2)$ is contained in the kernel of $R[A, B] \to \oplus_n I^n/I^{n+1}$ (by the above discussion this equals $(yA - xB)$), but also contains $(yA - xB)$ (by direct computation). It follows that the kernel of $R[A, B] \to Sym(I/I^2)$ must also be $(yA - xB)$, so the canonical inclusion $C_{X/Y} \to N_{X/Y}$ is an isomorphism. Alternatively,
    letting $M= R\cdot A \oplus R\cdot B$,
    we may use
    \cite[Tag~00DO]{stacks-project}
    to conclude that
    $\Sym(M)/(yA-xB) \cong \Sym(M/(yA-xB))$.

\item[(d)]
This is an easy computation using
the presentation found above.

\item[(e)] Every regularly embedded subscheme of $Y$ has to be Cohen--Macaulay, but $X$ is not --- the embedded point is an associated prime of $k[X]$ with height $1$. The unmixedness theorem says that this cannot happen for a CM scheme.
\end{itemize}

\item[\ref{cone7}]

    \begin{itemize}
        \item[(a)] It is easy to see that
            $(xz, yz) = (x, y) \cap (z)$.
        \item[(b)] Both $L$ and $H$ are regularly embedded in $Y$,
            so the first two equalities hold.
            The final equality $N_{X/Y} = C_{X/Y}$
            clearly holds away from the intersection point
            $[0:0:0:1] = L \cap H$,
            so we may restrict to the affine chart $(w = 0)$
            to see what happens near this point.
            Write this affine set as $\Spec(k[x,y,z])$,
            so $X$ is cut out by the ideal $I = (xz, yz)$.
            Let $R = k[x,y,z]/I$.
            One can show that the kernel of the graded homomorphism
            \[
                R[A, B] \to \bigoplus_{n \geq 0} I^{n}/I^{n+1}
                    = \Gamma(C_{X/Y}, \OO_{C_{X/Y}})
            \]
            is generated by $yA-xB$,
            so $\Gamma(C_{X/Y}, \OO_{C_{X/Y}}) \cong
            \Sym((R \cdot A \oplus R \cdot B)/(yA-xB))
            \cong \Sym(I/I^{2})$
            as desired.
    \end{itemize}

\item[\ref{cone8}]
The total space of the deformation to the normal cone can be constructed by blowing up \cite[Sec 5.1]{fulton} or by the graph construction \cite[Remark 5.1.1]{fulton}. Show that these two constructions are the same by proving the following:

Let $E$ be a vector bundle on a scheme $Y$ with a section $s$ and let $X = \VV(s)$. Then the blowup $\operatorname{Bl}_X Y$ is the closure of the image of $Y \setminus X$ under the rational map $s: Y  \dashrightarrow \PP(E)$. (Hint: Use the Blowup Closure Lemma \cite[\S 22.2.7]{vakil}.)
\end{enumerate}

\noindent
\textbf{Section 2}

\begin{enumerate}[leftmargin=.24in]
\item[\ref{vc1}]
\begin{enumerate}
\item[(a)] Write $[Y] = \sum_i m_i [Y_i]$, so the $Y_i$ are the irreducible components of $Y$ of respective multiplicities $m_i$. By definition, we have
\[
\sigma([Y]) = \sum_i m_i [C_{X_i/Y_i}]
\]
where $X_i = X \cap Y_i$. To show that the right hand side equals $[C_{X/Y}]$ it is enough to show that the $C_{X_i/Y_i}$ are the irreducible components of $C_{X/Y}$ with respective multiplicities $m_i$.

To prove this, let $M_XY^\circ$ be the total space of the deformation to the normal cone of $X$ in $Y$. We claim the irreducible components of $M_XY^\circ$ are $M_{X_i}Y_i^\circ$ of respective multiplicities $m_i$. This claim can be checked on any dense open subset of $M_XY^\circ$, for example the complement of any divisor, for example the complement of the zero fiber. But on this complement we have $M_XY^\circ = (\AA^1\setminus \{0\}) \times Y$ and $M_{X_i}Y_i^\circ = (\AA^1 \setminus \{0\}) \times Y_i$ and the claim is manifestly true.

So we know the irreducible components of $M_XY^\circ$ are $M_{X_i}Y_i^\circ$ of respective multiplicities $m_i$. Now apply \cite[Lem 1.7.2]{fulton}---this is where we use the hypothesis that $Y$ has pure dimension. It says the irreducible components of the 0 fiber are $C_{X_i/Y_i}$ with respective multiplicities $m_i$.

\item[(b)] Let $Y = \VV(xz, xy) \subseteq \PP^3_{x:y:z:w}$ and let $X = \VV(xz, xy, x+y+z) \subset Y$. Then $Y$ is the union of a line $L$ and a hyperplane $H$, and $X$ is isomorphic to $\PP^1$ with an embedded point. By definition, we have
\[
\sigma([Y]) = [C_{H \cap X / H}] + [C_{L \cap X / L}]
\]
and we wish to compare this to $[C_{X/Y}]$ in $A_*(C_{X/Y})$.

The embedding $X \to Y$ is regular, so $C_{X/Y} = N_{X/Y}$ is a line bundle on $X$. In particular we have $A_{*+1}(C_{X/Y}) \simeq A_*(X)$. Under this isomorphism, we have
\[
[C_{X/Y}] \mapsto [X], \quad \quad [C_{H \cap X / H}] \mapsto [X], \text{ and }\quad \quad [C_{L \cap X / L}] \mapsto [pt],
\]
so that $[C_{H \cap X / X}] + [C_{L \cap X / L}] \neq [C_{X/Y}]$. The cone computations may be justified as follows. First, $[C_{X/Y}]$ is just the fundamental class of the line bundle on $X$, which retracts to $X$. Second, $H \cap X$ is just the reduced subscheme $\PP^1 \subset X$, and the cone $C_{H \cap X / H}$ is the normal cone (i.e. normal space) to a regular embedding of $\PP^1$ in $\PP^2$. The class of this cone again retracts onto $X$. Finally $L \cap X$ is a (reduced) point, again a regular subvariety of $L$, so its normal cone equals its normal space equals a rank one vector bundle over the point. The class of this cone retracts onto a point in $X$.
        \end{enumerate}

\item[\ref{vc2}]

        The first statement is elementary.
        To describe the map of $\mls O_X$-algebras,
        let $\mls{E} = \oplus_{i} \OO(d_{i}) $.
        The $f_{i}$ assemble to give a section $s \colon \OO_{\PP^{n}} \to \mls{E}$.
        The dual of this section factors as
        \[
        \mls E^\vee \twoheadrightarrow I_{X/\PP^n} \hookrightarrow \mls O_{\PP^n},
        \]
       where $I = I_{X/\PP^n}$ is the ideal sheaf of $X$. The restriction of the first factor is a surjection $\mls E^\vee|_X \to I/I^2$ of $\mls O_X$-modules that induces the first factor of the composition
        \[Sym(\mls E^\vee) \to Sym(I_{X/\PP^n}/I_{X/\PP^n}^2) \to \oplus _{k \geq 0} I^k/I^{k+1}.\]
        This composed morphism of $\mls O_X$-modules is the one inducing the obstruction theory $C_{X/\PP^n} \to \relSpec_X(Sym(\oplus \mls O(-d_i)))$.

\item[\ref{vc5}]

(b) Recall that short exact sequences of locally free sheaves
    are always locally split.
    Hence, Zariski locally on $Y$,
    the morphism $E^{\prime} \to E$
    looks like the inclusion $\AA^{r}_{Y} \to \AA^{r+d}_{Y}$
    of a coordinate $r$-plane,
    where $r$ is the rank of $E^{\prime}$ and
    $d$ is the rank of $E^{\prime\prime}$. (Either reduce to the case where $Y$ is connected, so $r$ and $d$ are integers, or work here with locally constant integer-valued functions $r$ and $d$.)
    Thus,
    $E^{\prime} \to E$
    is a regular embedding of the same codimension as
    $Y \to E^{\prime\prime}$.

    In particular,
    the induced map
    $N_{E^{\prime}/E} \to p_{E^{\prime}}^{\ast}N_{Y/E^{\prime\prime}}
    \cong p_{E^{\prime}}^{\ast} E^{\prime\prime}$
    is a closed embedding of vector bundles of the same rank,
    so it must be an isomorphism as desired.

\end{enumerate}

\noindent
\textbf{Section 3}
\begin{enumerate}[leftmargin=.24in]

\item[\ref{excess3}]
\begin{enumerate}
    \item[(a)] The given generators of the ideal sheaf $I_{X/\PP^3}$ define a surjection $\mls O_{\PP^3}(-2)^{\oplus 3} \to I_{X/\PP^3} \subseteq \mls O_{\PP^3}$. A direct computation shows that the kernel of this surjection contains the image of the morphism $f: \mls O_{\PP^3}(-3) \to \mls O_{\PP^3}(-2)^{\oplus 3}$ corresponding to the matrix $(w, y, -z)^T$. Therefore on $\PP^3$  we have morphisms  \[
        \OO_{\PP^3}(-3)\xrightarrow{f} \OO_{\PP^3}(-2)^{\oplus 3}
        \twoheadrightarrow I_{X/\PP^3}
    \]
    whose composition is zero. Pulling back to $X$ and dualizing gives morphisms
    \[
    N_{X/\PP^3} \hookrightarrow \mls O_{\PP^3}(6)^{\oplus 3} \xrightarrow{f^\vee} \mls O_{\PP^3}(9)
    \]
    whose composition is zero. In particular $N_{X/\PP^3}$ injects into $\ker(f^\vee)$. To show that the inclusion $N_{X/\PP^3} \to \ker(f^\vee)$ is an isomorphism, it is enough to show $N_{X/\PP^3}$ and $\ker(f^\vee)$ have the same rank and same degree. (If $F$ and $G$ are coherent sheaves on $\PP^1$ of the same rank and $F \hookrightarrow G$, then the cokernel has rank zero hence is torsion, hence has positive degree if it is nonzero. So if $\deg(F) = \deg(G)$ then the inclusion $F \hookrightarrow G$ is an isomorphism.)

    Let $\PP^1_k = \mathrm{Proj}(k[s, t])$ and let $(s^3, s^2t, st^2, t^3)$ be the embedding $\PP^1 \to \PP^3$. The map $f^\vee$ is given by the matrix $(t^3, s^2t, -st^2)$. By computing the corresponding morphism of free modules in each of the two affine charts $D_+(s)$ and $D_+(t)$, we see that $\ker(f^\vee)$ is locally free of rank 2 and that $\mathrm{coker}(f^\vee)$ is a skyscraper sheaf with fiber $k$ supported at $[1:0]$. Therefore the degree of $\ker(f^\vee)$ is $18-9+1 = 10$.

    On the other hand $N_{X/\PP^3}$ is also locally free of rank 2 since $X$ is a smooth subscheme of $\PP^3$ of codimension 2. By the conormal exact sequence, we have
    \[\deg(N_{X/\PP^3}) = \deg(\Omega_X) - \deg(\Omega_{\PP^3}|_X) = -2 - 3\cdot(-4) = 10.\]
    This shows that $N_{X/\PP^3} \to \ker(f^\vee)$ is an inclusion of coherent (in fact, locally free) sheaves of the same rank and same degree as desired.

\item[(b)] Following the previous solution, we have an exact sequence
\[
0 \to N_{X/Y} \to \mls O_{\PP^1}(6)^{\oplus 3} \to \mls O_{\PP^1}(9) \xrightarrow{p} \OO_{[1:0]} \to 0,
\]
and truncating gives the short exact sequence
\[
0 \to N_{X/Y} \to \mls O_{\PP^1}(6)^{\oplus 3} \to \ker(p) \to 0.
\]
We know that $\ker(p)$ is torsion-free, being the subsheaf of a torsion-free sheaf,
so it is a vector bundle because we are on a curve.
We know that it has rank 1, so it is a line bundle.
We know that it has degree equal to $9-\deg(\OO_{[1:0]}) = 8$,
so it must be $\OO_{\PP^1}(8)$.

\item[(c)] By \eqref{eq:excess1} and part (b) we have $[\PP^1]^{vir}_E = c_{top}(\mls O_{\PP^1}(8)) \cap [\PP^1] = 8[pt].$
\end{enumerate}

\item[\ref{excess5}]
\begin{enumerate}

\item[(a)] The scheme $X$ is contained in the union of the distinguished affine open subsets $D_+(y), D_+(z) \subset \PP^2$. We have
\[
X \cap D_+(y) = \VV(x/y) \subseteq \Spec(k[x/y, z/y])
\]
which is a regular embedding in $D_+(y)$ with ideal sheaf $(x/y)$, so $C_{(X \cap D_+(y))/D_+(y)} = N_{(X \cap D_+(y))/D_+(y)}$ is a (trivial) rank one vector bundle. On the other hand,
\[
X \cap D_+(z) = \VV((x/z)^2, (x/z)(y/z)) \subseteq \Spec(k[x/z, y/z]).
\]
The normal cone $C_{(X \cap D_+(z))/D_+(z)}$ was computed in Exercise \ref{cone6}: its ring of global sections is
\[
k[x/z,y/z,A, B]/((x/z)^2, (x/z)(y/z), (y/z)A - (x/z)B).
\]
Using computer algebra software to compute the primary decomposition of the ideal, we see that this cone is the union of
\[
\VV(x/z, A) \quad \quad \text{and}\quad \quad \VV((x/z)^2, (x/z)(y/z), (y/z)^2, (y/z)A - (x/z)B)
\]
in $\Spec(k[x/z, y/z, A, B])$. The second scheme is $C_0$, and the first scheme (variety) glues with $C_{(X \cap D_+(y))/D_+(y)}$ to form $L$.

\item[(b)] To compute the multiplicity of $C_0$ we consider the local ring at the generic point:
\[
(k[x, y, A, B]/(x^2, xy, y^2, yA - xB))_{(x, y)} \simeq k(A, B)[x]/(x^2).
\]
This is a vector space of dimension 2 over the residue field $k(A, B)$, so $[C_0]$ is twice the class of its reduction $\AA^2_k$. Since $C_{X/Y} \to E$ is a closed embedding and $E$ is a rank-2 vector bundle, we see that the reduction of $C_0$ must be a fiber of this vector bundle, and $(p_E^*)^{-1}([C_0]) = 2[P]$.

    \item[(c)]
    We compute the closed embedding $L \to E$ after pulling back to the open set $X \cap D_+(y)$.  The equations for $X$ give a homomorphism $\mls O_{\PP^2}(-2) \oplus \mls O_{\PP^2}(-2) \to \mls O_{\PP^2}$ that, after restriction to $D_+(y)$, is given by the homomorphism of $k[x/y, z/y]$-modules
    \[
    k[x/y, z/y]^{\oplus 2} \to k[x/y, z/y],\quad \quad (a, b) \mapsto (x/y)^2a + (x/y)b \text{ for }a, b\in k[x/y, z/y].
    \]
    Observe that the image of this homomorphism is indeed the ideal $(x/y) \subset k[x/y, z/y]$ defining $X \cap D_+(y)$. Replacing the codomain with this ideal, and tensoring with $k[x/y, z/y]/(x/y)$, the kernel is the first summand of $k[x/y, z/y]^{\oplus 2}$, showing that the closed embedding $L|_{D_+(y)} \to \mls O_{\PP^2}(2) \oplus \mls O_{\PP^2}(2)|_{D_+(Y)}$ is the inclusion of the second summand.

    Since $L$ is a line bundle on the reduced subscheme of $X$ we conclude $L \simeq \mls O_{\PP^1}(2)$ and that the quotient bundle $E/L$ is $\mls O_{\PP^1}(2)$. By \eqref{eq:excess8} it follows that
    \[
    (p_E^*)^{-1}[L] = 2[P].
    \]

    \item[(d)] This is an immediate consequence of (a)-(c).

\end{enumerate}

\end{enumerate}

\noindent
\textbf{Section 4}
\begin{enumerate}[leftmargin=.24in]
\item[\ref{BF0}]
\begin{enumerate}

\item[(a)] We can reduce to a situation where $i, j$ are locally closed embeddings, $Y = \Spec(A)$, $Y' = \Spec(A')$, and $X$ is also affine. Let $\phi: A' \to A$ be the ring homomorphism inducing $f$. (Later, we will heavily use that $\phi$ is flat.) If $J \subset A'$ and $I \subset A$ are the ideals corresponding to the closed embeddings $X \to Y'$ and $X \to Y$, respectively,
then $I = \langle \phi(J)\rangle$ is the ideal generated by $\phi(J)$. A direct computation shows that for a general ring homomorphism $\phi: A' \to A$ and ideal $J \subset A'$, we have
\[
\langle \phi(J) \rangle^n = \langle \phi(J^n) \rangle,
\]
so in our case $I^n = \langle \phi(J^n) \rangle.$

Our goal is to show that the square
\[
\begin{tikzcd}
C_{X/Y} \arrow[r] \arrow[d] & C_{X/Y'} \arrow[d] \\
N_{X/Y} \arrow[r] & N_{X/Y'}
\end{tikzcd}
\]
is fibered, or in other words that this diagram is a tensor diagram:
\[
\begin{tikzcd}
\bigoplus_{n \geq 0} I^n/I^{n+1} & Sym(I/I^2) \arrow[l]\\
\bigoplus_{n \geq 0} J^n/J^{n+1} \arrow[u] & Sym(J/J^2) \arrow[l] \arrow[u].
\end{tikzcd}
\]

Since $\phi$ is flat, for $n \geq 1$ the tensor product of
\[
0 \to J^n \to A' \to A'/J^n \to 0
\]
with $A$ is still exact, which shows $J^n \otimes_{A'} A$ is isomorphic to $\langle \phi(J^n) \rangle$, i.e. to $I^n$ by the above discussion. In particular there is a tensor square
\[
\begin{tikzcd}
Sym(I/I^2) & A \arrow[l]\\
Sym(J/J^2) \arrow[u] & A' \arrow[l] \arrow[u],
\end{tikzcd}
\]
which reduces our problem to showing
\[
    (\bigoplus_{n \geq 0} J^n/J^{n+1}) \otimes_{A^{\prime}} A \simeq \bigoplus_{n \geq 0} I^n/I^{n+1}.
\]
But this is immediate from the formula $J^n \otimes_{A'} A = I^n$ and the fact that tensor is (always) right exact.

\item[(b)]
By Exercise \ref{cone1} the embedding $C_{X/Y} \to E_{X/Y}$ factors through a closed embedding $N_{X/Y} \to E_{X/Y}$. The morphism $i^*T_{Y/Y'} \to E_{X/Y}$ arising as a composition is therefore also a closed embedding,  and so has a well-defined quotient $E_{X/Y'}$ as in Definition \ref{def:quotient-bundle}. Build a commuting diagram of solid arrows
\[
\begin{tikzcd}
&&N_{X/Y} \arrow[r] \arrow[dd, bend left] & N_{X/Y'}\arrow[dd, bend left, dashrightarrow] \arrow[r] & 0 \\
0 \arrow[r] & i^*T_{Y/Y'} \arrow[ur] \arrow[r] & C_{X/Y} \arrow[r] \arrow[d] \arrow[u] & C_{X/Y'} \arrow[u] \arrow[d, dashrightarrow] \arrow[r] & 0\\
0 \arrow[r] &i^*T_{Y/Y'} \arrow[r] \arrow[u, equal] &E_{X/Y} \arrow[r] & E_{X/Y'} \arrow[r] & 0
\end{tikzcd}
\]
The top (crooked) sequence is exact, as is the bottom, so by the universal property of cokernels we get the dashed arrow $N_{X/Y'} \to E_{X/Y'}$. Now the dashed arrow $C_{X/Y'} \to E_{X/Y'}$ arises as a composition. The square with $C_{X/Y}, C_{X/Y'}, N_{X/Y},$ and $N_{X/Y'}$ is fibered by part (a); as is the square with $N_{X/Y}, N_{X/Y'}, E_{X/Y},$ and $E_{X/Y'}$ by Exercise \ref{excess6}; so the square with $C_{X/Y}, C_{X/Y'}, E_{X/Y},$ and $E_{X/Y'}$ is fibered. Since $C_{X/Y} \to E_{X/Y}$ is a closed embedding and $E_{X/Y} \to E_{X/Y'}$ is flat, the dashed arrow $C_{X/Y'} \to E_{X/Y'}$ is a closed embedding \cite[Tag 02L6]{stacks-project}; i.e., it is an obstruction theory.

To see that the virtual classes agree, we again use the fiber square
\[
\begin{tikzcd}
C_{X/Y} \arrow[r] \arrow[d] & C_{X/Y'} \arrow[d] \\
E_{X/Y} \arrow[r] & E_{X/Y'}.
\end{tikzcd}
\]
The virtual cycle associated to $i'$ is the class on $X$ whose flat pullback to (inverse image in) $E_{X/Y'}$ is $C_{X/Y'}$. Since the square is fibered, whatever this class is, it also pulls back to $C_{X/Y}$ in $E_{X/Y}$, so it must also be the virtual class of $i$.

\end{enumerate}

\item[\ref{BF2}]
We begin with commuting diagrams
\[
\begin{tikzcd}
U \arrow[r, "i"] \arrow[dr] & V \arrow[d] \\
& B
\end{tikzcd} \quad \quad \quad \begin{tikzcd}
U \arrow[r, "\Delta"] \arrow[dr, "f"'] & V\times_B V \arrow[d, "p_i"] \\
& V
\end{tikzcd}
\]
where $f$ and $\Delta$ are local immersions and $p_i: V \times_B V \to B$ is projection to the $i$th factor (in particular it is smooth). By \eqref{eq:cotangent} applied to the rightmost diagram we have a sequence of abelian cones
\[
0 \to \Delta^* T_{p_i} \to N_{U/V \times_B V} \xrightarrow{p_{i*}} N_{U/V} \to 0
\]
and by Exercise \ref{BF0} the normal cone $C_{U/V \times_B V} \subseteq N_{U/V \times_B V}$ is the preimage of $C_{U/V} \subseteq N_{U/V}$.
Since the compositions $V \to V \times_B V \xrightarrow{p_i} V$
are the identity for $i=1, 2$ (the first morphism $V \to V \times_B V$ being the diagonal), we have a map $s: N_{U/V} \to N_{U/V \times_B V}$ that splits both the maps $p_{i*}$. In particular we have commuting diagrams
\[
\begin{tikzcd}
0 \arrow[r] & \Delta^* T_{p_i} \arrow[r, "\partial p_i"] & N_{U/V \times_B V} \arrow[r, "p_{i*}"] & N_{U/V} \arrow[r] \arrow[l, bend left, "s"]&  0\\\
&&C_{U/V\times_B V} \arrow[u, hookrightarrow] \arrow[r] & C_{U/V} \arrow[u, hookrightarrow]\arrow[l, bend left, "s"]
\end{tikzcd}
\]
It remains to show that under the isomorphism $N_{U/V \times_B V} \simeq \Delta^*T_{p_1} \times  N_{U/V} \simeq T_{V/B} \times N_{U/V}$ induced by the splitting of $p_1$, the action of $i^*T_{V/B}$ on $N_{U/V}$ is identified with the map $p_{2*}$. (Since the isomorphism above restricts to an isomorphism $C_{U/V \times_B V} \simeq i^*T_{V/B} \times C_{U/V}$, commutativity of the above diagram shows the action of $i^*T_{V/B}$ preserves $C_{U/V}$.) More precisely, we want to show that the curved arrow in the diagram below is the canonical inclusion.
\begin{equation}\label{eq:action}
\begin{tikzcd}
&&\Delta^* T_{p_1} \times N_{U/V} \arrow[d, "\sim"] \arrow[dl]\\
i^*T_{V/B} \arrow[rrr, bend right] & \Delta^*T_{p_1} \arrow[l, "\sim"] \arrow[r, "\partial p_1"] & N_{U/V \times_B V}  \arrow[r, "p_{2*}"] & N_{U/V}
\end{tikzcd}
\end{equation}

To prove this compatibility, note that by functoriality of the cotangent  complex, the commuting diagram
\begin{equation}
    \label{eq:morphism of ctgt triangles}
    \begin{tikzcd}
        U \arrow[r, "\Delta"] \arrow[d, equal] & V \times_B V \arrow[r, "p_1"] \arrow[d, "p_2"] & V \arrow[d] \\
        U \arrow[r, "i"] & V \arrow[r] & B
    \end{tikzcd}
\end{equation}
leads to a morphism of distinguished triangles
\[
\begin{tikzcd}
i^*\LL_{V/B} \arrow[d, "\sim"] \arrow[r] & \LL_{U/B} \arrow[d] \arrow[r] & \LL_{U/V} \arrow[d] \arrow[r] & {}\\
\Delta^*\LL_{p_1} \arrow[r] & \LL_{U/V} \arrow[r] & \LL_{U/V \times_B V} \arrow[r] &.
\end{tikzcd}
\]
In particular, looking at the partially complete square at the right side of this diagram and taking homology in degree -1, we get a commuting square
\begin{equation}
    \label{eq:commuting square of ctgt and conormal}
\begin{tikzcd}
\mls C_{U/V} \arrow[r] \arrow[d] & i^*\Omega_{V/B} \arrow[d, "\sim"] \\
\mls C_{U/V \times_B V} \arrow[r] & \Delta^*\Omega_{p_1}.
\end{tikzcd}
\end{equation}
where $\mls C_{U/V}$ and $\mls C_{U/V \times_B V}$ are the conormal sheaves. (Commutativity of this square can also be checked directly, without the theory of the cotangent complex.) Applying $\relSpec_U(Sym(-))$ to this square and comparing with \eqref{eq:action} we see that it is precisely the compatibility we needed to show.

\item[\ref{BF3}] (Sketch) If $X \to B$ is a closed embedding of schemes then $\mf C_{X/B} \to \mf N_{X/B}$ is the usual embedding $C_{X/B} \to N_{X/B}$ of the normal cone in the normal space. If $X \to B$ is smooth then both $\mf C_{X/B}$ and $\mf N_{X/B}$ are the quotient stack $[X/T_{X/B}]$.

\item[\ref{BF5}]
\begin{itemize}
\item[(a)]

Without loss of generality,
we may assume that $X$ is a scheme. Let $x \in X$ be the image of the geometric point $\overline{x}$ and let $b \in B$ be the image of $x$. Let $W$ be an affine scheme with a smooth morphism $W \to B$ containing $b$.
Then, we may choose an affine scheme $U'$ with an \'etale morphism $U^{\prime} \to X \times_B W$ such that the image of $U'$ in $X$ contains $x$.
Moreover, since $U'$ is of finite type over $k$, it admits a closed embedding into $\AA^r_k$ for some $r \geq 0$.
The composition $U^{\prime} \xrightarrow{i} \AA^r_k \times_k W \to \AA^r$ is a closed embedding
and the second morphism is separated,
so $i$ is a closed embedding.

We can almost take $U = U'$ and $V = \AA^r_k \times_k W$, but $U^{\prime} \to X$ need not be \'etale.
By
\cite[Tag~057G, 056U]{stacks-project},
there is a locally closed immersion $U \to U'$ such that the composition $U \to U' \to X$ is \'etale and contains the image of $\bar x$.
Now $U' \to U \to \AA^r_k \times_k W$
is a locally closed immersion, so there is an open subset $V \subset \AA^r_k \times_k W$ (smooth over $B$) such that $U \to V$ is a closed embedding.

\end{itemize}

\item[\ref{BF4}]

For (a), define a functor from $[E^1/E^0]$ to $[F^1/F^0]$ to send $(P, \alpha) \in [E^1/E^0](T)$ to the pair
\[
(P \times_{E^0, \phi^0} F^0,  \phi \alpha + d )
\]
where $P \to T$ is a principal $E_0$-bundle, $\alpha: P \to E^1$ is an equivariant morphism, $P \times_{E^0, \phi^0} F^0$ is the quotient of $P \times F^0$ by the action of $E^0$ given by
\[
e \cdot (p, f) = (ep, f-\phi^0(e)),
\]
and $\phi\alpha + d$ is induced by the morphism  $P \times F^0 \to F^1$ given by
\[
(p, f) \mapsto \phi^1\alpha(p) + df.
\]

For (b), use the isomorphism of principal bundles
\[
P \times_{E^0, \phi^0} F^0 \to P \times_{E^0, \psi^0} F^0
\]
induced by the endomorphism of $P \times F^0$ given by the rule
\[
(p, f) \mapsto (p, f - k\alpha(p)).
\]
\end{enumerate}

\noindent
\textbf{Section 5}

\begin{itemize}[leftmargin=.24in]

\item[\ref{POT3}]
\begin{itemize}

\item[(a)]

\begin{lemma}\label{lem:pullback-qiso}
Let $A$ be a ring, and let $N_1 \to M$, $N_2 \to M$ be morphisms of $A$ -modules such that $N_1 \oplus N_2 \to M$ is surjective. Let $K$ be the kernel of $N_1 \oplus N_2 \to M$. Then $[K \to N_1]$ is quasi-isomorphic to $[N_2 \to M]$.
\end{lemma}
\begin{proof}
This can be solved by diagram chasing.
\end{proof}

By Exercise \ref{BF5}, there is a commuting diagram
\[
\begin{tikzcd}
U \arrow[d] \arrow[r] & V \arrow[d] \\
X \arrow[r] & B
\end{tikzcd}
\]
with $U \to V$ a closed embedding of schemes, $U \to X$ \'etale, and $V \to B$ smooth.
By replacing $V$ with an affine open subset and replacing $U$ with its intersection with this affine open subset, we can moreover assume $U = \Spec(A)$ and $V$ are affine. So localizing further if necessary we have that $\EE|_U \simeq [\mls E^{-1} \to \mls E^0]$ and $\mls E^{-1}$, $\mls E^0$, and $\Omega_{V/B}|_U$ correspond to projective $A$-modules.
Moreover we compute
\[
\tau_{\geq -1}(\LL_{X/B}|_U) = \tau_{\geq -1}\LL_{U/B} = [\mls I/\mls I^2 \to \Omega_{V/B}|_U]
\]
where
the first morphism uses Exercise \ref{POT1}(a) and the second uses property \ref{eq:C2} of cotangent complexes. Since $U$ is affine and $\mls E^{-1}|_U \to \mls E^{0}|_U$ corresponds to a morphism of projective modules, the morphism $\tau_{\geq -1}(\phi|_U)$
\[[\mls E^{-1}|_U \to \mls E^0|_U] \to [\mls I/\mls I^2 \to \Omega_{V/B}|_U]\]
in the derived category is represented by an honest morphism of complexes. That is, we have a commuting diagram
\[
\begin{tikzcd}
\mls E^{-1} \arrow[r, "d"] \arrow[d, "\phi_1"] & \mls E^0 \arrow[d, "\phi_0"] \\
\mls I/\mls I^2 \arrow[r] & \Omega_{V/B}
\end{tikzcd}
\]
of sheaves on $U$, where here and going forward we omit restrictions to $U$ from the notation.
This diagram induces a surjection on cohomology in degree $-1$ and an isomorphism in degree $0$.

We now construct a complex quasi-isomorphic to $\mls E^{-1} \to \mls E^{0}$ with a morphism to $\mls I/\mls I^2 \to \Omega_{V/B}$ that is a genuine surjection in degree $-1$ and an isomorphism in degree 0. For this, observe that $\Omega_{V/B}$ has a morphism to $\mathrm{coker}(d)$. Since $\Omega_{V/B}$ corresponds to a projective module this map lifts to a map $f: \Omega_{V/B} \to \mls E^0$. In fact $\mls E^{-1} \oplus \Omega_{V/B} \xrightarrow{d \oplus f} \mls E^0$ is surjective: given $a \in \mls E^0$, the element $a - f(\phi_0(a))$ maps to zero under $\mls E^0 \to \mathrm{coker}(d)$. Hence we may apply Lemma \ref{lem:pullback-qiso} and we see that $[\mls E^{-1} \to \mls E^0]$ is quasi-isomorphic to $[\mls E^{-1} \times_{\mls E^0} \Omega_{V/B} \to \Omega_{V/B}].$ There is an induced morphism of complexes
\[
[\mls E^{-1} \times_{\mls E^0} \Omega_{V/B} \to \Omega_{V/B}] \to [\mls I/\mls I^2 \to \Omega_{V/B}]
\]
which uses the identity map on $\Omega_{V/B}$ and induces a surjection on cohomology in degree -1 and an isomorphism in degree 0. Therefore we have a commuting diagram with exact rows
\[
\begin{tikzcd}
\ker(\mls E^{-1} \times_{\mls E^0} \Omega_{V/B} \to \Omega_{V/B}) \arrow[d, twoheadrightarrow] \arrow[r] & \mls E^{-1} \times_{\mls E^0} \Omega_{V/B} \arrow[d, "g"] \arrow[r] & \Omega_{V/B} \arrow[d, equal] \arrow[r] & \mathrm{coker}(\mls E^{-1} \times_{\mls E^0} \Omega_{V/B} \to \Omega_{V/B}) \arrow[d, equal] \\
\ker(\mls I/\mls I^2 \to \Omega_{V/B}) \arrow[r] & I/I^2 \arrow[r] & \Omega_{V/B} \arrow[r] & \mathrm{coker}(I/I^2 \to \Omega_{V/B})
\end{tikzcd}
\]
and hence by the 4-lemma the arrow labeled $g$ is surjective. The middle square in the diagram above is thus the square required in the exercise.

\item[(b)] A perfect obstruction theory $\EE_{X/B} \to \LL_{X/B}$ induces a morphism $\EE_{X/B} \to \tau_{\geq -1}\LL_{X/B}$. That the corresponding morphism of abelian cone stacks is a closed embedding can be checked \'etale locally on $X$. By part (a) and Lemma \ref{lem:BF88}, we can find an affine \'etale cover $U \to X$ where the morphism $st(\LL_{X/B}) \to st(\EE_{X/B})$ arises from applying $\relSpec_X(Sym(-))$ to a
commuting diagram
\[
\begin{tikzcd}
\mls E^{-1} \arrow[r] \arrow[d, "\phi_1"] & \mls E^0 \arrow[d, "\phi_0"] \\
\mls I/\mls I^2 \arrow[r] & \Omega_{V/B}
\end{tikzcd}
\]
with $\phi_1$ surjective and $\phi_0$ an isomorphism. Therefore $st(\LL_{X/B})_U \to st(\EE_{X/B})_U$ is a closed embedding.

\end{itemize}
\item[\ref{POT100}]

The natural complex map $\EE_{X/B} \to \tau_{\geq -1}\LL_{X/B}$ is induced by the commuting diagram
\[
\begin{tikzcd}
\mls E|_X \arrow[r] \arrow[d] & \Omega_{Y/B}|_X \arrow[d, equal]\\
\mls I/\mls I^2 \arrow[r] & \Omega_{Y/B}|_X
\end{tikzcd}
\]
Note that the left vertical arrow is surjective and the right is an isomorphism by construction.

   For ease of notation,
    write $\EE = \EE_{X/B}$ and $\LL = \LL_{X/B}$.
    Applying internal hom $R\Hom$ for $D_{qc}(X)$ to the distinguished triangle
    \[
        \tau_{\leq -2}\LL \to \LL \to \tau_{\geq -1}\LL \xrightarrow{+1}
    \]
    yields a long exact sequence
    \[
        \begin{tikzcd}
            \cdots \rar & H^{0}(X, R\Hom(\EE, \tau_{\leq -2}\LL)) \rar&
            H^{0}(X, R\Hom(\EE, \LL)) \rar
                        \ar[draw=none]{d}[name=X, anchor=center]{}&
            H^{0}(X, R\Hom(\EE, \tau_{\geq -1}\LL))
            \ar[rounded corners,
            to path={ -- ([xshift=2ex]\tikztostart.east)
                      |- (X.center) \tikztonodes
                      -| ([xshift=-2ex]\tikztotarget.west)
                  -- (\tikztotarget)}]{dll}[at end]{} \\
                                           & H^{1}(X, R\Hom(\EE, \tau_{\leq -2}\LL)) \rar & \cdots &.
        \end{tikzcd}
    \]
    Since $X$ is affine, $\EE$ corresponds to a finite complex of projective modules, and by \cite[Tag 0A66]{stacks-project} $R\Hom(\EE, -)$ is computed by the hom complex. Moreover since $X$ is affine the hypercohomology of the hom complex is just its cohomology. Applying these observations to the outer groups and  \cite[Tag 08JA]{stacks-project} to the inner groups, the long exact sequence above becomes
    \[
        \dots \to
        \Hom_{\mls O_X}(\EE, \tau_{\leq -2}\LL) \to \Hom_{D_{qc}(X)}(\EE, \LL)
        \to \Hom_{D_{qc}(X)}(\EE, \tau_{\geq -1}\LL)
        \to \Hom_{\mls O_X}(\EE, \tau_{\leq -2}\LL[1])
        \to \cdots
    \]
    where $\Hom_{\mls O_X}$ is the set of morphisms of chain complexes and $\Hom_{D_{qc}(X)}$ is the set of morphisms in the derived category.

    The groups on the two ends of the displayed sequence must vanish: this is because $\EE$ has terms in degrees $-1$ and above, but $\tau_{\leq -2}\LL$ and  $\tau_{\leq -2}\LL[1]$ have terms in degrees $-2$ and below. Thus we have a canonical isomorphism
    \[
    \Hom_{D_{qc}(X)}(\EE, \LL)
        \simeq \Hom_{D_{qc}(X)}(\EE, \tau_{\geq -1}\LL).
    \]

    It remains to see that the lift of $\EE \to \tau_{\geq -1}\LL$ to a morphism $\EE \to \LL$ is a perfect obstruction theory. This can be checked on the truncation $\EE \to \tau_{\geq -1}\LL$, where it is clear from the construction.

\item[\ref{POT5}]
 We first prove Proposition \ref{prop:local1}.
The fiber square
\[
\begin{tikzcd}
X \arrow[r]\arrow[d] & Y \arrow[d, "0"]\\
Y \arrow[r, "s"] & E
\end{tikzcd}
\]
leads to a canonical morphism of cotangent complexes $\LL_{Y/E} \to \LL_{X/Y}$. Since $0: Y \to E$ is l.c.i., its cotangent complex is $[\mls E \to 0]$ for $\mls E$ the dual of the sheaf of sections of $E$.
Applying $\tau_{\geq -1}$ to the morphism $\LL_{Y/E} \to \LL_{X/Y}$ yields a commuting square
\[
\begin{tikzcd}
\mls E^\vee|_X \arrow[d] \arrow[r] & 0 \arrow[d] \\
\mls I/\mls I^2 \arrow[r] & 0
\end{tikzcd}
\]
where the leftmost vertical arrow is the surjection induced by $s$.  It follows that $\EE_{X/Y} := \LL_{Y/E} \to \LL_{X/Y}$ is a Behrend-Fantechi obstruction theory, and that the inclusion of cones $\mf C_{X/Y} \to st(\EE_{X/Y})$ associated to the obstruction theory is precisely the inclusion $C_{X/Y} \to E$ induced by $s$. So the virtual pullbacks agree.

For Proposition \ref{prop:local2}, let $\phi: \EE_{X/B} \to \LL_{X/B}$ be a perfect obstruction theory. By Exercise \ref{POT3} we have a commuting square
\[
\begin{tikzcd}
U \arrow[d] \arrow[r] & Y \arrow[d] \\
X \arrow[r] & B
\end{tikzcd}
\]
where $U \to X$ is \'etale, $U \to Y$ is a closed embedding, and $Y \to B$ is smooth, and moreover the truncation of the restriction of $\EE_{X/B} \to  \LL_{X/B}$ to $U$ is quasi-isomorphic to a diagram
\begin{equation}\label{eq:otshere}
\begin{tikzcd}
\mls E^{-1} \arrow[r] \arrow[d, twoheadrightarrow] & \mls E^0 \arrow[d, "\simeq"] \\
\mls I/\mls I^2 \arrow[r] & \Omega_{Y/B}
\end{tikzcd}
\end{equation}
where $\mls I$ is the ideal sheaf of $U$ in $Y$.
By further localizing on $U$ we can assume $U$ is affine and $\mls E^{-1}$ is a free sheaf of rank $r$. By Nakayama's lemma, after further localizing on $Y$ we can lift the surjection $\mls O_U^{\oplus r} = \mls E^{-1}  \to \mls I/\mls I^2$ to a surjection
\begin{equation}\label{eq:surjhere}
\mls O_Y^{\oplus r} \to \mls I.
\end{equation}
We set $E = \relSpec_Y(Sym(\mls O_Y^{\oplus r}))$ and we let $s: Y \to E$ be the section arising from \eqref{eq:surjhere}.
With these definitions we have that \eqref{eq:otshere} is precisely the truncation of the obstruction theory for $U \to X$ associated to $(E, s)$. Since truncated obstruction theories lift uniquely to obstruction theories (Lemma \ref{lem:lift}) the result follows.

\end{itemize}

\noindent
\textbf{Section 6}

\begin{itemize}[leftmargin=.24in]

\item[\ref{POT14}]

\begin{itemize}

  \item[(a)]      Consider the commutative diagram

    \[\begin{tikzcd}
	T & {X^{\prime}} & X \\
	{T^{\prime}} & {B^{\prime}} & B
	\arrow["h", from=1-1, to=1-2]
	\arrow["i"', from=1-1, to=2-1]
	\arrow["q", from=1-2, to=1-3]
	\arrow["g"', from=1-2, to=2-2]
	\arrow["f", from=1-3, to=2-3]
	\arrow[dashed, from=2-1, to=1-2]
	\arrow[color={rgb,255:red,214;green,92;blue,92}, dashed, from=2-1, to=1-3]
	\arrow[from=2-1, to=2-2]
	\arrow["p"', from=2-2, to=2-3]
\end{tikzcd}\]

    Since the right square is cartesian,
    producing the black dotted arrow
    is equivalent to producing the red dotted arrow.

The black arrow exists if and only if the obstruction $o(h)$ given by
\[
h^{\ast}\LL_{X'/B'}
    \to \LL_{T/T^{\prime}} \to \tau_{\geq -1} \LL_{T/T^{\prime}},
\]
vanishes, while the red arrow exists if and only if the obstruction $o(q \circ h)$ given by
\[
(q \circ h)^{\ast}\LL_{X/B}
    \to \LL_{T/T^{\prime}} \to \tau_{\geq -1} \LL_{T/T^{\prime}},
\]
vanishes. Since $o(h)$ maps to $o(q \circ h)$ under the canonical morphism $\Ext^1(Lh^*\LL_{X'/B'}, I) \to \Ext^1(L(q \circ h)^*\LL_{X/B}, I)$ (given by precomposition with the canonical map $(q \circ h)^*\LL_{X/B} \to h^*\LL_{X'/B'}$), we have verified
    part (a) of
    Proposition \ref{prop:criterion}, item 2.

    For part (b) of the item,
    we note that the canonical morphism
    $h^{\ast}q^{\ast}\LL_{X/B} \to h^{\ast}\LL_{X^{\prime}/B^{\prime}}$
    induces an isomorphism on $h^{0}$,
    since cokernels are preserved by pullbacks and the canonical map $q^*\Omega_{X/B} \to \Omega_{X'/B'}$ is an isomorphism.
    Thus,
    the morphism
    $\Ext^{0}(h^{\ast}\LL_{X^{\prime}/B^{\prime}}, I) \to
    \Ext^{0}(h^{\ast}q^{\ast}\LL_{X/B}, I)$
    is an isomorphism.

    \item[(b)]
    The statement that $\EE_{X/B} \to \LL_{X/B}$ induces an isomorphism on $h^0$ and a surjection on $h^{-1}$ is equivalent to the statement that the mapping cone of $\EE_{X/B} \to \LL_{X/B}$ has $h^i=0$ for $i \geq -1$. The latter statement is clearly preserved by derived pullback, so $Lp^*\EE_{X/B} \to Lp^*\LL_{X/B}$ induces an isomorphism on $h^0$ and a surjection on $h^{-1}$. Since $Lp^*\LL_{X/B} \to \LL_{X'/B'}$ has the same properties by part (a), we see that the composition is an obstruction theory. Finally the complex $Lp^*\EE_{X/B}$ is perfect of the correct tor amplitude because pullback preserves this property.
\end{itemize}

\item[\ref{POT7}]

By the axioms of triangulated categories we have a morphism of distinguished triangles
\[
\begin{tikzcd}
\EE_{X/Y'} \arrow[r] \arrow[d] & \EE_{X/Y} \arrow[r, "h"] \arrow[d] & i^*\LL_{Y/Y'}[1] \arrow[r] \arrow[d, equal]& {}\\
\LL_{X/Y'} \arrow[r] & \LL_{X/Y} \arrow[r] & i^*\LL_{Y/Y'}[1] \arrow[r] & {}.
\end{tikzcd}
\]
This leads to a morphism of long exact sequences
\[
\begin{tikzcd}
0 \arrow[r] & h^{-1}(\EE_{X/Y'}) \arrow[d]\arrow[r] & h^{-1}(\EE_{X/Y}) \arrow[r] \arrow[d]& \Omega_{Y/Y'}|_X \arrow[r]\arrow[d] & h^{0}(\EE_{X/Y'}) \arrow[r] \arrow[d]& h^{0}(\EE_{X/Y}) \arrow[r] \arrow[d]& 0\\
0 \arrow[r] & h^{-1}(\LL_{X/Y'}) \arrow[r] & h^{-1}(\LL_{X/Y}) \arrow[r] & \Omega_{Y/Y'}|_X \arrow[r] & h^{0}(\LL_{X/Y'}) \arrow[r] & h^{0}(\LL_{X/Y}) \arrow[r] & 0
\end{tikzcd}
\]
Using the top row and the fact that $\EE_{X/Y}$ has tor amplitude in $[-1, 0]$ we see that $\EE_{X/Y'}$ is perfect. Multiple applications of the 4-lemma show that $\EE_{X/Y'} \to \LL_{X/Y'}$ is an obstruction theory.

To show that the virtual classes agree, it is enough to show the existence of a fiber diagram
\begin{equation}\label{eq:rew}
\begin{tikzcd}
\mf C_{X/Y} \arrow[r] \arrow[d] & \mf N_{X/Y} \arrow[r] \arrow[d] & st(\EE_{X/Y}) \arrow[d] \\
\mf C_{X/Y'} \arrow[r] & \mf N_{X/Y'} \arrow[r] & st(\EE_{X/Y'}).
\end{tikzcd}
\end{equation}
The squares can be checked independently and locally on $X$. For the leftmost square, we can find a commuting diagram
\[
\begin{tikzcd}
U\arrow[r, hookrightarrow] \arrow[d] & V \arrow[d] \\
X \arrow[r] \arrow[dr] & Y \arrow[d] \\
&Y'
\end{tikzcd}
\]
with $U \to X$ \'etale, $U\to V$ a closed embedding, and $V \to Y$ smooth. Then the leftmost square is the square
\[
\begin{tikzcd}
{[C_{U/V}/T_{V/Y}|_U]} \arrow[d] \arrow[r] & {[N_{U/V}/T_{V/Y}|_U]} \arrow[d] \\
{[C_{U/V}/T_{V/Y'}|_U]} \arrow[r] & {[N_{U/V}/T_{V/Y'}|_U]}
\end{tikzcd}
\]
which is easily seen to be fibered.

To show the other square is fibered, by Exercise \ref{POT3} we can in fact choose $U$ and $V$ such that $\EE_{X/Y'}|_U \to \LL_{X/Y'}|_U$ and $\EE_{X/Y}|_U \to \LL_{X/Y}|_U$ are given respectively by commuting squares
\[
\begin{tikzcd}
\mls F^{-1} \arrow[r] \arrow[d, twoheadrightarrow] & \Omega_{V/Y'} \arrow[d, "\simeq"] \\
I/I^2 \arrow[r] & \Omega_{V/Y'}
\end{tikzcd} \quad\quad\text{ and }\quad\quad \begin{tikzcd}
\mls E^{-1} \arrow[r] \arrow[d, twoheadrightarrow] & \Omega_{V/Y} \arrow[d, "\simeq"] \\
I/I^2 \arrow[r] & \Omega_{V/Y}
\end{tikzcd}
\]
with $\mls E^{-1}$ and $\mls F^{-1}$ locally free sheaves on $U$ and $I$ the ideal sheaf of $U \hookrightarrow V$. We have moreover that the induced morphism of mapping cones
\begin{equation}\label{eq:wer}
\begin{tikzcd}
\mls F^{-1} \arrow[d] \arrow[r, "i"] & \mls E^{-1} \oplus \Omega_{V/Y'} \arrow[r, "p"]\arrow[d ] & \Omega_{V/Y} \arrow[d]\\
I/I^2 \arrow[r, "j"] & I/I^2 \oplus \Omega_{V/Y'} \arrow[r, "q"] & \Omega_{V/Y}
\end{tikzcd}
\end{equation}
is a quasi-isomorphism. Note that $i$ and $j$ are injective and $q$ and $p$ are surjective, as are all vertical arrows. Therefore we have a commuting diagram of short exact sequences
\[
\begin{tikzcd}
0 \arrow[r] & \mls F^{-1} \arrow[d] \arrow[r, "i"] & \ker(p) \arrow[r, "p"]\arrow[d ] & \ker(p)/\mathrm{img}(i) \arrow[r] \arrow[d, "\sim"]& 0\\
0 \arrow[r] & I/I^2 \arrow[r, "j"] & \ker(q) \arrow[r, "q"] & \ker(q)/\mathrm{img}(j) \arrow[r] & 0
\end{tikzcd}
\]
where all vertical arrows are still surjective (and the rightmost is again an isomorphism). The leftmost square in this diagram is Cartesian: injectivity of the map
\[
\mls F^{-1} \to I/I^2 \times_{\ker(q)} \ker(p)
\]
follows from injectivity of $i$, and surjectivity is a diagram chase. It follows that the leftmost square of \eqref{eq:wer} is also Cartesian, and hence the canonical morphism $\mls F^{-1} \to \mls E^{-1}$ is an isomorphism. At last, we see that the rightmost square of \eqref{eq:rew}, after pullback to $U$, may be written
\[
\begin{tikzcd}
{[C_{U/V}/T_{V/Y}|_U]} \arrow[r] \arrow[d] & {[E^1/T_{V/Y}|_U]} \arrow[d]\\
{[C_{U/V}/T_{V/Y'}|_U]} \arrow[r] & {[E^1/T_{V/Y'}|_U]}
\end{tikzcd}
\]
where $E^1 = \relSpec_U(Sym(\mls E^{-1})) \simeq \relSpec_U(Sym(\mls F^{-1}))  $. This square is clearly fibered.

\end{itemize}

\noindent
\textbf{Section 7}
\begin{itemize}[leftmargin=.24in]
\item[\ref{maps2}]
Since $X$ is smooth we have $\LL_X = \Omega_X[0]$. Let $\EE = R\pi_*(f^*\Omega_X \otimes \omega[1])$. Perfection of $\EE$ is a smooth local property by definition, and the tor amplitude of $\EE$ can also be computed locally by \cite[Tag 0DJJ]{stacks-project}. So replace $\mc M_{g, n}(X)$ by an \'etale cover $U$ that is a Noetherian scheme (in particular $\EE$ is now a complex on $U$).

 Then \cite[Tag 08EV]{stacks-project} says this complex is perfect, and it remains to show it has tor amplitude in $[-1, 0]$. Since the cohomology of $\EE$ is nonzero only in $[-1, 0]$, by \cite[Prop 8.3.6.4]{fgaexplained}, to show that $\EE$ has tor amplitude in $[-1, 0]$ it is enough to show that for every $u \in U$ the derived fiber $\kappa(u) \otimes^{\mathrm{L}} \EE$ has vanishing cohomology in degree -2 (here $\kappa(u)$ is the residue field of $u$). From the fiber square
\[
\begin{tikzcd}
C \arrow[d, "\pi_C"'] \arrow[r, "j"] & \mls C_{\mc M_{g, n}(X)} \arrow[d, "\pi"]\\
\kappa(u) \arrow[r, "i"] & \mc M_{g, n}(X),
\end{tikzcd}
\]
 since $\pi$ is flat, by \cite[Tag 08IR]{stacks-project} we have
 \[
 \kappa(u) \otimes^{\mathrm{L}} \EE = R\pi_{C, *} Lj^*(f^*\Omega_X \otimes \omega[1]).
 \]
 But since $X$ is smooth, $f^*\Omega_X \otimes \omega[1]$ is a locally free sheaf (in degree -1), so we can compute $Lj^*$ of this complex by directly applying $j^*$, and the derived fiber is
 \[
 R\pi_{C, *}((f \circ j)^*\Omega_X \otimes \omega_C[1]).
 \]
 This has vanishing cohomology in degree -2.

\item[\ref{maps3}]
\begin{itemize}

\item[(a)]

The rank of $R\pi_*(f^*\Omega_X \otimes \omega_\pi[1])$ may be computed after replacing $\mathcal{M}_{g, n}(X)$ with any of its geometric points, assuming the rank turns out to be constant (not just locally constant). So letting $f_K: C \to X$ be a stable map corresponding to a $K$-point of $\mathcal{M}_{g,n}(X)$ and $\pi_K: C \to \Spec(K)$ be the structure morphism, for $K$ an algebraically closed field, we wish to compute the rank of $\EE = R\pi_{K,*}(f_K^*\Omega_X \otimes \omega_{\pi_K}[1])$ (using the derived base change formula). Writing $\EE = [\sE^{-1} \xrightarrow{a} \sE^0 ]$ for some $K$-modules $\sE^i$, we have that the rank is $\rank(\sE^0) - \rank(\sE^{-1})$. Since rank is additive in short exact sequences of vector spaces we have
\begin{align*}
\rank(\EE) &= \big(h^0(\EE) + \mathrm{dim}(\mathrm{img}(a))\big) - \big(h^{-1}(\EE) +  \mathrm{dim}(\mathrm{img}(a))\big)\\
&= h^0(\EE) - h^{-1}(\EE)\\
&= h^1(C, f_K^*\Omega_X \otimes \omega_{\pi_K})  - h^0(C, f_K^*\Omega_X \otimes \omega_{\pi_K})\\
&= -\chi(f_K^*\Omega_X \otimes \omega_{\pi_K}).
\end{align*}
By Serre duality this is equal to $\chi(f_K^*\Omega_X^\vee).$

\end{itemize}
\end{itemize}

\noindent
\textbf{Section 8}
\begin{itemize}[leftmargin=.24in]
\item[\ref{GW1}]

\begin{itemize}
\item[(a)]
Recall the statement of Grothendieck duality for $\pi: C \to M$ a family of curves (a sheafified version of the statement that $\pi^*(-) \otimes \omega_\pi$ is right adjoint to $R\pi_*$):
\[
    R\mls Hom_M(R\pi_*F, G) = R\pi_*R\mls Hom(F, \pi^*G \otimes \omega_\pi[1]).
\]
We will work in the situation where $\pi$ is the universal family of curves on $\overline{\mc M}_{g, n}(X)$.

First set $G = \mls O_M$ and $F = f^*\Omega_X \otimes \omega_{\pi}[1]$. Then Grothendieck duality says
\begin{equation}\label{eq:sdf}
(R\pi_*(f^*\Omega_X \otimes \omega_\pi[1]))^\vee = R\pi_*(f^*\Omega_X^\vee).
\end{equation}
(We will only apply the dual symbol ${}^\vee$ to complexes that are known to be perfect, so applying dual twice returns the original complex.) The complexes on both sides of this equation have potentially nonzero cohomology in degrees 0 and 1. By assumption $h^{-1}(R\pi_*(f^*\Omega_X \otimes \omega_\pi[1])) = 0.$ We claim this implies $h^1$ of the left-hand side of \eqref{eq:sdf} vanishes. Indeed, this can be checked locally, where since $R\pi_*(f^*\Omega_X \otimes \omega_\pi[1])$ is perfect we may represent it as a 2-term complex of finite rank locally free sheaves $\mls E^{-1} \to \mls E^0$. The assumption that $h^{-1}$ of this complex vanishes implies that the map is injective; i.e., $\mls E^{-1} \hookrightarrow \mls E^0$. The dual of $R\pi_*(f^*\Omega_X \otimes \omega_\pi[1])$ can also be computed locally and we see it is given by a surjection $(\mls E^0)^\vee \twoheadrightarrow (\mls E^{-1})^\vee$. In particular its $h^1$ vanishes. The conclusion of this whole paragraph is that, using \eqref{eq:sdf}, $h^1$ of $R\pi_*(f^*\Omega_X^\vee)$ vanishes.

Now we use Grothendieck duality again, this time with $G = \mls O_M$ and $F = f^*\Omega_X^\vee$. We obtain
\begin{equation}\label{eq:fds}
(R\pi_*f^*\Omega_X^\vee)^\vee = R\pi_*(f^*\Omega_X\otimes\omega_{\pi}[1])
\end{equation}
We claim that since $h^1(R\pi_*f^*\Omega_X^\vee)$ vanishes, $h^0$ of both sides of \eqref{eq:fds} is locally free (and this is precisely what we want to show). As before this claim can be checked locally on $M$ where $R\pi_*f^*\Omega_X^\vee$ is represented by a complex of finite rank locally free sheaves $\mls F^0 \to \mls F^1$. The assumption that $h^1$ vanishes implies this map is surjective; i.e., we have $\mls F^0 \twoheadrightarrow \mls F^1$. Therefore the kernel $\mls K$ of the map is locally free. Dualizing we see that $(R\pi_*f^*\Omega_X^\vee)^\vee$ is locally represented by the complex $(\mls F^1)^\vee \hookrightarrow (\mls F^0)^\vee$ with locally free cokernel $\mls K^\vee$. In other words, $h^0$ of $(R\pi_*f^*\Omega_X^\vee)^\vee$ is locally free.

\item[(b)] To show $h^{-1}(\EE)=0$, it is enough to show $R^1\pi_*f^*\Omega^\vee_X = 0$. Indeed, if $R^1\pi_*f^*\Omega_X^\vee = 0$ then the complex $R\pi_*f^*\Omega_X^\vee$ is locally represented by a complex of finite rank locally free sheaves $\mls F^0 \to \mls F^1$ where the morphism is surjective. Dualizing we see that $(R\pi_*f^*\Omega_X^\vee)^\vee$ is locally represented by an injective morphism $(\mls F^1)^\vee \hookrightarrow (\mls F^0)^\vee$, and in particular $h^{-1}$ of this complex vanishes. By the isomorphism \eqref{eq:fds} we have that $h^{-1}(\EE)$ vanishes.

Moreover, to show $R^1\pi_*f^*\Omega_X^\vee = 0$, it is enough to show $H^1(C, f^*\Omega^\vee_X)=0$ for each stable map $f: C \to X$ over an algebraically closed field.
This is because the vector spaces $H^1(C, f^*\Omega^\vee_X)$ are the fibers of $R^1\pi_*f^*\Omega_X^\vee$ (by Cohomology and Base Change, \cite[Thm 25.1.6]{vakil}) and we can apply Nakayama's lemma.

We now check that $H^1(C, f^*\Omega_X^\vee) = 0$ for each stable map $f: C \to X$ in each of the situations given.

For $\overline{\mc M}_{0, n}(X, 0)$, every stable map $f: C \to X$ is constant. Hence $f^*T_X = \mls O_{C}$ and $H^1(C, f^*T_X)=0$ since $C$ has genus zero.

For $\overline{\mc M}_{g, n}(\Spec(k), d)$ we have that $T_X$ and hence  $f^*T_X$ are the zero locally free sheaf for any stable map $f: C \to \Spec(k)$. Clearly $H^1(C, 0)=0$.

For $\overline{\mc M}_{0, n}(X, \beta)$ for $X$ convex, we have $H^1(C, f^*\Omega_X^\vee) = 0$ for every stable map by definition of convexity.
\end{itemize}

\item[\ref{GW3}]

\begin{itemize}
\item[(a)] The dimension of the right hand side is
\[
\dim \overline{\mc M}_{g, n} + \dim X - g \dim X = (3g-3+n) + \dim X(1-g).
\]
This is equal to the dimension of the left hand side by Lemma \ref{lem:maps6}.

\item[(b)]
The product family on $\overline{\mc M}_{g, n} \times X$ defines a morphism
\begin{equation}\label{eq:triv-family}
\overline{\mc M}_{g, n} \times X \to \overline{\mc M}_{g, n}(X, 0).
\end{equation}
We claim that for any degree-zero stable map $f: C \to X$ over a base scheme $S$, there is a unique morphism $S \to X$ such that $C \to S \to X$ recovers $f$. Granting this claim, it is easy to check that \eqref{eq:triv-family} is essentially surjective and fully faithful.

    We now prove the claim.
    Since $X$ is projective,
    we may embed it in some $\PP^{n}$.
    Thus,
    it suffices to prove the claim when $X = \PP^{n}$. Furthermore, by descent it suffices to prove the claim when $S$ is affine, and since $C \to S$ is of finite presentation we may assume $S$ is Noetherian as well. (We will need $S$ to be Noetherian later, when we apply the Cohomology and Base Change theorem.)

    The data of $f$ is then equivalent to specifying a line bundle
    $L$ on $C$
    along with $n + 1$ sections $s_{0}, \dots, s_{n}$
    that do not vanish simultaneously.
    The condition on curve class says that $L$ must be fibrewise trivial.
    We claim that in this case,
    there must exist a line bundle $M$ on $S$
    such that $\pi^{\ast}(M) \cong L$.

    Let us first prove that the pushforward
    $\pi_{\ast}L$ is locally free.
    Indeed,
    the Cohomology and Base Change theorem
    says that it suffices to verify that
    for every $s \in S$,
    the canonical morphism
        $(\pi_{\ast}L) \otimes \kappa(s) \to H^{0}(C_{s }, L_{s })$
        is surjective,
        where $\kappa(s)$ is the residue field of $s$.
        Since $C_{s }$ is fibrewise trivial,
        we know that $H^{0}(C_{s }, L_{s })$ is one-dimensional.
        Surjectivity then amounts to verifying that there exists some section of
        $L$ whose image in $H^{0}(C_{s }, L_{s })$ is nonzero.
        This holds because $L$ is globally generated.

    Now,
    let $M = \pi_{\ast}L$.
    By our previous computation,
    $M$ is locally free of rank $1$.
    We claim that $\pi^{\ast}M \cong L$,
     and in fact this isomorphism is witnessed by the counit of adjunction
    $\pi^{\ast}\pi_{\ast}L \to L$.
    We first show that this morphism is surjective.
    By Nakayama's lemma,
    this amounts to verifying that for all
    $s \in S$ and $y \in C_{s }$ with $s = \pi(y)$,
    the natural map $(\pi_{\ast}L)_{s} \to L_{y }$
    is surjective.
    This is again a direct consequence of the fact that $L$
    is globally generated. (Alternatively, $\pi_*L$ is a line bundle whose formation commutes with arbitrary base change, so $(\pi_* L)_s = H^0(C_s, L_s) = k$ surjects onto $L_y $.)

    To show that $\pi^{\ast}M \to L$ is injective,
    note that the kernel of $\pi^{\ast}\pi_{\ast}L \to L$
    must be locally free of rank $0$,
    and thus is zero.
    Thus,
    $\pi^{\ast}\pi_{\ast}L \cong L$,
    and the sections $s_{0}, \dots, s_{n}$
    are all pulled back from $S$.
    Hence,
    the morphism
    $C \to \PP^{n}$
    must factor through $S$,
    as desired.

    For uniqueness, we use that surjective flat morphisms are epimorphisms \cite[Tag 02VW]{stacks-project}.

\item[(c)] By Lemma \ref{lem:BF8} we have that
\[
[\overline{\mc M}_{g, n}(X, 0)]^{vir}_\phi = c_{top}(E) \cap [\overline{\mc M}_{g, n}(X, 0)],
\]
where $E = \relSpec_{\overline{\mc M}_{g, n}(X, 0)}(Sym(h^{-1}(\EE)))$. By part (b), the universal family on $\overline{\mc M}_{g, n}(X, 0)$ has the form
\[
\begin{tikzcd}
\mls C \times X \arrow[d, "\pi"] \arrow[r, "f"] & X\\
\overline{\mc M}_{g, n} \times X
\end{tikzcd}
\]
where $\mls C \to \overline{\mc M}_{g, n}$ is the universal curve and $f = pr_2 \circ \pi$.
It follows that
\[
h^{-1}(\EE) = R^0\pi_*(\pi^*pr_2^*\Omega_X \otimes (\omega_{\mls C} \boxtimes \mls O_X))
\]
and by the projection formula \cite[Tag 01E8]{stacks-project}, since $\Omega_X$ is finite rank locally free, we have
\[
h^{-1}(\EE) = R^0\pi_*\omega_{\mls C} \boxtimes \Omega_X.
\]
Since the rank of $\pi_*\omega_{\mls C}$ is $g$ and the rank of $\Omega_X$ is $\dim X$, applying $\relSpec_{\overline{\mc M}_{g, n}(X, 0)}(Sym(-))$ to this locally free sheaf yields a vector bundle whose sheaf of sections is dual to $h^{-1}(\EE)$.

\end{itemize}

\item[\ref{GW4}]

\begin{itemize}

\item[(a)] By Lemma \ref{lem:maps6} the dimension of the right hand side is
\[
(1)(r-3) - K_{\PP^r} \cdot d[C] + n - (\ell d + 1)
\]
where $[C]$ is the class of a line in $H_2(\PP^r)$.
Let $[L] \in H_{2}(X)$ be the unique class whose pushforward to $\PP^{r}$
is $[C]$.
The same lemma implies that the dimension of the left hand side is
\[
(1)(r-1-3) - K_X \cdot d[L] + n.
\]
The adjunction formula states $K_X = (K_{\PP^r} + [X])|_X$.
Then the projection formula implies
$K_X \cdot d[L] = K_{\PP^r} \cdot d[C]+ [X] \cdot d[C]$
(since $[L]$ is by definition the class in $H_2(X)$ that pushes forward to $[C] \in H_2(\PP^r)$).
Since $X$ is a hypersurface of degree $\ell$, we have $[X] \cdot d[C] = \ell d$ and the two expressions are equal.

\item[(b)]
We claim first that $R^1\pi_*f^*\mls O_{\PP^r}(\ell) = 0$. Indeed, as explained in the solution to Exercise \ref{GW1}(b), to prove this it is enough to show $H^1(C, f^*\mls O_{\PP^r}(\ell))=0$ for every degree-$d$ stable map $f: C \to \PP^r$ over an algebraically closed field.
This is clear when $C$ is smooth; when $C$ is nodal we use the following lemma.

\begin{lemma}\label{lem:ql2}
    Let $C$ be a tree of $\PP^1$'s, and let $L$ be a line bundle on $C$ such that the restriction of $L$ to each component of $C$ has nonnegative degree. Then $h^1(C, L) = 0$.
\end{lemma}
\begin{proof}
By Serre duality it is enough to show
\begin{equation}\label{eq:h1-vanishing}
\quad H^{0}(C^{\prime}, (\omega_{C} \otimes L^{\vee})|_{C^{\prime}}) = 0
\end{equation}
    for every irreducible component $C'$ of $C$.  Note that by \cite[Tag~0E34]{stacks-project}, the degree of $\omega_C|_{C'}$ is $n_{C'}-2$, where $n_{C'}$ is the number of nodes of $C$ contained in $C'$. Hence $\deg((\omega_C \otimes L^\vee)|_{C'}) \leq n_{C'}-2$.

 We prove \eqref{eq:h1-vanishing} by rooting the dual graph of $C$ (which is a tree) at a fixed vertex, and proving that \eqref{eq:h1-vanishing} holds on every vertex (i.e. component of $C$) by inducting on the height of the minimal subtree containing that vertex.    Note that to show \eqref{eq:h1-vanishing} holds on a component $C'$, because of the bound $\deg((\omega_C \otimes L^\vee)|_{C'}) \leq n_{C'}-2$, it is enough to show that any global section $s \in H^0(C, (\omega_C \otimes L^\vee)|_{C'})$ has to vanish along at least $n_{C'}-1$ points of $C'$. This vanishing holds when $C'$ has a unique node, i.e. the corresponding vertex is a leaf, i.e. $n_{C'}=1$. This handles the base case. For the inductive step,
    take any vertex $v$ and let $C^{\prime}$ be the corresponding component.
    By induction hypothesis,
    we may assume that the restriction of any global section
    $s \in H^{0}(C, (\omega_C \otimes L^\vee)|_{C'})$ to any child vertex of $v$ must vanish.
    Since $v$ has exactly $n_{C^{\prime}}-1$ child vertices,
    this shows that $s|_{C^{\prime}}$ must vanish along at least
    $n_{C^{\prime}}-1$ points, as desired.
    \end{proof}

By \cite[Tag 0D4E]{stacks-project} it follows that $\pi_*f^*\mls O_{\PP^r}(\ell)$ is locally free. Its rank is computed by the rank of any fiber, namely by $H^0(C, f^*\mls O_{\PP^r}(\ell))$. This is easily seen to be $d\ell+1$ when $C$ is smooth. When $C$ is nodal we apply the following lemma.

\begin{lemma}
    \label{lem: RR for tree of PP1}
    Let $C$ be a tree of $\PP^{1}$'s,
    and let $L$ be a line bundle on $C$
    such that the restriction of $L$ to each component of $C$
    has positive degree.
    Then $h^{0}(C, L) = 1 + \deg(L).$
\end{lemma}

\begin{proof}
    We induct on the number of irreducible components of $C$.
In the base case $C$ is $\PP^1$ and the claim is familiar.

    For the inductive step, write $C = C^{\prime} \cup C^{\prime\prime}$
    as the union of two subcurves intersecting at a node $q \in C$.
    We may assume,
    by the induction hypothesis,
    that our desired statement holds for both
    $C^{\prime}$ and $C^{\prime\prime}.$

    Consider the short exact sequence
    \[
        0 \to \OO_{C} \to \OO_{C^{\prime}} \oplus \OO_{C^{\prime\prime}}
        \to \OO_{q} \to 0.
    \]
    Tensoring with $L$ gives the short exact sequence
    \[
        0 \to L \to L|_{C^{\prime}} \oplus L_{C^{\prime\prime}} \to \OO_{q} \to 0.
    \]
    Lemma \ref{lem:ql2} shows that
    $H^{1}(C, L)$ vanishes,
    so
    \begin{align*}
        h^{0}(C, L) &= h^{0}(C^{\prime}, L|_{C^{\prime}}) +
            h^{0}(C^{\prime\prime}, L|_{C^{\prime\prime}}) - 1 \\
                    &= 1 + \deg(L|_{C^{\prime}}) + \deg(L|_{C^{\prime\prime}}) \\
                    &= 1 + \deg(L).
    \end{align*}
\end{proof}

\item[(c)] Global sections of $\pi_*f^*\mls O_{\PP^r}(\ell)$ can be identified with global sections of $f^*\mls O_{\PP^r}(\ell)$. We take $\sigma$ to be the global section of $\pi_*f^*\mls O_{\PP^r}(\ell)$ that maps to $f^*s$ under this identification. The section $\sigma$ vanishes on the $T$-point of $\overline{\mc M}_{0, n}(\PP^r, d)$ given by $f_T: C_T \to \PP^r$ if $f_T^*s=0$, or equivalently if $f_T$ factors through $X$. Hence $\VV(\sigma) = \overline{\mc M}_{0, n}(X, d)$.

\end{itemize}

\end{itemize}

\noindent
\textbf{Section 9}

\begin{itemize}[leftmargin=.3in]

\item[\ref{siebert1}]

\begin{itemize}
\item[(a)]
We first construct a fiber diagram
\begin{equation}\label{eq:s1}
\begin{tikzcd}
Z \arrow[r] \arrow[d] & X \arrow[r] \arrow[d] & \fX \arrow[d, "\delta"] \\
X \times_S Y \arrow[r] & {[(X \times_S Y)/H]} \arrow[r] \arrow[d] & \fX \times_S \fX \arrow[d, "pr_1"]\\
& X \arrow[r] & \fX
\end{tikzcd}
\end{equation}
where in the middle, $H$ acts on $X$ trivially and on $Y$ by the given action. The fiber product $\fX \times_S \fX$ is a global quotient $[(X \times_S Y)/(G \times_S H)]$ and the horizontal arrows in the right column are all quotients by $H$. It is clear that the bottom right square is fibered, as is the composition of the two squares in the right column, hence all squares in the right column are fibered. We define $Z$ to be the fiber product of the top left square, so the entire diagram is fibered. In particular the top left square implies $X = [Z/H]$.

Note that $Z$ is the fiber product $(X \times_S Y) \times_{\fX \times_S \fX} \fX.$

There is an analogous fiber diagram with the roles of $X$ and $Y$ reversed:
\[
\begin{tikzcd}
Z \arrow[r] \arrow[d] & Y \arrow[r] \arrow[d] & \fX \arrow[d, "\delta"] \\
X \times_S Y \arrow[r] & {[(X \times_S Y)/G]} \arrow[r] \arrow[d] & \fX \times_S \fX \arrow[d, "pr_1"]\\
& Y \arrow[r] & \fX.
\end{tikzcd}
\]
The $Z$ in this diagram is the same as the $Z$ in the former, since both are given by $(X \times_S Y) \times_{\fX \times_S \fX} \fX$. However the second diagram implies that $Y = [Z/G]$.

\item[(b)]

From the fiber diagram \eqref{eq:s1} we extract the fiber square
\begin{equation}\label{eq:siebert101}
\begin{tikzcd}
Z \arrow[r] \arrow[d] & X \times_S Y \arrow[d] \\
X \arrow[r] & {[(X \times_S Y)/H]}.
\end{tikzcd}\end{equation}
Since $X$ and $Y$ are cones, by Exercise \ref{siebert-1} the fiber product $X \times_S Y$ is also a cone and $[(X \times_S Y)/H]$ is (a global presentation for) a cone stack. By Exercise \ref{siebert0}(a) we have a presentation $X \simeq [Z/H]$. An analogous argument constructs a presentation $Y = [Z/G]$.

\item[(c)]
Let $\mls H_\bullet$, $\mls Z_\bullet$, and $\mls X_\bullet$ be graded quasicoherent sheaves on $S$ such that $H = \relSpec_S(\mls H_\bullet)$ and so forth. So we have homomorphisms $\mls X_\bullet \to \mls Z_\bullet \to \mls H_\bullet$ and $\mls H_\bullet = Sym(\mls H_1)$ with $\mls H_1$ finite rank locally free. To show that $0 \to H \to Z \to X \to 0$ is exact we must check three things.

First, we must show that $\mls X_\bullet \to \mls Z_\bullet$ is injective. Since $\pi: Z \to X$ is an $H$-torsor, in particular it is faithfully flat, so $\pi^*$ is fully faithful. Hence the unit $\mls O_X \to \pi_*\mls O_Z$ for the $(\pi^*, \pi_*)$ adjunction is a monomorphism; i.e., an injective morphism of quasicoherent sheaves. The pushforward of this morphism to $S$ is still injective (since pushforward is left exact) and this recovers the morphism $\mls X_\bullet \to \mls Z_\bullet$.

Second, we must show that $\mls Z_\bullet \to \mls H_\bullet$ is surjective. From \eqref{eq:siebert101} we can deduce the existence of a fiber diagram
\[
\begin{tikzcd}
H \arrow[r] \arrow[d] & Z \arrow[d, "\pi"]\\
S \arrow[r, "0"] & X
\end{tikzcd}
\]
where crucially $0: S \to X$ is a closed embedding (the zero section, corresponding to the surjection $\mls X_\bullet \to \mls X_0$). Hence $H \to Z$ is a closed embedding and $\mls Z_\bullet \to \mls H_\bullet$ is surjective.

Finally, we must show that locally on $S$ the morphism $\mls Z_1 \to \mls H_1$ has a section (necessarily injective) and that this section (locally) realizes $Z$ as the fiber product $H \times_S X$. The required section exists locally because $\mls H_1$ is locally free, hence locally on $S$ corresponds to a projective module. This section induces a morphism
\[
\mls H_\bullet = Sym(\mls H_1) \to Sym(\mls Z_1)
\]
that is a section of the composition $Sym(\mls Z_1) \to \mls Z_\bullet \to \mls H_\bullet$. Therefore the composition of the arrows
\[
H \to Z \to A(Z) \to H
\]
is the identity,
where $A(Z)  = \relSpec_S(Sym(\mls Z_1))$ is the abelian hull of $Z$. Recall that a priori the action of $H$ is defined on $A(Z)$, and the action of $H$ on $Z$ is defined to be the restriction of this a-priori action. Note also that $A(Z) \to H$ is a homomorphism of abelian group schemes as it arises from a homomorphism of coherent sheaves. Commutativity of the diagram then shows that $A(Z) \to H$ is equivariant for the action of $H$, and hence $Z \to H$ is also equivariant. Then
\[
\begin{tikzcd}
Z \arrow[r] \arrow[d, "\pi"] & H \arrow[d] \\
X \arrow[r] & S
\end{tikzcd}
\]
is a commuting square whose vertical arrows are $H$-torsors, hence fibered.

\item[(d)] We have two equalities $s(X) = c(H) \cap s(Z)$ and $s(Y) = c(G) \cap s(Z)$. Then
\[
c(G) \cap s(X) = c(G) \cap (c(H)\cap s(Z)) = c(H) \cap (c(G) \cap s(Z)) = c(H) \cap s(Y).
\]
\end{itemize}

\item[\ref{siebert4}]

    By Section \ref{sec:globalres} we have
\[
[X]^{vir}_\phi = 0^!_{E^1}[C(E)]
\]
i.e., the virtual class is the intersection of the cone $C(E)$ (representable over $X$) with the zero section of $E^1$, a vector bundle on $X$. By \cite[Example 4.1.8]{fulton}, this is equal to the product
\[
\{ \; c(E^1) \cap s(C(E)) \;\}_{d - \rank(E^1)}
\]
where $d$ is the dimension of $C(E)$. Since $\mathfrak{C}_{X/B} = [C(E)/E^0]$ has pure dimension $\dim B$ (Exercise \ref{POT0}), we see that $\dim C(E) = \rank(E^0) + \dim B$.  So $d-\rank(E^1) = \dim B + \rank(\EE_{X/B})$, and the above is equal to
\[
\{ \; (c(E^1) \cup c(E^0)^{-1}) \cap (c(E^0) \cap s(C(E)) )\;\}_{\dim B + \rank(\EE_{X/B})}
\]
But $c(E^0) \cap s(C(E))$ is Fulton's canonical class and we recover the desired formula.

\item[\ref{siebert101}]
\begin{itemize}
\item[(a)]

        For ease of notation,
        we write $\mc{M}^{\prime} = \overline{\mc{M}}_{0, n}(X, \beta)$,
        $\mc{M} = \overline{\mc{M}}_{0, n}(\PP^{r}, d)$,
        and $\mf{M} = \mf{M}_{0, n}$.
        Since $j \colon \mc{M}^{\prime} \to \mc{M}$ is a closed embedding
        and $\mc{M} \to \mf{M}$ is smooth,
        we have $\mf{C}_{\mc{M}^{\prime}/\mf{M}} =
        [C_{\mc{M}^{\prime}/\mc{M}}/j^{\ast}T_{\mc{M}/\mf{M}}]$.
        Thus,
        the Segre class of $\mf{C}_{\mc{M}^{\prime}/\mf{M}}$
        is
        \[
            s(\mf{C}_{\mc{M}^{\prime}/\mf{M}}) =
            c(j^{\ast}T_{\mc{M}/\mf{M}}) \cap s(C_{\mc{M}^{\prime}/\mc{M}}),
        \]
        and by Siebert's formula (Theorem \ref{thm:siebert}),
        we have
        \begin{equation}
            \label{eq:ex9.5 virt class}
            [\mc{M}^{\prime}]^{vir}_{\phi} =
            \left\{
                c(\EE_{\mc{M}^{\prime}/\mf{M}}^\vee)^{-1}c(j^{\ast}T_{\mc{M}/\mf{M}})
                \cap s(C_{\mc{M}^{\prime}/\mc{M}})
            \right\}_{\dim(\mf M) + \rank(\EE_{\mc M'/\mf M})}.
        \end{equation}

        We  proceed to simplify this expression. From the cotangent exact sequence
        \[
            0 \to \OO_{X}(-\ell) \to \Omega_{\PP^{r}}|_{X} \to \Omega_{X} \to 0
        \]
        we obtain the distinguished triangle
        \begin{equation}
            \label{eq:ex9.5 dist triangle}
            R\pi_{\ast}(f^{\ast}\OO_{X}(-\ell) \otimes \omega_{\pi}[1]) \to
            R\pi_{\ast}(f^{\ast}\Omega_{\PP^{r}}|_{X} \otimes \omega_{\pi}[1]) \to
            R\pi_{\ast}(f^{\ast}\Omega_{X} \otimes \omega_{\pi}[1]).
            \xrightarrow{+1}
        \end{equation}
        By Grothendieck duality,
        the first term is
        $(R\pi_{\ast}(f^{\ast}\OO_{X}(\ell)))^{\vee}$.
        From Exercise \ref{GW4} we know that
        $R^{1}\pi_{\ast}(f^{\ast}\OO_{X}(\ell))$
        vanishes,
        so the first term is precisely
        $\mls{E}^{\vee}[0]$.
        By compatibility of the dualising complex with base change
        and the derived base change formula,
        we see that the middle term is
        $j^{\ast}\EE_{\mc{M}/\mf{M}} =
        j^{\ast}\LL_{\mc{M}/\mf{M}} = j^{\ast}(\Omega_{\mc{M}/\mf{M}})[0]$.
The last term is $\EE_{\mc M'/\mf M}.$

        Hence,
        since the Chern class is multiplicative across
        distinguished triangles,
        dualising \eqref{eq:ex9.5 dist triangle} gives
        \[
            c(\EE_{\mc{M}^{\prime}/\mf{M}}^{\vee}) =
            c(E)^{-1}c(j^{\ast}T_{\mc{M}/\mf{M}}).
        \]
        Substituting this into
        \eqref{eq:ex9.5 virt class} gives
        \[
            [\mc{M}^{\prime}]^{vir}_{\phi} =
            \left\{ c(E) \cap s(C_{\mc{M}^{\prime}/\mc{M}})\right\}_{\dim(\mf M) + \rank(\EE_{\mc M'/\mf M})}.
        \]

        On the other hand, by Exercise \ref{vc6} we have
        \[
        [\mc{M}^{\prime}]^{vir}_{\mls{E}} = \left\{ c(E) \cap s(C_{\mc{M}^{\prime}/\mc{M}})\right\}_{\dim(\mc M) - \rank(E)}.
        \]
        We are thus done,
        since $\dim(\mc M) = \dim(\mf M) + \rank(\EE_{\mc M/\mf M})$ and from the distinguished triangle \eqref{eq:ex9.5 dist triangle} we have
        \[
        \rank(\EE_{\mc M'/\mf M}) = \rank(\EE_{\mc M/\mf M}) - \rank(\mls E^\vee).
        \]

\item[(b)]Lemma \ref{lem:maps66} follows immediately from (a) and the generalization of \eqref{lem:vc2} to Deligne-Mumford stacks.

\end{itemize}
\end{itemize}

\newpage

\normalem
\printbibliography

@article {KKP,
    AUTHOR = {Kim, Bumsig and Kresch, Andrew and Pantev, Tony},
     TITLE = {Functoriality in intersection theory and a conjecture of
              {C}ox, {K}atz, and {L}ee},
   JOURNAL = {J. Pure Appl. Algebra},
  FJOURNAL = {Journal of Pure and Applied Algebra},
    VOLUME = {179},
      YEAR = {2003},
    NUMBER = {1-2},
     PAGES = {127--136},
      ISSN = {0022-4049,1873-1376},
   MRCLASS = {14N35 (14C17 14J10)},
  MRNUMBER = {1958379},
MRREVIEWER = {Kevin\ Joseph\ Costello},
       DOI = {10.1016/S0022-4049(02)00293-1},
       URL = {https://doi.org/10.1016/S0022-4049(02)00293-1},
}

@book {weibel,
    AUTHOR = {Weibel, Charles A.},
     TITLE = {An introduction to homological algebra},
    SERIES = {Cambridge Studies in Advanced Mathematics},
    VOLUME = {38},
 PUBLISHER = {Cambridge University Press, Cambridge},
      YEAR = {1994},
     PAGES = {xiv+450},
   MRCLASS = {18-01 (16-01 17-01 20-01 55Uxx)},
  MRNUMBER = {1269324},
MRREVIEWER = {Kenneth\ A.\ Brown},
       DOI = {10.1017/CBO9781139644136},
       URL = {https://doi.org/10.1017/CBO9781139644136},
}

@article {kresch,
    AUTHOR = {Kresch, Andrew},
     TITLE = {Cycle groups for {A}rtin stacks},
   JOURNAL = {Invent. Math.},
  FJOURNAL = {Inventiones Mathematicae},
    VOLUME = {138},
      YEAR = {1999},
    NUMBER = {3},
     PAGES = {495--536},
      ISSN = {0020-9910,1432-1297},
   MRCLASS = {14A20 (14C17)},
  MRNUMBER = {1719823},
MRREVIEWER = {Burt\ Totaro},
       DOI = {10.1007/s002220050351},
       URL = {https://doi.org/10.1007/s002220050351},
}

@book{fulton,
    AUTHOR = {Fulton, William},
     TITLE = {Intersection theory},
    SERIES = {Ergebnisse der Mathematik und ihrer Grenzgebiete. 3. Folge. A
              Series of Modern Surveys in Mathematics [Results in
              Mathematics and Related Areas. 3rd Series. A Series of Modern
              Surveys in Mathematics]},
    VOLUME = {2},
   EDITION = {2nd ed},
 PUBLISHER = {Springer-Verlag, Berlin},
      YEAR = {1998},
     PAGES = {xiv+470},
      ISBN = {0-387-98549-2},
   MRCLASS = {14C17 (14-02)},
  MRNUMBER = {1644323},
       DOI = {10.1007/978-1-4612-1700-8},
       URL = {https://doi.org/10.1007/978-1-4612-1700-8},
}

@article {webb,
    AUTHOR = {Webb, Rachel},
     TITLE = {The moduli of sections has a canonical obstruction theory},
   JOURNAL = {Forum Math. Sigma},
  FJOURNAL = {Forum of Mathematics. Sigma},
    VOLUME = {10},
      YEAR = {2022},
     PAGES = {Paper No. e78, 47},
      ISSN = {2050-5094},
   MRCLASS = {14D23 (18D99)},
  MRNUMBER = {4479830},
MRREVIEWER = {Ryo\ Ohkawa},
       DOI = {10.1017/fms.2022.61},
       URL = {https://doi.org/10.1017/fms.2022.61},
}

@inproceedings {pand,
    AUTHOR = {Pandharipande, Rahul},
     TITLE = {Cohomological field theory calculations},
 BOOKTITLE = {Proceedings of the {I}nternational {C}ongress of
              {M}athematicians---{R}io de {J}aneiro 2018. {V}ol. {I}.
              {P}lenary lectures},
     PAGES = {869--898},
 PUBLISHER = {World Sci. Publ., Hackensack, NJ},
      YEAR = {2018},
      ISBN = {978-981-3272-87-3},
   MRCLASS = {14H10 (14H60 14H81 14N35)},
  MRNUMBER = {3966747},
MRREVIEWER = {Hsian-Hua\ Tseng},
}

@incollection {fantechi,
    AUTHOR = {Fantechi, Barbara},
     TITLE = {Stacks for everybody},
 BOOKTITLE = {European {C}ongress of {M}athematics, {V}ol. {I} ({B}arcelona,
              2000)},
    SERIES = {Progr. Math.},
    VOLUME = {201},
     PAGES = {349--359},
 PUBLISHER = {Birkh\"auser, Basel},
      YEAR = {2001},
      ISBN = {3-7643-6417-3},
   MRCLASS = {14A20},
  MRNUMBER = {1905329},
MRREVIEWER = {Gabriele\ Vezzosi},
}

@misc{stacks-project,
    author       = {The {Stacks Project Authors}},
    title        = {Stacks Project},
    howpublished = {\url{https://stacks.math.columbia.edu}},
    year         = {2026},
  }

@article {Olsson07,
    AUTHOR = {Olsson, Martin},
     TITLE = {Sheaves on {A}rtin stacks},
   JOURNAL = {J. Reine Angew. Math.},
  FJOURNAL = {Journal f\"ur die Reine und Angewandte Mathematik. [Crelle's
              Journal]},
    VOLUME = {603},
      YEAR = {2007},
     PAGES = {55--112},
      ISSN = {0075-4102,1435-5345},
   MRCLASS = {14A20 (14D20)},
  MRNUMBER = {2312554},
MRREVIEWER = {Charles\ D.\ Cadman},
       DOI = {10.1515/CRELLE.2007.012},
       URL = {https://doi.org/10.1515/CRELLE.2007.012},
}

@book {fgaexplained,
    AUTHOR = {Fantechi, Barbara and G\"ottsche, Lothar and Illusie, Luc and
              Kleiman, Steven L. and Nitsure, Nitin and Vistoli, Angelo},
     TITLE = {Fundamental algebraic geometry},
    SERIES = {Mathematical Surveys and Monographs},
    VOLUME = {123},
      NOTE = {Grothendieck's FGA explained},
 PUBLISHER = {American Mathematical Society, Providence, RI},
      YEAR = {2005},
     PAGES = {x+339},
      ISBN = {0-8218-3541-6},
   MRCLASS = {14-06 (14A15 14D15 14F20)},
  MRNUMBER = {2222646},
MRREVIEWER = {Liam\ O'Carroll},
       DOI = {10.1090/surv/123},
       URL = {https://doi.org/10.1090/surv/123},
}

@incollection {siebert,
    AUTHOR = {Siebert, Bernd},
     TITLE = {Virtual fundamental classes, global normal cones and
              {F}ulton's canonical classes},
 BOOKTITLE = {Frobenius manifolds},
    SERIES = {Aspects Math.},
    VOLUME = {E36},
     PAGES = {341--358},
 PUBLISHER = {Friedr. Vieweg, Wiesbaden},
      YEAR = {2004},
      ISBN = {3-528-03206-5},
   MRCLASS = {14N35 (14C15)},
  MRNUMBER = {2115776},
MRREVIEWER = {Charles\ D.\ Cadman},
}

@article {BF,
    AUTHOR = {Behrend, K. and Fantechi, B.},
     TITLE = {The intrinsic normal cone},
   JOURNAL = {Invent. Math.},
  FJOURNAL = {Inventiones Mathematicae},
    VOLUME = {128},
      YEAR = {1997},
    NUMBER = {1},
     PAGES = {45--88},
      ISSN = {0020-9910,1432-1297},
   MRCLASS = {14F99 (14C15 14D20)},
  MRNUMBER = {1437495},
MRREVIEWER = {Tohru\ Nakashima},
       DOI = {10.1007/s002220050136},
       URL = {https://doi.org/10.1007/s002220050136},
}

@article {CJW,
    AUTHOR = {Chen, Qile and Janda, Felix and Webb, Rachel},
     TITLE = {Virtual cycles of stable (quasi-)maps with fields},
   JOURNAL = {Adv. Math.},
  FJOURNAL = {Advances in Mathematics},
    VOLUME = {385},
      YEAR = {2021},
     PAGES = {Paper No. 107781},
      ISSN = {0001-8708,1090-2082},
   MRCLASS = {14N35 (14D23)},
  MRNUMBER = {4261166},
MRREVIEWER = {Nathan\ Priddis},
       DOI = {10.1016/j.aim.2021.107781},
       URL = {https://doi.org/10.1016/j.aim.2021.107781},
}

@misc{vakil,
  title        = {The Rising Sea: Foundations of Algebraic Geometry},
  author       = {Vakil, Ravi},
  howpublished = "\url{https://math.stanford.edu/~vakil/216blog/FOAGoct2125public.pdf}",
  year         = 2025,
}

@article {AGOT,
    AUTHOR = {Abramovich, Dan and Graber, Tom and Olsson, Martin and Tseng,
              Hsian-Hua},
     TITLE = {On the global quotient structure of the space of twisted
              stable maps to a quotient stack},
   JOURNAL = {J. Algebraic Geom.},
  FJOURNAL = {Journal of Algebraic Geometry},
    VOLUME = {16},
      YEAR = {2007},
    NUMBER = {4},
     PAGES = {731--751},
      ISSN = {1056-3911,1534-7486},
   MRCLASS = {14D22 (14A20 14C05 14C40)},
  MRNUMBER = {2357688},
MRREVIEWER = {Charles\ D.\ Cadman},
       DOI = {10.1090/S1056-3911-07-00443-2},
       URL = {https://doi.org/10.1090/S1056-3911-07-00443-2},
}

@article {BCM,
    AUTHOR = {Battistella, Luca and Carocci, Francesca and Manolache,
              Cristina},
     TITLE = {Virtual classes for the working mathematician},
   JOURNAL = {SIGMA Symmetry Integrability Geom. Methods Appl.},
  FJOURNAL = {SIGMA. Symmetry, Integrability and Geometry. Methods and
              Applications},
    VOLUME = {16},
      YEAR = {2020},
     PAGES = {Paper No. 026},
      ISSN = {1815-0659},
   MRCLASS = {14C17 (14D23 14N35)},
  MRNUMBER = {4082488},
MRREVIEWER = {Nicola\ Pagani},
       DOI = {10.3842/SIGMA.2020.026},
       URL = {https://doi.org/10.3842/SIGMA.2020.026},
}

@article {vistoli,
    AUTHOR = {Vistoli, Angelo},
     TITLE = {Intersection theory on algebraic stacks and on their moduli
              spaces},
   JOURNAL = {Invent. Math.},
  FJOURNAL = {Inventiones Mathematicae},
    VOLUME = {97},
      YEAR = {1989},
    NUMBER = {3},
     PAGES = {613--670},
      ISSN = {0020-9910,1432-1297},
   MRCLASS = {14C17 (14A20 14D20)},
  MRNUMBER = {1005008},
MRREVIEWER = {Steven\ E.\ Landsburg},
       DOI = {10.1007/BF01388892},
       URL = {https://doi.org/10.1007/BF01388892},
}

@incollection {FP,
    AUTHOR = {Fulton, W. and Pandharipande, R.},
     TITLE = {Notes on stable maps and quantum cohomology},
 BOOKTITLE = {Algebraic geometry---{S}anta {C}ruz 1995},
    SERIES = {Proc. Sympos. Pure Math.},
    VOLUME = {62, Part 2},
     PAGES = {45--96},
 PUBLISHER = {Amer. Math. Soc., Providence, RI},
      YEAR = {1997},
      ISBN = {0-8218-0493-6},
   MRCLASS = {14H10 (14E99 14N10)},
  MRNUMBER = {1492534},
MRREVIEWER = {Alexandre\ I.\ Kabanov},
       DOI = {10.1090/pspum/062.2/1492534},
       URL = {https://doi.org/10.1090/pspum/062.2/1492534},
}

@article {GP,
    AUTHOR = {Graber, T. and Pandharipande, R.},
     TITLE = {Localization of virtual classes},
   JOURNAL = {Invent. Math.},
  FJOURNAL = {Inventiones Mathematicae},
    VOLUME = {135},
      YEAR = {1999},
    NUMBER = {2},
     PAGES = {487--518},
      ISSN = {0020-9910,1432-1297},
   MRCLASS = {14C17 (14D20 14N10 14N35)},
  MRNUMBER = {1666787},
MRREVIEWER = {Paolo\ Aluffi},
       DOI = {10.1007/s002220050293},
       URL = {https://doi.org/10.1007/s002220050293},
}

@MISC {example,
    TITLE = {Is there an example of a variety over the complex numbers with no embedding into a smooth variety?},
    AUTHOR = {David Rydh},
    HOWPUBLISHED = {MathOverflow},
    NOTE = {URL:https://mathoverflow.net/q/203 (version: 2009-10-08)},
    EPRINT = {https://mathoverflow.net/q/203},
    URL = {https://mathoverflow.net/q/203}
}

@misc{RydhLuna,
Author = {David Rydh},
Title = {A generalization of Luna's fundamental lemma for stacks with good moduli spaces},
Year = {2020},
Eprint = {arXiv:2008.11118},
}

@article {manolache-pullback,
    AUTHOR = {Manolache, Cristina},
     TITLE = {Virtual pull-backs},
   JOURNAL = {J. Algebraic Geom.},
  FJOURNAL = {Journal of Algebraic Geometry},
    VOLUME = {21},
      YEAR = {2012},
    NUMBER = {2},
     PAGES = {201--245},
      ISSN = {1056-3911,1534-7486},
   MRCLASS = {14C15 (14A20 14N35)},
  MRNUMBER = {2877433},
MRREVIEWER = {Hsian-Hua\ Tseng},
       DOI = {10.1090/S1056-3911-2011-00606-1},
       URL = {https://doi.org/10.1090/S1056-3911-2011-00606-1},
}

@unpublished{webbtalk,
title= {GIT, Stacks, and Quasimap Theory},
author = {Webb, Rachel},
year = {2024},
note= {Physics Latam seminar series},
URL= {https://www.youtube.com/watch?v=_-I4wIo2mQI&list=PLaFLp8EAyd7WugflX8GXIXth5RivmqgZ6&index=3},
}

@article {LT,
    AUTHOR = {Li, Jun and Tian, Gang},
     TITLE = {Virtual moduli cycles and {G}romov-{W}itten invariants of
              algebraic varieties},
   JOURNAL = {J. Amer. Math. Soc.},
  FJOURNAL = {Journal of the American Mathematical Society},
    VOLUME = {11},
      YEAR = {1998},
    NUMBER = {1},
     PAGES = {119--174},
      ISSN = {0894-0347,1088-6834},
   MRCLASS = {14D20 (14D15 14N10)},
  MRNUMBER = {1467172},
MRREVIEWER = {Ralph\ Martin\ Kaufmann},
       DOI = {10.1090/S0894-0347-98-00250-1},
       URL = {https://doi.org/10.1090/S0894-0347-98-00250-1},
}

@article {STV,
    AUTHOR = {Sch\"urg, Timo and To\"en, Bertrand and Vezzosi, Gabriele},
     TITLE = {Derived algebraic geometry, determinants of perfect complexes,
              and applications to obstruction theories for maps and
              complexes},
   JOURNAL = {J. Reine Angew. Math.},
  FJOURNAL = {Journal f\"ur die Reine und Angewandte Mathematik. [Crelle's
              Journal]},
    VOLUME = {702},
      YEAR = {2015},
     PAGES = {1--40},
      ISSN = {0075-4102,1435-5345},
   MRCLASS = {14D20 (14F05)},
  MRNUMBER = {3341464},
MRREVIEWER = {Nicolas\ Perrin},
       DOI = {10.1515/crelle-2013-0037},
       URL = {https://doi.org/10.1515/crelle-2013-0037},
}

@article {HS,
    AUTHOR = {Hirschi, Amanda and Swaminathan, Mohan},
     TITLE = {Global {K}uranishi charts and a product formula in symplectic
              {G}romov-{W}itten theory},
   JOURNAL = {Selecta Math. (N.S.)},
  FJOURNAL = {Selecta Mathematica. New Series},
    VOLUME = {30},
      YEAR = {2024},
    NUMBER = {5},
     PAGES = {Paper No. 87, 74},
      ISSN = {1022-1824,1420-9020},
   MRCLASS = {53D45 (14N35)},
  MRNUMBER = {4807086},
       DOI = {10.1007/s00029-024-00982-y},
       URL = {https://doi.org/10.1007/s00029-024-00982-y},
}

@article {BehrendGW,
    AUTHOR = {Behrend, K.},
     TITLE = {Gromov-{W}itten invariants in algebraic geometry},
   JOURNAL = {Invent. Math.},
  FJOURNAL = {Inventiones Mathematicae},
    VOLUME = {127},
      YEAR = {1997},
    NUMBER = {3},
     PAGES = {601--617},
      ISSN = {0020-9910,1432-1297},
   MRCLASS = {14D20 (14C25 14D22)},
  MRNUMBER = {1431140},
MRREVIEWER = {Barbara\ Fantechi},
       DOI = {10.1007/s002220050132},
       URL = {https://doi.org/10.1007/s002220050132},
}

@article {albano-katz,
    AUTHOR = {Albano, Alberto and Katz, Sheldon},
     TITLE = {Lines on the {F}ermat quintic threefold and the infinitesimal
              generalized {H}odge conjecture},
   JOURNAL = {Trans. Amer. Math. Soc.},
  FJOURNAL = {Transactions of the American Mathematical Society},
    VOLUME = {324},
      YEAR = {1991},
    NUMBER = {1},
     PAGES = {353--368},
      ISSN = {0002-9947,1088-6850},
   MRCLASS = {14J30 (14C30 14K30)},
  MRNUMBER = {1024767},
MRREVIEWER = {Fabio\ Bardelli},
       DOI = {10.2307/2001512},
       URL = {https://doi.org/10.2307/2001512},
}

@ARTICLE{LW2,
       author = {{Lu}, Xuanchun and {Webb}, Rachel},
        title = "{Siebert's formula for virtual pullbacks}",
      journal = {arXiv e-prints},
         year = 2026,
        month = sep,
          eid = {arXiv:2609.12213},
        pages = {arXiv:2609.12213},
archivePrefix = {arXiv},
       eprint = {2609.12213},
 primaryClass = {math.AG},
       adsurl = {https://ui.adsabs.harvard.edu/abs/2026arXiv260912213L}
}
\end{document}